\documentclass[reqno,11pt]{amsart}
\usepackage[left=1in,right=1in,top=1in,bottom=1in]{geometry}
\usepackage{enumitem}
\usepackage{graphicx}
\usepackage{tikz}
\usetikzlibrary{arrows, decorations.markings,matrix,trees}
\usepackage{hyperref}
\hypersetup{
   colorlinks=true,
   citecolor=blue,
   filecolor=blue,
   linkcolor=blue,
   urlcolor=blue
}

\def\XXint#1#2#3{{\setbox0=\hbox{$#1{#2#3}{\int}$}
  \vcenter{\hbox{$#2#3$}}\kern-.5\wd0}}

\newcommand{\al}{\alpha}       
\newcommand{\vt}{\vartheta}    
\newcommand{\lda}{\lambda}
\newcommand{\om}{\Omega}            
\newcommand{\pa}{\partial}
\newcommand{\va}{\varepsilon}       
\newcommand{\ud}{\mathrm{d}}
\newcommand{\be}{\begin{equation}} 
\newcommand{\ee}{\end{equation}}
      
\newcommand{\Lda}{\Lambda}

\newcommand{\cA}{\mathcal{A}}

\newcommand{\cB}{\mathcal{B}}

\newcommand{\cC}{\mathcal{C}}
\newcommand{\C}{\mathbb{C}}

\newcommand{\cD}{\mathcal{D}}

\newcommand{\cE}{\mathcal{E}}

\newcommand{\cF}{\mathcal{F}}

\newcommand{\bG}{\mathbb{G}}

\newcommand{\cL}{\mathcal{L}} 

\newcommand{\Z}{\mathbb{Z}}

\newcommand{\M}{\mathcal{M}}
\newcommand{\MM}{\mathbb{M}}

\newcommand{\cN}{\mathcal{N}}
\newcommand{\N}{\mathbb{N}}

\newcommand{\cP}{\mathcal{P}}

\newcommand{\R}{\mathbb{R}}   

\newcommand{\cR}{\mathcal{R}}

\newcommand{\Ss}{\mathbb{S}}

\newcommand{\wc}{\rightharpoonup}        

\newcommand{\HH}{\mathcal{H}}

\newcommand{\vp}{\varphi}

\newcommand{\T}{\mathrm{T}}
\newcommand{\ga}{\gamma}
\newcommand{\Ga}{\Gamma}

\newcommand{\sg}{\sigma} 
\newcommand{\ift}{\infty} 
\newcommand{\wt}{\widetilde}
\newcommand{\wh}{\widehat}
\newcommand{\f}{\frac}

\newcommand{\ol}{\overline}

\newcommand{\op}{\operatorname}

\newcommand{\na}{\nabla}

\DeclareMathOperator{\dist}{dist}

\DeclareMathOperator{\diam}{diam}

\DeclareMathOperator{\sgn}{sgn}
\DeclareMathOperator{\supp}{supp}

\DeclareMathOperator{\tr}{tr}

\DeclareMathOperator{\sing}{sing}

\DeclareMathOperator{\rank}{rank}

\DeclareMathOperator{\loc}{loc}

\def\<{\langle}\def\>{\rangle}
\def\({\left(}\def\){\right)}
\numberwithin{equation}{section}
\theoremstyle{plain}
\newtheorem{thm}{Theorem}[section]

\newtheorem{lem}[thm]{Lemma}
\newtheorem{prop}[thm]{Proposition}

\newtheorem{mainthm}{Theorem} 

\theoremstyle{definition}
\newtheorem{defn}[thm]{Definition}

\theoremstyle{remark}
\newtheorem{rem}[thm]{Remark}

\title[Three-Dimensional Global Minimizers of Ginzburg--Landau-Type Functionals]{Singular Limits for Three-Dimensional Global Minimizers of Ginzburg--Landau-Type Functionals: \\ Uniform Estimates and Singular Sets}

\date{}

\author{Giacomo Canevari}
\address{Universit\`{a} di Verona Strada Le Grazie 15, 37134 Verona, Italy}
\email{giacomo.canevari@univr.it}

\author{Haotong Fu}
\address{School of Mathematical Sciences, Peking University, Beijing 100871, China}
\email{2301110012@pku.edu.cn}

\author{Wei Wang}
\address{School of Mathematical Sciences, Peking University, Beijing 100871, China}
\email{gjmtamag@gmail.com,\,\,2201110024@stu.pku.edu.cn}

\begin{document}

\begin{abstract}
We study the asymptotic behavior of global minimizers of a Ginzburg--Landau-type functional with general compact vacuum manifold $\cN$ on bounded domains in $\R^3$, in the regime where the energy grows at a logarithmic rate. We show that the normalized energy measures converge, up to a subsequence, to a measure supported on a finite union of closed line segments connecting prescribed singularities on the boundary. The limit map is a harmonic map valued locally by minimizing $\cN$ away from this singular set. We also establish uniform $W^{1,q}$-estimates with $ q\in(1,2) $ and uniform potential estimates for minimizers, independent of the parameter $\va$. Finally, we prove that the singular set of the limiting measure solves the homotopical Plateau problem in codimension $2$. 
\end{abstract}

\subjclass[2020]{47J30, 49Q20, 58E12, 58E20}

\keywords{Ginzburg--Landau-type functional, global minimizer, stationary varifold, harmonic map, uniform estimate, homotopical Plateau problem}

\maketitle

\section{Introduction}

\subsection{Background and problem setting}

Let $\om\subset\R^n$, $n\geq 2$, be a bounded domain, and let $\cN$ be a compact, connected, smooth Riemannian manifold without boundary, isometrically embedded in $\R^m$ for some integer $m=m(\cN)\geq 1$. We consider the Ginzburg--Landau-type functional
\be
E_{\va}(u,\om):=\int_{\om}e_{\va}(u)\ud x,
\quad
e_{\va}(u):=\f{1}{2}|\na u|^2+\frac{1}{\va^2}f(u),
\tag{GL${}_{\va}$}\label{GLfunctional}
\ee
where $u:\om\to\R^m$ and the potential $f:\R^m\to[0,+\ift)$ is nonnegative, with zero set $ f^{-1}(0)=\cN $.

The basic example is the complex Ginzburg--Landau model, where $u:\om\to\C$ and
\[
f(u)=\f{1}{4}(|u|^2-1)^2,
\]
so that the vacuum manifold is $\Ss^1$. This functional has been studied extensively. In two dimensions, the monograph by Bethuel, Brezis, and H\'{e}lein \cite{BBH93} provides the asymptotic analysis of minimizers in simply connected planar domains with prescribed boundary data taking values in $\Ss^1$. Struwe \cite{Str94} and Rivi\`{e}re \cite{Riv99} studied the asymptotics of critical points. The three-dimensional and higher-dimensional cases were treated by Bethuel, Br\'{e}zis, and Orlandi \cite{BBO01}, Bethuel, Orlandi, and Smets \cite{BOS05}, Bourgain, Br\'{e}zis, and Mironescu \cite{BBM04}, and Lin and Rivi\`{e}re \cite{LR99,LR01}. Much of this progress relies on the special geometric structure of $\Ss^1$.

For a general compact manifold $\cN$, the analysis is more involved. As in the complex Ginzburg--Landau model, two natural energy regimes arise. The first is the uniformly bounded energy regime. In this setting, Lin and Wang \cite{LW99} studied critical points of the functional, while Majumdar and Zarnescu \cite{MZ10}, Nguyen and Zarnescu \cite{NZ13}, Contreras, Lamy, and Rodiac \cite{CLR18}, and Contreras and Lamy \cite{CL22} studied minimizers. In particular, \cite{CL22} treats anisotropic energies. The second regime concerns logarithmic energy bounds of the form
\[
E_{\va}(u_{\va},\om)\leq C(|\log\va|+1).
\]
In two dimensions, the asymptotics of minimizers for the uniaxial Landau--de Gennes model with $\cN=\R\mathbb{P}^2$ were studied in \cite{Can15}, and the case of a general vacuum manifold $\cN$ was treated by Monteil, Rodiac, and Van Schaftingen \cite{MRS21,MRS22}. In three dimensions, the analysis of \cite{Can15} was extended in \cite{Can17}, and the biaxial Landau--de Gennes model with $\cN=\Ss^3/Q_8$ was studied in \cite{WZ24}, where $Q_8$ denotes the quaternion group. In arbitrary dimensions, the functional \eqref{GLfunctional} was analyzed in \cite{CO21} via the $\Ga$-convergence.

Based on \cite{Can17,WZ24}, we study \eqref{GLfunctional} in three dimensions. Throughout the paper, $\om\subset\R^3$ is a bounded domain. We assume that the smooth nonnegative potential $f\in C^{\ift}(\R^m,[0,+\ift))$ satisfies the following structural hypotheses.

\begin{enumerate}[label=$(\op{H}\theenumi)$]
\item\label{A1} The fundamental group $\pi_1(\cN)$ of $\cN=f^{-1}(0)$ is finite.

\item\label{A2} The potential $f$ has the non-vanishing property at infinity:
\[
\liminf_{|y|\to+\ift}f(y)>0.
\]

\item\label{A3} There exists $R_0>0$ such that
\[
|y|\geq R_0\quad\Longrightarrow\quad y\cdot Df(y)\geq 0.
\]

\item\label{A4} The potential $f$ vanishes non-degenerately on $\cN$: for every $x\in\cN$ and every $v\in(T_x\cN)^{\perp}\backslash\{0\}$,
\[
D^2f(x)[v,v]>0.
\]
\end{enumerate}

We mainly consider global minimizers of \eqref{GLfunctional} subject to prescribed boundary data. Following the framework introduced by Lin and Rivi\`{e}re \cite{LR99} for the three-dimensional complex Ginzburg--Landau model, we allow boundary data that are smooth away from finitely many singular points on $\pa\om$ and satisfy uniform blow-up estimates near those points. We now define the class of admissible boundary data used throughout the paper.

\begin{defn}\label{defnsuitabledata}
Let $M_0>0$, and let $\om\subset\R^3$ be a bounded $C^{2,1}$ domain; see Definition \ref{RegularityBoundary}. Assume that $\pa\om$ is simply connected. Let $\cA:=\{a_i\}_{i=1}^k\subset\pa\om$ be a collection of $k\geq 2$ distinct points. We say that a family $\{g_{\va}\}_{\va\in(0,1)}\subset C^2(\pa\om,\R^m)$ is a family of suitable boundary data with respect to $\cA$ and $M_0$ if the following four conditions hold.

\begin{enumerate}[label=$(\theenumi)$]

\item\label{boundarycondition1} For any $j\in\{0,1,2\}$, $\va\in(0,1)$, and $x\in\pa\om$,
\[
|D_{\pa\om}^j g_{\va}(x)|
\leq M_0\bigl(\max\{\dist(x,\cA),\va\}\bigr)^{-j},
\]
where $D_{\pa\om}$ denotes the surface derivative on $\pa\om$, with the convention $D_{\pa\om}^0 g_{\va}:=g_{\va}$.

\item\label{boundarycondition2} Set
\[
\rho_0:=\min\left\{\frac{1}{10}|a_i-a_j|: i,j\in\Z\cap[1,k]\},\ i\neq j\right\}.
\]
For any $\rho\in(0,\rho_0)$ and any $0<\va_1<\va_2<\rho$, one has
\[
g_{\va_1}(x)=g_{\va_2}(x)\in\cN\quad\text{for any }x\in\pa\om\backslash\(\bigcup_{i=1}^k B_\rho(a_i)\).
\].

\item\label{boundarycondition3} For any $0<\va<\rho<\rho_0$ and any $i\in\Z\cap[1,k]$, the loop
$g_{\va}|_{\pa\om\cap\pa B_\rho(a_i)}$ is homotopically non-trivial in $\cN$; see Definition \ref{definitionfreehomotopyclass}.

\item\label{boundarycondition4} There exists $\rho_1\in(0,\rho_0)$ such that, for any $x_0\in\cA$, there exist a diffeomorphism $\Phi_0$ of $B_{\rho_1}(x_0)$ and a rotation $R$ of $\R^3$ such that, for any $\va\in(0,1)$ and any $x\in\pa\om\cap B_{\rho_1}(x_0)$,
\be
g_{\va}(x)=g_{\va}^0\left(\frac{(\Phi_0\circ R)(x)}{\va}\right),
\label{expressionof gvalocal}
\ee
where
\[
\left\{
\begin{aligned}
&\Phi_0(x_0)=x_0,\quad \na\Phi_0(x_0)=\op{Id}_3,\quad
\|D^2\Phi_0\|_{L^{\ift}(B_{\rho_1}(x_0))}\leq M_0,\\
&g_{\va}^0(y_1,y_2,y_3)
=h_{\va}\left(\frac{(y_1,y_2)}{(y_1^2+y_2^2)^{\frac{1}{2}}}\right)
\chi\left((y_1^2+y_2^2)^{\frac{1}{2}}\right).
\end{aligned}
\right.
\]
Here, $\chi$ is an increasing function satisfying $\chi(0)=0$, $\chi\equiv 1$ on $[1,+\ift)$, and $\|\chi'\|_{L^{\ift}(\R)}\leq 100$. Moreover, $h_{\va}:\Ss^1\to\cN$ satisfies
\[
\|D_{\Ss^1}^j h_{\va}\|_{L^{\ift}(\Ss^1)}\leq M_0
\quad\text{for any } j\in\{1,2\}.
\]
\end{enumerate}
\end{defn}

\begin{rem}
Conditions \ref{boundarycondition1}--\ref{boundarycondition3} prescribe a fixed finite set of topological singularities $\cA$ on $\pa\om$ through the boundary data $g_{\va}$ and give uniform control of their behavior near $\cA$. The condition \ref{boundarycondition4} imposes a stronger local structure near each singular point. It is analogous to the assumption in \cite[Page 243, (iii)]{LR99} for the complex Ginzburg--Landau model. It requires that, after a suitable change of coordinates and rescaling, $g_{\va}$ has an angular profile with uniform regularity. This structure will be used later to obtain finer boundary estimates for minimizers.
\end{rem}

The first main result gives compactness and partial regularity for global minimizers as $\va\to0^+$.

\begin{mainthm}\label{globalminimizersproperties}
Let $\om\subset\R^3$ be a bounded $C^{2,1}$ domain such that $\ol{\om}$ is strongly convex at any point $x\in\pa\om$. Let $\cA=\{a_i\}_{i=1}^k\subset\pa\om$, let $M_0>0$, and let $\{g_{\va}\}_{\va\in(0,1)}\subset C^2(\pa\om,\R^m)$ be a family of suitable boundary data with respect to $\cA$ and $M_0$. Let $g\in C^2(\pa\om\backslash\cA,\cN)$ be such that, as $\va\to0^+$,
\[
g_{\va}\to g
\quad\text{in } C_{\loc}^2(\pa\om\backslash\cA,\R^m).
\]
Assume that $\{u_{\va}\}_{\va\in(0,1)}$ is a family of global minimizers of \eqref{GLfunctional} satisfying $u_{\va}=g_{\va}$ on $\pa\om$. Then
\be
E_{\va}(u_{\va},\om)\leq C(|\log\va|+1),
\quad
\|u_{\va}\|_{L^{\ift}(\om)}\leq C,
\label{EvaCbound}
\ee
where $C>0$ depends only on $\cA,f,M_0,\cN $, and $\om$.

Moreover, there exist a Radon measure $\mu_*\in(C^0(\ol{\om}))'$ and a sequence $\va_i\to0^+$ such that
\[
\mu_{\va_i}:=\frac{e_{\va_i}(u_{\va_i})}{|\log\va_i|}\ud x
\wc^*\mu_*
\quad\text{weakly}^* \text{ in } (C^0(\ol{\om}))'.
\]
Set $S_*=\supp(\mu_*)$. Then the following properties hold.

\begin{enumerate}[label=$(\theenumi)$]

\item $S_*\cap\om$ is $1$-rectifiable and $\HH^1(S_*\cap\om)<+\ift$.

\item There exists $u_*\in H^1(\om\backslash S_*,\cN)$ such that $u_*$ is a local minimizer of the Dirichlet energy, in the sense of Definition \ref{Localminimizerdef}. Here, for an open set $D\subset\om$,
\begin{equation}
E_0(u,D):=\frac{1}{2}\int_D|\nabla u|^2\ud x.
\label{Dirichlet}
\end{equation}
Moreover, $ u_{\va_i}\to u_* $ strongly in $ H^1_{\loc}(\om\backslash S_*,\R^m) $.

\item The singular set of $u_*$, denoted by $S_0=\sing(u_*)$, is locally finite in $\om\backslash S_*$, and
\[
u_*\in C^{\ift}(\om\backslash(S_*\cup S_0),\cN).
\]
Moreover, $ u_{\va_i}\to u_* $ in $ C^0_{\loc}(\om\backslash(S_*\cup S_0),\R^m) $.
\end{enumerate}
\end{mainthm}

We next describe the finer structure of the limiting concentration set. Under the strong convexity assumption on $\om$, the set $S_*$ is a finite weighted network whose endpoints are precisely the prescribed boundary singularities.

Recall that $[\Ss^1,\cN]$ denotes the set of free homotopy classes of loops in $\cN$ (see Definition~\ref{definitionfreehomotopyclass} below); for a loop $\ell:\Ss^1\to\cN$, we write $[\ell]_{\cN}$ for its free homotopy class. This set is equipped with the norm $|\cdot|_*$ defined in Section~\ref{SectionPre} below --- more precisely, see Definition~\ref{DefinitionEsg}.

\begin{mainthm}\label{structureofthelimitset}
Assume that $\om,\cA,S_*,g,u_*$, and $\mu_*$ are as in Theorem
\ref{globalminimizersproperties}. Then the following properties hold.

\begin{enumerate}[label=$(\theenumi)$]

\item $S_*\cap\pa\om=\cA$, and there exists a finite collection of closed line segments
$\{L_j\}_{j\in J}\subset\ol{\om}$ such that $\#J<+\infty$,
\[
S_*=\bigcup_{j\in J}L_j,
\]
and the relative interiors of the segments $L_j$ are pairwise disjoint. Moreover,
$S_*\subset\operatorname{Conv}(\cA)$, where $\operatorname{Conv}(\cA)$ denotes the convex hull of $\cA$.

\item Let $j\in J$, and let $x\in\op{Int}L_j$ $($$\op{Int}(L_j)$ denotes the relative interior of $ L_j $$)$. Suppose that
$\ol{D}\subset\om$ is a closed two-dimensional disk centered at $x$ such that
$D\cap S_*=\{x\}$ and $\pa D\cap S_0=\emptyset$. Then
\begin{equation}
\Theta^1(\mu_*,x):=\lim_{r\to0^+}
\frac{\mu_*(\ol{B}_r(x))}{2r}
=|\al_j|_*,
\label{alifirstappear}
\end{equation}
where $ \al_j:=[u_*|_{\pa D}]_{\cN}\in[\Ss^1,\cN] $. The class $\al_j$ is independent of the choice of $x\in\op{Int}L_j$ and of the disk $D$ satisfying the above conditions. In particular,
\[
\mu_*=\sum_{j\in J}|\al_j|_*\HH^1\llcorner L_j.
\]

\item Let $x\in S_*\cap\om$ be an interior branching point of $S_*$, that is, an endpoint in $\om$ of at least two segments in the family $\{L_j\}_{j\in J}$. Set
\[
J_x:=\{j\in J:x \text{ is an endpoint of } L_j\}.
\]
For each $j\in J_x$, let $v_j$ be the unit direction vector of $L_j$ pointing away from $x$. Then
\[
\sum_{j\in J_x}|\al_j|_*v_j=0.
\]

\item Let $j\in J$, and let $x$ be an endpoint of $L_j$. Then one of the following alternatives holds: either
\begin{enumerate}[label=$(\operatorname{\alph*})$]

\item $x$ is a branching point of $S_*$ lying in $\om$, or

\item $x\in\cA$. Moreover, $ [g|_{\pa D_x}]_{\cN}=[u_*|_{\pa D}]_{\cN} $, where $D_x$ is a geodesic disk on $\pa\om$ centered at $x$ such that
$D_x\cap\cA=\{x\}$, and $\ol{D}\subset\om$ is a closed two-dimensional disk centered at an interior point of $L_j$ such that
$D\cap S_*=\{\text{center of }D\}$ and
$\pa D\cap(\operatorname{Conv}(\cA)\cup S_0)=\emptyset$.

\end{enumerate}
\end{enumerate}
\end{mainthm}

\begin{rem}\label{remarkfirsttheorem}
We record several comments on Theorems \ref{globalminimizersproperties} and \ref{structureofthelimitset}.
\begin{enumerate}

\item The results of \cite{Can17, WZ24} are primarily related to the interior behavior of the minimizers in $\om$. Since we consider global minimizers with prescribed boundary data, Theorems \ref{globalminimizersproperties} and \ref{structureofthelimitset} provide both interior and boundary descriptions of the limiting behavior. In particular, the characterization of $S_*\cap\pa\om$ and the endpoint condition in Theorem \ref{structureofthelimitset} are new.

\item A key ingredient in the proof is the clearing-out property, both in the interior of $\om$ and near $\pa\om$. It states that, in a ball or half-ball, if the normalized energy $\f{E_{\va}(u_{\va},B_2)}{|\log\va|}$ is sufficiently small, then the energy $E_{\va}(u_{\va},B_1)$ is uniformly bounded.

\item The limiting behavior described above is analogous to that of sequences of $p$-harmonic maps as $p\to2^-$ whose energies are bounded from above by $\f{C}{2-p}$ for some $p$-independent constant~$C$, as studied by Bulanyi--Van Schaftingen--Vaerenbergh \cite{BSV25}. Several arguments in our proof are adapted from \cite{BSV25}. However, the potential term and the different structures of the boundary conditions introduce additional difficulties that do not occur in the $p$-harmonic map setting.

\end{enumerate}
\end{rem}

The next theorem gives estimates for the minimizers $u_{\va}$ that are uniform as $\va\to0^+$.

\begin{mainthm}\label{W1pestimate}
Let $\cA,M_0,\om,g_{\va}$, and $u_{\va}$ be as in Theorem \ref{globalminimizersproperties}. Then the following properties hold.
\begin{enumerate}[label=$(\theenumi)$]

\item For any $\va\in(0,1)$ and any $q\in(1,2)$,
\be
\|\na u_{\va}\|_{L^q(\om)}\leq C,
\label{W1pestimateeq1}
\ee
where $C>0$ depends only on $\cA,f,M_0,\cN,\om$, and $q$.

\item For any $\va\in(0,1)$,
\be
\f{1}{\va^2}\int_{\om}f(u_{\va})\ud x\leq C,
\label{W1pestimateeq11}
\ee
where $C>0$ depends only on $\cA,f,M_0,\cN$, and $\om$.

\end{enumerate}
\end{mainthm}

\begin{rem}\label{remarkuniformestimates}
We make several remarks on the uniform estimates in Theorem \ref{W1pestimate}.
\begin{enumerate}

\item For the complex Ginzburg--Landau model, interior estimates of the form \eqref{W1pestimateeq1} were first obtained by Bourgain--Br\'{e}zis--Mironescu \cite{BBM04} and Bethuel--Orlandi--Smets \cite{BOS05}. The potential estimate \eqref{W1pestimateeq11} was established in \cite{BOS05}. For the Landau--de Gennes model, Fu--Wang--Wang \cite{FWW25b} obtained interior versions of both \eqref{W1pestimateeq1} and \eqref{W1pestimateeq11}. Theorem \ref{W1pestimate} extends these estimates to the present setting and gives bounds up to the boundary.

\item For the $p$-harmonic map model with $p\to2^-$, $W^{1,q}$-estimates with $q\in(1,2)$ were obtained in \cite{BSV25}. However, the method of \cite{BSV25} does not apply directly to the Ginzburg--Landau-type setting. In the proof of Theorem \ref{W1pestimate}, we follow the approach of \cite{FWW25b}. The main new difficulty is the extension of the analysis to neighborhoods of $\pa\om$, where one needs boundary regularity estimates for the minimizers.

\item Both estimates in Theorem \ref{W1pestimate} are sharp. Indeed, in view of \eqref{EvaCbound}, the exponent $q\in(1,2)$ in \eqref{W1pestimateeq1} cannot be improved to $q=2$. The sharpness of \eqref{W1pestimateeq11} is discussed in \cite[Section 5]{FWW25b}.

\end{enumerate}
\end{rem}

Since the maps $u_{\va}$ are global minimizers, the limiting weighted network is expected to inherit a minimizing property. We now introduce the admissible competitors used to formulate this property. The definition follows the spirit of \cite[Theorem 1.2]{BSV25}, with boundary data allowed to have finitely many singular points.

\begin{defn}[Admissible competitors]\label{defadmissible}
Let $U\subset\R^3$ be a bounded Lipschitz domain. Let
$h\in H^{\f{1}{2}}(\pa U,\cN)$. An admissible competitor for $(U,h)$ is a finite family
\[
\Gamma=\{(K_i,\sg_i):i\in I\},
\]
where $I$ is a finite index set, each $K_i\subset\ol{U}$ is a closed line segment, and each
$\sg_i\in[\Ss^1,\cN]$ is a free homotopy class; see Definition \ref{definitionfreehomotopyclass}. Set
\[
S_\Gamma:=\bigcup_{i\in I}K_i.
\]
We require the following conditions.
\begin{enumerate}[label=$(\theenumi)$]

\item\label{property1} Each $K_i$ has positive length, and $ \op{Int}K_i\cap\op{Int}K_j=\emptyset $ whenever $ i\neq j $.

\item\label{property2} For any $i\in I$ and any endpoint $x$ of $K_i$, either
$x\in\pa U$, or $x$ is an interior branching point of $S_\Gamma$, in the sense that
$x\in U$ and $x$ is also an endpoint of $K_j$ for some $j\neq i$.

\item\label{property3} The set $S_\Gamma\cap\pa U$ is finite.

\item\label{property4} There exist a locally finite set
$F\subset U\backslash S_\Gamma$ and a map
\[
v\in H_{\loc}^1(\ol{U}\backslash S_\Gamma,\cN)
\cap C^0(U\backslash(S_\Gamma\cup F),\cN)
\]
such that $v|_{\pa U}=h$ in the local trace sense and the following meridian condition holds. For any $i\in I$, any $x\in\op{Int}K_i$, and any closed two-dimensional disk $\ol D\subset U$ centered at $x$ such that $D\cap S_\Gamma=\{x\}$ and $\pa D\cap F=\emptyset$, one has $ [v|_{\pa D}]_{\cN}=\sg_i $.
\end{enumerate}
We denote by $\cC(U,h)$ the class of all admissible competitors for $(U,h)$.
\end{defn}

For any finite family $\Gamma=\{(K_i,\sg_i):i\in I\}$ of the form above, we define its weighted length by
\[
\MM(\Gamma):=\sum_{i\in I}|\sg_i|_*\HH^1(K_i).
\]

With the class of admissible competitors fixed, we can state the minimality property of the limiting singular set. The labeled network obtained in the limit has a weighted length no larger than that of any admissible competitor with the same boundary datum.

\begin{mainthm}\label{minimalitytheorem}
Let $g,S_*,\{L_j\}_{j\in J}$, and $\{\al_j\}_{j\in J}$ be as in Theorem
\ref{structureofthelimitset}. Set
\[
\Gamma_*:=\{(L_j,\al_j):j\in J\}.
\]
Then $\Gamma_*\in\cC(\om,g)$ and, for any admissible competitor $\Gamma\in\cC(\om,g)$,
\[
\MM(\Gamma_*)\leq \MM(\Gamma).
\]
Equivalently, if $\Gamma=\{(K_i,\sg_i):i\in I\}\in\cC(\om,g)$, then
\[
\sum_{j\in J}|\al_j|_*\HH^1(L_j)\leq\sum_{i\in I}|\sg_i|_*\HH^1(K_i).
\]
In particular, the labeled network $\Gamma_*$ solves the homotopical Plateau problem in codimension $2$ within the class $\cC(\om,g)$.
\end{mainthm}

\begin{rem}
The proof proceeds by constructing, for each admissible competitor
$\Gamma=\{(K_i,\sg_i):i\in I\}$, a competitor for the functional \eqref{GLfunctional} whose rescaled energy converges to $ \MM(\Ga) $. The minimality of $u_{\va}$ then gives the desired inequality in the limit. The construction follows the strategy of \cite{BSV25}.
\end{rem}

\subsection{Organization of this paper}

In Section \ref{SectionPre}, we collect the preliminary material used throughout the paper, including the topological analysis of the target manifold $\cN$ through free homotopy classes, their multi-valued addition, and the associated norm, as well as the monotonicity formula, energy bounds, and compactness properties. In Section \ref{SectionLuc}, we prove Luckhaus-type interpolation and extension results, which connect Ginzburg--Landau maps to $\cN$-valued maps across thin annuli and cylinders and will be used in the clearing-out and singular-set arguments. In Section \ref{SectionClear}, we derive the clearing-out properties, both in the interior and near the boundary; that is, we show that if the normalized energy on a ball or half-ball is small compared to~$|\log\va|$, then the energy on a smaller ball or half-ball is uniformly bounded with respect to~$\va$. In Section \ref{SectionPartial}, we prove interior and boundary partial regularity results. In Section \ref{SectionSingular}, we analyze the singular set $S_*$ of the limiting measure and prove Theorems \ref{globalminimizersproperties} and \ref{structureofthelimitset}. In Section \ref{SectionUniform}, we establish the uniform estimates and prove Theorem \ref{W1pestimate}. In Section \ref{SectionMinimality}, we prove the minimality result, Theorem \ref{minimalitytheorem}.

\subsection{Notations and conventions}

\begin{itemize}

\item Throughout this paper, $C$ denotes a positive constant that may change from line to line. When we need to specify the dependence of $C$ on parameters $a,b,...$, we write $C(a,b,...)$.

\item We adopt the Einstein summation convention: repeated indices are implicitly summed over their range.

\item For $k\in\Z$ with~$k\geq 2$, we write $ B_r^k(x):=\{y\in\R^k:|y-x|<r\} $. We denote the origin in $\R^k$ by $0^k$. When $k=3$, we omit the superscript $k$; when $x=0^k$, we omit the center.

\item We denote the Euclidean inner product in $\R^n$ by $u\cdot v$ for $u,v\in\R^n$. We use $u:v$ for the Euclidean inner product in $\R^m$ when this notation is more convenient.

\item Let $k\in\{1,2,3\}$, and let $\M\subset\R^3$ be a Lipschitz Riemannian $k$-submanifold of $\R^3$. We denote by $\na_{\M}$ the tangential gradient on $\M$. This gradient exists $\HH^k$-a.e. by Rademacher's theorem. If $\M$ is $C^j$ with $j\in\Z$, $j\geq 2$, we denote by $D_{\M}^j$ the corresponding higher-order tangential derivatives on $\M$. If $u\in H^1(\M,\R^m)$ and $0<\va<1$, we define
\[
e_{\va}(u,\M):=\f{1}{2}|\na_{\M}u|^2+\f{1}{\va^2}f(u),
\quad
E_{\va}(u,\M):=\int_{\M} e_{\va}(u,\M)\ud\HH^k.
\]
When $k=3$ and $\M$ is a three-dimensional domain, we write $\na_{\M}=\na$, $D_{\M}^j=D^j$, and $e_{\va}(u,\M)=e_{\va}(u)$. If there is no ambiguity, we use $\na_{\top}$ to denote tangential derivatives with respect to the relevant submanifold.

\item We denote by $\HH^k$ the $k$-dimensional Hausdorff measure. When $k=3$, $\HH^3$ coincides with the Lebesgue measure on $\R^3$. We sometimes omit $\ud\cL^3$ in volume integrals.

\item Let $U\subset\R^j$, $j\in\Z$, $j\geq 2$ be a bounded domain, let $x_0\in\ol U$, and let $r>0$. We define
\[
U_r(x_0):=U\cap B_r(x_0),\quad
\ol{U}_r(x_0):=\ol{U}\cap\ol{B}_r(x_0),\quad
T_r^U(x_0):=\pa U\cap B_r(x_0).
\]
When $B_r(x_0)\subset\subset U$, we have $U_r(x_0)=B_r(x_0)$ and $T_r^U(x_0)=\emptyset$. If $x_0=0^j$, we omit $x_0$ from the notation.

\item We denote the upper half-space by $ \R_+^3:=\{y\in\R^3:y_3>0\} $. For $x\in\pa\R_+^3$ and $r>0$, we define
\begin{align*}
B_r^+(x)&:=\{y\in\R^3:|y-x|<r,\ y_3>0\},\\
T_r(x)&:=\{y\in\pa\R_+^3:|y-x|<r\},\\
(\pa B_r(x))^+&:=\{y\in\R^3:|y-x|=r,\ y_3>0\}.
\end{align*}
For simplicity, we write
\[
B_r:=B_r(0),\quad T_r:=T_r(0),\quad
B_r^+:=B_r^+(0),\quad(\pa B_r)^+:=(\pa B_r(0))^+.
\]
Then $\pa(B_r^+)=(\pa B_r)^+\cup T_r$.

\end{itemize}

\section{Preliminaries}\label{SectionPre}

\subsection{Basic properties of the manifold \texorpdfstring{$\cN$}{} and the potential \texorpdfstring{$f$}{}}

For $\delta>0$, we denote the $\delta$-neighborhood of $\cN$ by
\[
\cN_{\delta}:=\{y\in\R^m:\dist(y,\cN)<\delta\}.
\]
The next lemma collects the geometric and analytic properties of $f$ near $\cN$ that will be used throughout the paper.

\begin{lem}\label{Nfproperty}
Assume that $f$ and $\cN$ satisfy \ref{A1}--\ref{A4}. Then there exists $\delta_{\cN}>0$, depending only on $\cN$, such that the nearest-point projection onto $\cN$ is well-defined. More precisely, there exists a unique $C^{\infty}$ map $\Pi_{\cN}:\cN_{\delta_{\cN}}\to\cN$ such that
\begin{equation}
|y-\Pi_{\cN}(y)|=\dist(y,\cN)
\quad\text{for any } y\in\cN_{\delta_{\cN}}.
\label{Nearestpointprojection}
\end{equation}
Moreover, there exist $\delta_f\in(0,\delta_{\cN})$ and constants $0<m_f<M_f<+\infty$ such that, for any $y\in\cN_{\delta_f}$,
\begin{equation}
\begin{gathered}
m_f(\dist(y,\cN))^2\leq Df(y):(y-\Pi_{\cN}(y))\leq M_f(\dist(y,\cN))^2,\\
m_f(\dist(y,\cN))^2\leq f(y)\leq M_f(\dist(y,\cN))^2,
\end{gathered}\label{mfMfestimates}
\end{equation}
and
\begin{equation}
f(ty+(1-t)\Pi_{\cN}(y))\leq M_f t^2 f(y)
\quad\text{for any } t\in[0,1].
\label{fBconvex}
\end{equation}
\end{lem}

\begin{proof}
The existence and smoothness of $\Pi_{\cN}$ follow from \cite[Lemma 2.1]{MRS21}. The estimates in \eqref{mfMfestimates} follow from \cite[Lemmas 2.2 and 2.3]{MRS21} together with \ref{A4}. Finally, \eqref{fBconvex} follows from Taylor's theorem applied along the segment joining $\Pi_{\cN}(y)$ to $y$, after decreasing $\delta_f$ if necessary and increasing $M_f$.
\end{proof}

\begin{rem}\label{remfproperty}
Let $\delta_f$ be as in Lemma \ref{Nfproperty}. For any $y\in\cN_{\delta_f}\backslash\cN$, set
\[
\nu(y):=\frac{y-\Pi_{\cN}(y)}{|y-\Pi_{\cN}(y)|}.
\]
We decompose
\[
Df(y)=(Df(y))^{\perp}+(Df(y))^{\top},
\]
where $(Df(y))^{\top}$ is the component of $Df(y)$ parallel to $y-\Pi_{\cN}(y)$ and $(Df(y))^{\perp}$ is orthogonal to $y-\Pi_{\cN}(y)$. Explicitly,
\[
(Df(y))^{\top}:=(Df(y):\nu(y))\nu(y).
\]
For $y\in\cN$, we set $(Df(y))^{\top}=0$. Since $|(Df(y))^{\top}|\leq |Df(y)|$, the lower bound in \eqref{mfMfestimates} gives
\begin{equation}
|Df(y)|\geq |(Df(y))^{\top}|
\geq m_f\dist(y,\cN).
\label{Dfyestimate}
\end{equation}
\end{rem}

\subsection{Boundary analysis}

We begin by specifying the regularity assumptions on the boundary of~$\Omega$. 

\begin{defn}\label{RegularityBoundary}
Let $U\subset\R^j$ be a bounded domain of dimension~$j\geq 2$. We say that $U$ is $C^{k,1}$, with $k$ a nonnegative integer if there exist constants $M_{U,k}>0$ and $r_{U,k}>0$ such that the following property holds. For any $x_0\in\partial U$, there exist a $C^{k,1}$ function $\psi:\R^{j-1}\to\R$ and a coordinate system, obtained by a translation and a rotation, which sends $x_0$ to $0^j\in\R^j$ and such that, in these coordinates,
\begin{equation}
U_r(x_0):=U\cap B_r(x_0)
=
\{y=(y',y_j)\in\R^j:y_j>\psi(y')\}\cap B_r(x_0)
\label{UcapB}
\end{equation}
for any $r\in(0,60(M_{U,k}+1)r_{U,k})$. Moreover,
\begin{equation}
\psi(0^{j-1})=0,
\quad
\|D^i\psi\|_{L^{\infty}(\R^{j-1})}\leq M_{U,k}
\quad\text{for any } i\in\Z\cap[1,k+1].
\label{Dkest}
\end{equation}
When $k=0$, this means that $U$ is Lipschitz. When $k\geq 1$, we choose the above coordinate system so that, in addition, $\na\psi(0^{j-1})=0$. We call $U$ a bounded $C^{k,1}$ domain with parameters $r_{U,k}$ and $M_{U,k}$ if $U$ satisfies \eqref{UcapB} and \eqref{Dkest}. We say that $U$ is smooth if it is $C^{k,1}$ for any $k\in\Z$, $k\geq 0$.
\end{defn}

Having fixed the regularity of the boundary, we introduce the notion of local minimizers, both in the interior and up to a prescribed portion of the boundary.

\begin{defn}\label{Localminimizerdef}
Let $U\subset\R^n$ be a bounded Lipschitz domain, let $\Omega\supset U$ be an open set, and let
\[
F:\Omega\times\R^m\times\R^{m\times n}\to\R
\]
be a measurable integrand. For an open set $D\subset U$, define
\[
\cF(u,D):=\int_D F(x,u,\nabla u)\ud x.
\]
Let $\Gamma\subset\partial U$ be relatively open in $\partial U$.

\begin{enumerate}[label=$(\theenumi)$]

\item We say that $u\in H^1(U,\R^m)$ is a local minimizer of $\cF$ in $U$ if, for any ball $B_r(x)\subset\subset U$ and any $v\in H^1(B_r(x),\R^m)$ with $v=u$ on $\partial B_r(x)$ in the trace sense,
\[
\cF(u,B_r(x))\leq \cF(v,B_r(x)).
\]

\item We say that $u\in H^1(U,\R^m)$ is a local minimizer of $\cF$ in $U$ up to the boundary $\Gamma$ if, for any $x\in U\cup\Gamma$ and any $r>0$ such that $U_r(x)\subset\Omega$ and $ \partial(U_r(x))\cap\partial U\subset\Gamma $, and for any $v\in H^1(U_r(x),\R^m)$ with $v=u$ on $\partial(U_r(x))$ in the trace sense, one has
\[
\cF(u,U_r(x))\leq \cF(v,U_r(x)).
\]
Here $U_r(x)=U\cap B_r(x)$ is as in \eqref{UcapB}.

\end{enumerate}
\end{defn}

\subsection{Topology of the manifold \texorpdfstring{$\cN$}{}}

In this subsection, we recall the topological notation used to describe line defects. We introduce free homotopy classes in $\cN$, the natural multi-valued addition on these classes, and the associated norm. These notions go back to \cite{Mer79}; see also \cite{Can15,CO21,MRS21,MRS22}. For a more complete account, we refer to \cite[Section 2.2]{GFW26}.

\begin{defn}\label{definitionfreehomotopyclass}
Let $\ell_1,\ell_2\in C^0(\Ss^1,\cN)$ be two loops. We say that $\ell_1$ and $\ell_2$ are freely homotopic if there exists $G\in C^0([0,1]\times\Ss^1,\cN)$ such that
\[
G(0,\cdot)=\ell_1
\quad\text{and}\quad
G(1,\cdot)=\ell_2.
\]
We denote this relation by $\ell_1\sim_{\cN}\ell_2$. Free homotopy is an equivalence relation on $C^0(\Ss^1,\cN)$, and we denote the quotient set as
\[
[\Ss^1,\cN]:=C^0(\Ss^1,\cN)/\sim_{\cN}.
\]
For $\ell\in C^0(\Ss^1,\cN)$, we denote its class by $[\ell]_{\cN}$. We will say that a loop~$\ell\in C^0(\Ss^1, \cN)$ is homotopically trivial if~$[\ell]_{\cN}$ contains constant loops.
\end{defn}

Unlike homotopy classes defining the fundamental group $\pi_1(\cN,x)$, free homotopy classes do not involve a fixed base point. It follows from \cite[Exercise 6, p. 38]{Hat02} that, after fixing a base point in $\cN$, there is a natural bijection
\be
[\Ss^1,\cN]\longleftrightarrow
\{\text{conjugacy classes of }\pi_1(\cN)\}.
\label{bijection}
\ee

We next recall the operation on free homotopy classes that records how topological charges combine when several line defects merge.

\begin{defn}\label{Definitionplus}
Let $\cP([\Ss^1,\cN])$ be the power set of $[\Ss^1,\cN]$. We define an operation
\[
+:[\Ss^1,\cN]\times[\Ss^1,\cN]\to\cP([\Ss^1,\cN])
\]
as follows. Let $\{B_{r_i}^2(x_i)\}_{i=0}^2$ be three balls in $\R^2$ such that
\[
B_{r_1}^2(x_1)\cup B_{r_2}^2(x_2)\subset B_{r_0}^2(x_0)
\quad\text{and}\quad
\ol{B}_{r_1}^2(x_1)\cap\ol{B}_{r_2}^2(x_2)=\emptyset.
\]
Given $\al,\beta\in[\Ss^1,\cN]$, we define $\al+\beta$ as the set of all $\sg\in[\Ss^1,\cN]$ for which there exists $G\in C^0(\ol{B}_{r_0}^2(x_0),\cN)$ such that
\[
[G|_{\pa B_{r_1}^2(x_1)}]_{\cN}=\al,
\quad
[G|_{\pa B_{r_2}^2(x_2)}]_{\cN}=\beta,
\quad\text{and}\quad
[G|_{\pa B_{r_0}^2(x_0)}]_{\cN}=\sg.
\]
\end{defn}

By \eqref{bijection}, each element of $[\Ss^1,\cN]$ can be identified with a conjugacy class of $\pi_1(\cN)$. Under this identification, for any $\al,\beta\in[\Ss^1,\cN]$,
\be
\al+\beta
=
\{\text{conjugacy classes of }ab:a\in\al,\ b\in\beta\}.
\label{conjugacy}
\ee
See, for instance, \cite[Section III.C]{Mer79} or \cite[Lemma 2.2]{Can15} for a proof of \eqref{conjugacy}. 
The operation may be multi-valued when $\pi_1(\cN)$ is non-Abelian, because~$ab$ need not be conjugate to~$a^\prime b^\prime$ if~$a$ is conjugate to~$a^\prime$ and~$b$ is conjugate to~$b^\prime$.

We record the elementary algebraic properties of this operation. We write $0$ for the free homotopy class of constant loops and identify a singleton subset of $[\Ss^1,\cN]$ with its unique element when no confusion can arise. Then:
\begin{enumerate}[label=(\roman*)]

\item $0+\al=\al$ for any $\al\in[\Ss^1,\cN]$;

\item for any $\al\in[\Ss^1,\cN]$, there exists a class $-\al\in[\Ss^1,\cN]$ such that
\[
0\in\al+(-\al);
\]

\item the operation is associative in the multi-valued sense, namely
\[
(\al+\beta)+\gamma
:=
\bigcup_{\sg\in\al+\beta}(\sg+\gamma)
=
\bigcup_{\sg\in\beta+\gamma}(\al+\sg)
=:
\al+(\beta+\gamma)
\]
for any $\al,\beta,\gamma\in[\Ss^1,\cN]$;

\item $\al+\beta=\beta+\al$ for any $\al,\beta\in[\Ss^1,\cN]$.

\end{enumerate}
These properties follow directly from \eqref{conjugacy}. For example, commutativity follows from the identity $ab=(aba^{-1})a$ in $\pi_1(\cN)$.

For $\sg\in[\Ss^1,\cN]$, define
\[
E_{\op{min}}(\sg)
:=
\inf\left\{
\frac{1}{2}\int_{\Ss^1}|u'|^2\ud\HH^1:
u\in H^1(\Ss^1,\cN),\ [u]_{\cN}=\sg
\right\}.
\]
Here, an $H^1$ loop is identified with its continuous representative. By the direct method of the calculus of variations, the infimum is attained by a geodesic representative of the class $\sg$.

The following quantity is the cost assigned to a topological charge. The definition allows a class to split into several classes before the energies are added.

\begin{defn}\label{DefinitionEsg}
For $\sg\in[\Ss^1,\cN]$, define
\be
|\sg|_*:=\inf\left\{
\sum_{i=1}^k E_{\op{min}}(\sg_i):
\sg\in\sum_{i=1}^k\sg_i,\ \sg_i\in[\Ss^1,\cN],\ k\in\Z, \ k\geq 1
\right\}.
\label{normsg}
\ee
Here $\sum_{i=1}^k\sg_i$ is understood in the iterated multi-valued sense, using the associativity stated above.
\end{defn}

By Assumption \ref{A1}, the group $\pi_1(\cN)$ is finite. Hence $[\Ss^1,\cN]$ is finite, and the infimum in \eqref{normsg} is achieved. We call $|\cdot|_*$ a norm on $[\Ss^1,\cN]$ because it satisfies the following properties:
\begin{itemize}

\item $|\sg|_*=0$ if and only if $\sg=0$;

\item if $\sg,\sg'\in[\Ss^1,\cN]$ and $0\in\sg+\sg'$, then $|\sg|_*=|\sg'|_*$;

\item if $\sg\in\sg_1+\sg_2$, then $|\sg|_*\leq |\sg_1|_*+|\sg_2|_*$.

\end{itemize}
All these properties easily follow from Definition~\ref{DefinitionEsg}.
Moreover, there exists a constant $c_0=c_0(\cN)>0$ such that
\be
|\sg|_*\geq c_0
\quad\text{for any } \sg\in[\Ss^1,\cN]\backslash\{0\}.
\label{c0geq}
\ee

\begin{rem}\label{remequvalence}
In \cite{MRS21,MRS22}, the authors use the singular energy $\cE^{\op{sg}}(\cdot)$. This notion is equivalent to the norm $|\cdot|_*$ in the sense that it assigns the same cost to each free homotopy class.
\end{rem}

\subsection{Extension theory}

We recall two extension results for manifold-valued maps that will be used later on. The first one treats homotopically trivial boundary data on a two-dimensional disk.

\begin{lem}[\cite{GFW26}, Lemma 2.7]\label{ExtensionLemma1}
Let $r>0$, and let $g\in H^1(\pa B_r^2,\cN)$ be homotopically trivial. Then there exists $u\in H^1(B_r^2,\cN)$ such that $u|_{\pa B_r^2}=g$ and
\begin{equation}
\|\na u\|_{L^2(B_r^2)}^2
\leq
Cr\|\na_{\top}g\|_{L^2(\pa B_r^2)}^2,
\label{ExtensionLemma1eq}
\end{equation}
where $C>0$ depends only on $\cN$.
\end{lem}

The next lemma gives the corresponding extension estimate for balls in $\R^3$.

\begin{lem}[\cite{GFW26}, Lemma 2.8]\label{ExtensionLemma2}
Let $r>0$, and let $g\in H^1(\pa B_r,\cN)$. Then there exists $u\in H^1(B_r,\cN)$ such that $u|_{\pa B_r}=g$ and
\begin{equation}
\|\na u\|_{L^2(B_r)}^2
\leq
Cr\|\na_{\top}g\|_{L^2(\pa B_r)}^2,
\label{ExtensionLemma2eq}
\end{equation}
where $C>0$ depends only on $\cN$.
\end{lem}

\subsection{A priori regularity estimates}

Let $U\subset\R^3$ be a bounded domain. If $u$ is a local minimizer of \eqref{GLfunctional}, then it satisfies the Euler--Lagrange equation
\begin{equation}
\Delta u=\frac{1}{\va^2}D_u f(u).
\label{EL}
\end{equation}
In this subsection, we collect the a priori gradient estimates that we will use later. The first is an interior estimate for the bounded classical solutions of \eqref{EL}.

\begin{lem}\label{Apriori}
Let $\va\in(0,1)$, $M,r>0$, and $x\in\R^3$. Suppose that
$u\in C^2(B_{2r}(x),\R^m)$ is a solution of \eqref{EL} satisfying
$\|u\|_{L^\infty(B_{2r}(x))}\leq M$. Then
\begin{equation}
\|\na u\|_{L^\infty(B_r(x))}
\leq
C(\va^{-1}+r^{-1}),
\label{Apriori1}
\end{equation}
where $C>0$ depends only on $f,M$, and $\cN$.
\end{lem}

\begin{proof}
This follows directly from \cite[Lemma A.1]{BBH93}.
\end{proof}

We also need a boundary analogue of Lemma~\ref{Apriori}. We first record a standard elliptic boundary estimate and then apply it to \eqref{EL}.

\begin{lem}\label{aprioriboundary}
Let $U\subset\R^3$ be a bounded $C^{2,1}$ domain with parameters $M_{U,2}$ and $r_{U,2}$. Let $x_0\in\pa U$, let $r\in(0,r_{U,2})$, and let
$g\in C^2(T_{2r}^U(x_0),\R^m)$. Suppose that
$u\in C^2(\ol{U}_{2r}(x_0),\R^m)$ is a classical solution of
$\Delta u=F$ in $U_{2r}(x_0)$, where $F\in L^\infty(U_{2r}(x_0),\R^m)$, and that $u=g$ on $T_{2r}^U(x_0)$. Then
\begin{equation}
\begin{aligned}
\|\na u\|_{L^\infty(U_r(x_0))}
\leq C\bigl(
&r^{-1}\|u\|_{L^\infty(U_{2r}(x_0))}
+r\|F\|_{L^\infty(U_{2r}(x_0))} \\
&+\| |\na_{\top}g|+r|D_{\top}^2g| \|_{L^\infty(T_{2r}^U(x_0))}
\bigr),
\end{aligned}
\label{aprioriboundaryeq}
\end{equation}
where $C>0$ depends only on $M_{U,2}$ and $r_{U,2}$.
\end{lem}

\begin{proof}
After straightening the boundary and rescaling, the estimate follows from \cite[Lemma~11]{NZ13}, together with the Caccioppoli inequality.
\end{proof}

Lemma~\ref{aprioriboundary} implies the following boundary gradient estimate for solutions of the Ginzburg--Landau equation.

\begin{lem}\label{AprioriBoundary}
Let $\va\in(0,1)$, $M>0$, and let $U\subset\R^3$ be a bounded $C^{2,1}$ domain with parameters $M_{U,2}$ and $r_{U,2}$. Let $x_0\in\pa U$ and $r\in(0,r_{U,2})$. Assume that $g\in C^2(T_{2r}^U(x_0),\R^m)$ satisfies
\[
\|D_{\pa U}^j g\|_{L^\infty(T_{2r}^U(x_0))}
\leq M\va^{-j}
\]
for any $j\in\{0,1,2\}$, with the convention $D_{\pa U}^0g:=g$. Suppose that $u\in H^1(U_{2r}(x_0),\R^m)$ is a weak solution of \eqref{EL} in $U_{2r}(x_0)$ satisfying
\[
u=g\text{ on }T_{2r}^U(x_0)
\quad\text{and}\quad
\|u\|_{L^\infty(U_{2r}(x_0))}\leq M.
\]
Then $u\in C^1(\ol{U}_r(x_0),\R^m)$ and
\begin{equation}
\|\na u\|_{L^\infty(U_r(x_0))}
\leq
C\va^{-2}r^{-1}(r+\va)^2,
\label{Apriori11}
\end{equation}
where $C>0$ depends only on $f,M,\cN,M_{U,2}$, and $r_{U,2}$.
\end{lem}

\begin{proof}
By elliptic regularity, $u$ is sufficiently regular in $\ol{U}_r(x_0)$ to apply Lemma \ref{aprioriboundary}. We use Lemma \ref{aprioriboundary} with
$F:=\va^{-2}D_uf(u)$. Since $\|u\|_{L^\infty(U_{2r}(x_0))}\leq M$, we have
\[
\|F\|_{L^\infty(U_{2r}(x_0))}
\leq
C\va^{-2}.
\]
The assumptions on $g$ give
\[
\| |\na_{\top}g|+r|D_{\top}^2g| \|_{L^\infty(T_{2r}^U(x_0))}
\leq
C(\va^{-1}+r\va^{-2}).
\]
Together with \eqref{aprioriboundaryeq}, this yields
\[
\|\na u\|_{L^\infty(U_r(x_0))}
\leq
C(r^{-1}+\va^{-1}+r\va^{-2})
\leq
C\va^{-2}r^{-1}(r+\va)^2.
\]
Then we complete the proof.
\end{proof}

\subsection{Monotonicity formula}

Let $U\subset\R^3$ be open. Throughout this subsection, we assume that
$u\in C^\infty(U,\R^m)$ satisfies \eqref{EL}. For $x\in U$ and
$r\in(0,\dist(x,\pa U))$, define
\begin{equation}
\theta_{\va}(u;x,r):=\frac{1}{r}\int_{B_r(x)}e_{\va}(u).
\label{density1}
\end{equation}

We use the following cutoff function to localize the monotonicity formula.

\begin{defn}\label{defnofphi}
Let $\phi\in C^\infty([0,100),[0, +\infty))$ satisfy the following conditions.
\begin{enumerate}[label=$(\theenumi)$]

\item $\supp\phi\subset[0,100)$, $\phi(t)\geq 1$ for any $t\in[0,90]$, and $0\leq\phi(t)\leq 100$ for any $t\in[0,100)$.

\item $\phi'(t)\leq 0$ for any $t\in[0,100)$.

\item $1\leq -\phi'(t)\leq 2$ for any $t\in[0,90]$.

\item $|\phi'(t)|\leq 100$ for any $t\in[0,100)$.

\end{enumerate}
\end{defn}

For $x\in\R^3$ and $r>0$, define $\phi_{x,r},\overline{\phi}_{x,r}:\R^3\to\R$ by
\[
\phi_{x,r}(y):=\phi\left(\frac{|y-x|^2}{r^2}\right),
\quad
\overline{\phi}_{x,r}(y):=\phi'\left(\frac{|y-x|^2}{r^2}\right).
\]
By Definition \ref{defnofphi}, for any $y\in\R^3$,
\begin{gather}
-\overline{\phi}_{x,\f{r}{10}}(y)
\leq 100\chi_{B_r(x)}(y)
\leq 100\bigl(-\overline{\phi}_{x,\f{r}{9}}(y)\bigr),
\label{phiproperty0}\\
\phi_{x,\f{r}{10}}(y)
\leq 100\chi_{B_r(x)}(y)
\leq 100\phi_{x,\f{r}{9}}(y),
\label{phiproperty1}
\end{gather}
where $\chi_{B_r(x)}$ denotes the characteristic function of $B_r(x)$.

We also define the $\phi$-weighted density
\begin{equation}
\vt_{\va}(u;x,r):=\frac{1}{r}\int_{B_{10r}(x)}e_{\va}(u)\phi_{x,r},
\label{density11}
\end{equation}
whenever $B_{10r}(x)\subset\subset U$. When $x\in U$ and $r>\dist(x,\pa U)$, the boundary affects the density. We therefore set
\begin{align*}
\theta_{\va}^U(u;x,r)&:=\frac{1}{r}\int_{U_r(x)}e_{\va}(u),\\
\vt_{\va}^U(u;x,r)&:=\frac{1}{r}\int_U e_{\va}(u)\phi_{x,r}.
\end{align*}

With these definitions in place, we establish the interior monotonicity formula.

\begin{prop}[Interior monotonicity]\label{MonotoneInterior}
Let $u\in C^\infty(B_r(x_0),\R^m)$ be a solution of \eqref{EL}. Then the following properties are satisfied.
\begin{enumerate}[label=$(\theenumi)$]

\item For any $0<s<t<r$,
\begin{equation}
\begin{aligned}
&\theta_{\va}(u;x_0,t)-\theta_{\va}(u;x_0,s)\\
&\quad\quad=
\int_s^t\left(
\frac{1}{\rho}\int_{\pa B_\rho(x_0)}
\left|\frac{y-x_0}{\rho}\cdot\na u\right|^2\ud\HH^2(y)+\frac{2}{\va^2\rho^2}\int_{B_\rho(x_0)}f(u)
\right)\ud\rho.
\end{aligned}
\label{Mo11}
\end{equation}

\item For any $0<s<t<\f{r}{10}$,
\begin{equation}
\begin{aligned}
&\vt_{\va}(u;x_0,t)-\vt_{\va}(u;x_0,s)\\
&\quad\quad=
\int_s^t\left(
-\frac{2}{\rho^2}\int_{B_r(x_0)}
\left|\frac{y-x_0}{\rho}\cdot\na u\right|^2\overline{\phi}_{x_0,\rho}
+\frac{2}{\va^2\rho^2}\int_{B_r(x_0)} f(u)\phi_{x_0,\rho}
\right)\ud\rho.
\end{aligned}
\label{Monotone1}
\end{equation}

\end{enumerate}
\end{prop}

\begin{proof}
The proof of \eqref{Mo11} follows by the same argument as \cite[Lemma 2]{MZ10}; we omit the details. We prove \eqref{Monotone1}. Without loss of generality, assume $x_0=0$. Define the stress-energy tensor
$T_{\va}(u)=(T_{\va}^{ij}(u))_{i,j\in\{1,2,3\}}$ by
\begin{equation}
T_{\va}^{ij}(u):=e_{\va}(u)\delta_{ij}-\pa_i u:\pa_j u.
\label{stress}
\end{equation}
Using \eqref{EL}, one verifies that
\begin{equation}
\pa_j T_{\va}^{ij}(u)=0.
\label{stressidentity}
\end{equation}
Write $\phi_\rho:=\phi_{0,\rho}$ and $\overline{\phi}_\rho:=\overline{\phi}_{0,\rho}$. Testing \eqref{stressidentity} against $y\phi_\rho$ and integrating by parts gives
\begin{equation}
\int \frac{2|y|^2}{\rho^2}\overline{\phi}_\rho e_{\va}(u)
-\int \frac{2}{\rho^2}\overline{\phi}_\rho |y\cdot\na u|^2
+\int e_{\va}(u)\phi_\rho
+\int \frac{2}{\va^2}f(u)\phi_\rho=0.
\label{addingone}
\end{equation}
A direct computation yields
\begin{equation}
\begin{aligned}
\frac{\ud}{\ud\rho}\left(\frac{1}{\rho}\int e_{\va}(u)\phi_\rho\right)
&=
-\frac{1}{\rho^2}\int e_{\va}(u)\phi_\rho
-\frac{2}{\rho^4}\int e_{\va}(u)|y|^2\overline{\phi}_\rho\\
&\stackrel{\eqref{addingone}}{=}
-\frac{2}{\rho^4}\int |y\cdot\na u|^2\overline{\phi}_\rho
+\frac{2}{\va^2\rho^2}\int f(u)\phi_\rho.
\end{aligned}
\label{similarcompute}
\end{equation}
Integrating over $\rho\in(s,t)$ completes the proof of \eqref{Monotone1}.
\end{proof}

The next result gives the corresponding monotonicity formula near the boundary.

\begin{prop}[Boundary monotonicity]\label{MonotoneBoundary}
Let $U\subset\R^3$ be a bounded $C^{2,1}$ domain with parameters $M_{U,2}$ and $r_{U,2}$. Let $x_0\in\pa U$ and $r\in(0,r_{U,2})$. Suppose that $g\in C^1(T_{2r}^U(x_0),\cN)$ satisfies
\begin{equation}
r\|\na_{\top}g\|_{L^\infty(T_{2r}^U(x_0))}\leq M_0.
\label{gleqM0}
\end{equation}
Assume that $u\in C^2(\ol{U}_{2r}(x_0),\R^m)$ is a solution of \eqref{EL} in $U_{2r}(x_0)$, with $u=g$ on $T_{2r}^U(x_0)$, and that
\begin{equation}
\theta_{\va}^U(u;x_0,2r)\leq M_1.
\label{thetavaleqM1}
\end{equation}
Set
\[
\al(M_0,M_1,r)
:=
\frac{CM_0}{r}(M_0+M_1^{\f{1}{2}})
+C(M_0^2+M_1),
\]
where $C>0$ depends only on $M_{U,2}$ and $r_{U,2}$. Then the following properties hold.

\begin{enumerate}[label=$(\theenumi)$]

\item For any $0<s<t<r$,
\begin{equation}
\begin{aligned}
&\theta_{\va}^U(u;x_0,t)-\theta_{\va}^U(u;x_0,s)
+\al(M_0,M_1,r)(t-s)\\
&\quad\geq
\int_s^t\left(
\frac{1}{\rho}\int_{U\cap\pa B_\rho(x_0)}
\left|\frac{y-x_0}{\rho}\cdot\na u\right|^2\ud\HH^2(y)
+\frac{2}{\va^2\rho^2}\int_{U_\rho(x_0)}f(u)
\right)\ud\rho.
\end{aligned}
\label{Monotone11b}
\end{equation}

\item For any $0<s<t<\f{r}{10}$,
\begin{equation}
\begin{aligned}
&\vt_{\va}^U(u;x_0,t)-\vt_{\va}^U(u;x_0,s)
+\al(M_0,M_1,r)(t-s)\\
&\quad\geq
\int_s^t\left(
-\frac{2}{\rho^2}\int_U
\left|\frac{y-x_0}{\rho}\cdot\na u\right|^2\overline{\phi}_{x_0,\rho}
+\frac{2}{\va^2\rho^2}\int_U f(u)\phi_{x_0,\rho}
\right)\ud\rho.
\end{aligned}
\label{Monotone11}
\end{equation}

\end{enumerate}
\end{prop}

\begin{proof}
We first prove \eqref{Monotone11}. Set $x_0=0$, and use the notation
$\phi_\rho$ and $\overline{\phi}_\rho$ from the proof of Proposition \ref{MonotoneInterior}. The identity \eqref{stressidentity} holds in $U_{2r}$. Testing against $y\phi_\rho$ for $\rho\in(0,\f{r}{10})$ and integrating by parts yields
\begin{equation}
\begin{aligned}
&\int_U \frac{2|y|^2}{\rho^2}\overline{\phi}_\rho e_{\va}(u)
-\int_U \frac{2}{\rho^2}\overline{\phi}_\rho |y\cdot\na u|^2
+\int_U e_{\va}(u)\phi_\rho
+\int_U \frac{2}{\va^2}f(u)\phi_\rho\\
&\quad=
\int_{\pa U}T_{\va}^{ij}(u)y_i\nu_j\phi_\rho\ud\HH^2,
\end{aligned}
\label{evaboundart}
\end{equation}
where $\nu:\pa U\to\Ss^2$ is the unit outer normal to $\pa U$. The same computation as in \eqref{similarcompute} gives
\begin{equation}
\begin{aligned}
\frac{\ud}{\ud\rho}\vt_{\va}^U(u;0,\rho)
&=
-\frac{2}{\rho^4}\int_U |y\cdot\na u|^2\overline{\phi}_\rho
+\frac{2}{\va^2\rho^2}\int_U f(u)\phi_\rho-\frac{1}{\rho^2}\int_{\pa U}T_{\va}^{ij}(u)y_i\nu_j\phi_\rho\ud\HH^2.
\end{aligned}
\label{dvtva1}
\end{equation}
It remains to estimate the boundary integral from below. Since $g\in\cN$ on $T_{2r}^U$, we have $f(u)=0$ there. Decomposing $y$ into normal and tangential parts along $\pa U$ with $ \tau(y):=y-(y\cdot\nu)\nu $, we obtain
\begin{equation}
\begin{aligned}
-\int_{\pa U}T_{\va}^{ij}(u)y_i\nu_j\phi_\rho\ud\HH^2
&=
\int_{\pa U}(y\cdot\na u)\cdot(\pa_\nu u)\phi_\rho\ud\HH^2
-\int_{\pa U}e_{\va}(u)(y\cdot\nu)\phi_\rho\ud\HH^2\\
&=
\int_{\pa U}(\tau\cdot\na u)\cdot(\pa_\nu u)\phi_\rho\ud\HH^2
+\frac{1}{2}\int_{\pa U}|\pa_\nu u|^2(y\cdot\nu)\phi_\rho\ud\HH^2\\
&\quad
-\frac{1}{2}\int_{\pa U}(|\na u|^2-|\pa_\nu u|^2)(y\cdot\nu)\phi_\rho\ud\HH^2.
\end{aligned}
\label{Tijidentity}
\end{equation}
Since $U$ is $C^{2,1}$ with parameters $M_{U,2}$ and $r_{U,2}$, the boundary representation in Definition \ref{RegularityBoundary} gives
\begin{equation}
|y\cdot\nu|\leq C\rho^2
\quad\text{on } \pa U\cap B_{10\rho},
\label{ycdotnu}
\end{equation}
where $C=C(M_{U,2},r_{U,2})>0$. Applying \eqref{gleqM0}, \eqref{Tijidentity}, \eqref{ycdotnu}, the bound $|\tau|\leq C\rho$, and the Cauchy--Schwarz inequality, we find
\begin{equation}
\begin{aligned}
-\frac{1}{\rho^2}\int_{\pa U}T_{\va}^{ij}(u)y_i\nu_j\phi_\rho\ud\HH^2
&\geq
-\frac{1}{\rho^2}
\left(\int_{\pa U}|\pa_\nu u|^2\phi_\rho\ud\HH^2\right)^{\f{1}{2}}
\left(\int_{\pa U}\rho^2|\na_{\top}g|^2\phi_\rho\ud\HH^2\right)^{\f{1}{2}}\\
&\quad
-C\int_{\pa U}|\pa_\nu u|^2\phi_\rho\ud\HH^2
-\frac{CM_0^2}{r}\\
&\geq
-\frac{CM_0}{r}
\left(\int_{\pa U}|\pa_\nu u|^2\phi_{\f{r}{9}}\ud\HH^2\right)^{\f{1}{2}}-C\int_{\pa U}|\pa_\nu u|^2\phi_{\f{r}{9}}\ud\HH^2
-\frac{CM_0^2}{r}.
\end{aligned}
\label{Tijgeq}
\end{equation}
To bound the remaining normal derivative terms, choose a vector field
$X\in C^1(\ol U,\R^3)$ that
\[
X=\nu\text{ on }\pa U
\quad\text{and}\quad
\|\na X\|_{L^\infty(U)}\leq C(M_{U,2},r_{U,2}).
\]
Testing \eqref{stressidentity} against $X\phi_{\f{r}{9}}$ and using
\eqref{phiproperty0}--\eqref{phiproperty1}, we get
\begin{equation}
\begin{aligned}
\int_{\pa U}|\pa_\nu u|^2\phi_{\f{r}{9}}\ud\HH^2
&\leq C\int_{\pa U}|\na_{\top}g|^2\phi_{\f{r}{9}}\ud\HH^2
+C\int_U e_{\va}(u)\phi_{\f{r}{9}}+\frac{C}{r}\int_U e_{\va}(u)\chi_{B_{2r}}\\
&\leq CM_0^2+C\theta_{\va}^U(u;0,2r)
\leq C(M_0^2+M_1).
\end{aligned}
\label{similarusefor}
\end{equation}
Substituting \eqref{similarusefor} into \eqref{Tijgeq} yields
\[
-\frac{1}{\rho^2}\int_{\pa U}T_{\va}^{ij}(u)y_i\nu_j\phi_\rho\ud\HH^2
\geq
-\frac{CM_0}{r}(M_0+M_1^{\f{1}{2}})
-C(M_0^2+M_1)
=
-\al(M_0,M_1,r).
\]
Inserting this bound into \eqref{dvtva1} and integrating over $\rho\in(s,t)$ proves \eqref{Monotone11}.

The proof of \eqref{Monotone11b} is the same, except that one tests \eqref{stressidentity} against the vector field $y$ instead of $y\phi_\rho$. The corresponding boundary estimate is obtained as above, and this gives \eqref{Monotone11b}.
\end{proof}

We close this subsection by proving a logarithmic energy bound for minimizers 
$\{u_{\va}\}_{\va\in(0,1)}$. This estimate will be used repeatedly to control the
rescaled energy densities.

\begin{prop}\label{locallogestimate}
Let $\om,\cA,g_{\va}$, and $u_{\va}$ be as in Theorem
\ref{globalminimizersproperties}. Then there exists $C>0$, depending only on
$\cA,f,M_0,\cN,\om$, and $\rho_1$, such that the following properties hold.
\begin{enumerate}[label=$(\theenumi)$]

\item For any $\va\in(0,1)$,
\[
E_{\va}(u_{\va},\om)\leq C(|\log\va|+1),
\quad
\|u_{\va}\|_{L^\infty(\om)}\leq C.
\]

\item For any $\va\in(0,1)$, any $x\in\ol{\om}$, and any $r>0$,
\[
\theta_{\va}^{\om}(u_{\va};x,r)\leq C(|\log\va|+1).
\]

\end{enumerate}
\end{prop}

\begin{proof}
Since $\om$ is strongly convex, it is strictly star-shaped. Thus, there exist
$x_0\in\om$ and $c_1(\om)>0$ such that $ (x-x_0)\cdot\nu(x)\geq c_1(\om) $ for any $ x\in\pa\om $, where $\nu$ denotes the unit outer normal to $\pa\om$. Using an argument analogous to that of \cite[Lemma~25]{Can17},
\be
E_{\va}(u_{\va},\om)
\leq
C\diam(\om)E_{\va}(g_{\va},\pa\om)
\leq
C(|\log\va|+1),
\label{Euvalogleq}
\ee
where $C=C(\cA,f,M_0,\cN,\om,\rho_1)>0$. Applying
\cite[Lemma~8.3]{Lam14} together with Assumption \ref{A3}, we obtain
$\|u_{\va}\|_{L^\infty(\om)}\leq C$. 

It remains to prove the density bound. The densities $\vt_{\va}^{\om}$ and
$\theta_{\va}^{\om}$ are comparable, up to universal constants and a harmless
change of radius, by \eqref{phiproperty1}. We therefore work with
$\vt_{\va}^{\om}$.

We first consider points away from the singular boundary set $\cA$. On compact
subsets of $\ol{\om}\backslash\cA$, the boundary data $g_{\va}$ take values in
$\cN$ and satisfy the uniform tangential gradient bounds required by Proposition
\ref{MonotoneBoundary}. Combining the global energy estimate \eqref{Euvalogleq}
with Propositions \ref{MonotoneInterior} and \ref{MonotoneBoundary}, and then
using a finite covering argument, we obtain
\begin{equation}
\vt_{\va}^{\om}(u_{\va};x,\rho)
\leq C(|\log\va|+1)
\label{x0notincA}
\end{equation}
whenever $x$ stays in a compact fixed subset of $\ol{\om}\backslash\cA$ and
$\rho$ is smaller than the corresponding covering radius.

It remains to treat neighborhoods of the points in $\cA$. Fix $x_0\in\cA$ and
choose $r_0>0$ so that
\[
B_{10r_0}(x_0)\cap\cA=\{x_0\}.
\]
After translation and rotation, we may assume that $x_0=0$ and that the tangent
plane to $\pa\om$ at $0$ is horizontal. Since $g_{\va}$ need not take values in
$\cN$ near $x_0$, Proposition \ref{MonotoneBoundary} cannot be applied directly.
Nevertheless, using the notation of its proof and repeating the calculations
leading to \eqref{dvtva1} and \eqref{Tijidentity}, we obtain, for any
$\rho\in(0,\f{r_0}{2})$,
\begin{equation}
\begin{aligned}
\frac{\ud}{\ud\rho}\vt_{\va}^{\om}(u_{\va};0,\rho)
&\geq
\frac{1}{\rho^2}\int_{\pa\om}
(\tau\cdot\na u_{\va})\cdot(\pa_{\nu}u_{\va})
\phi_{\rho}\ud\HH^2+\frac{1}{2\rho^2}\int_{\pa\om}
|\pa_{\nu}u_{\va}|^2(y\cdot\nu)\phi_{\rho}\ud\HH^2 \\
&\quad-\frac{1}{2\rho^2}\int_{\pa\om}
(|\na u_{\va}|^2-|\pa_{\nu}u_{\va}|^2)
(y\cdot\nu)\phi_{\rho}\ud\HH^2
-\frac{1}{\va^2\rho^2}\int_{\pa\om}
f(g_{\va})(y\cdot\nu)\phi_{\rho}\ud\HH^2,
\end{aligned}
\label{dvtva_gva}
\end{equation}
where $\tau(y):=y-(y\cdot\nu)\nu$ is the tangential component of $y$ along
$\pa\om$.

Using Definition \ref{defnsuitabledata}, the geometric estimate
\[
|y\cdot\nu|\leq C\rho^2
\quad\text{on }\pa\om\cap B_{10\rho},
\]
and the identity
$|\na u_{\va}|^2-|\pa_{\nu}u_{\va}|^2=|\na_{\top}g_{\va}|^2$ on $\pa\om$, we
deduce from \eqref{dvtva_gva} that
\begin{equation}
\begin{aligned}
\frac{\ud}{\ud\rho}\vt_{\va}^{\om}(u_{\va};0,\rho)
&\geq
\frac{1}{\rho^2}\int_{\pa\om}
(\tau\cdot\na u_{\va})\cdot(\pa_{\nu}u_{\va})
\phi_{\rho}\ud\HH^2 \\
&\quad
-C\int_{\pa\om}|\pa_{\nu}u_{\va}|^2\phi_{\rho}\ud\HH^2
-\int_{\pa\om}|\na_{\top}g_{\va}|^2\phi_{\rho}\ud\HH^2
-\frac{C}{\rho}.
\end{aligned}
\label{dvtva_est1}
\end{equation}
Since $\phi_{\rho}\leq C\phi_{\f{r_0}{9}}$ for $\rho\in(0,\f{r_0}{10})$, and since
$E_{\va}(g_{\va},\pa\om)\leq C(|\log\va|+1)$ by Definition
\ref{defnsuitabledata}, an estimate analogous to \eqref{similarusefor}, with
$\phi_{\f{r_0}{9}}$ in place of $\phi_{\f{r}{9}}$, gives
\[
\int_{\pa\om}|\pa_{\nu}u_{\va}|^2\phi_{\f{r_0}{9}}\ud\HH^2
\leq
C\bigl(E_{\va}(g_{\va},\pa\om)
+\theta_{\va}^{\om}(u_{\va};0,2r_0)\bigr)
\leq
C(|\log\va|+1).
\]
Substituting this estimate into \eqref{dvtva_est1}, we obtain
\begin{equation}
\frac{\ud}{\ud\rho}\vt_{\va}^{\om}(u_{\va};0,\rho)
\geq
\frac{1}{\rho^2}\int_{\pa\om}
(\tau\cdot\na u_{\va})\cdot(\pa_{\nu}u_{\va})
\phi_{\rho}\ud\HH^2
-C(|\log\va|+1)-\frac{C}{\rho}.
\label{combinede}
\end{equation}

We now estimate the cross term. By the fourth property in Definition
\ref{defnsuitabledata} and boundary parameterization
\cite[Formulae~(II.28)--(II.29)]{LR99},
\[
|\tau(y)\cdot\na u_{\va}|
\leq
\left|\left(y_i\frac{\pa g_{\va}^0}{\pa y_i}\right)\circ\Phi_0\right|
+C\rho^2|\na_{\top}g_{\va}|
\leq
\frac{C\rho}{\va}\chi_{\cA_{\va}}
+C\rho^2|\na_{\top}g_{\va}|,
\]
where
\[
\cA_{\va}:=\{y\in\pa\om:\dist(y,\cA)<\va\}.
\]
Using this estimate, the gradient bounds in Lemmas \ref{Apriori} and
\ref{AprioriBoundary}, and the preceding normal derivative estimate, we find
that the part supported in $\cA_{\va}$ is bounded below by $-C\rho^{-1}$, while the
remaining part is bounded below by $-C(|\log\va|+1)$. Hence
\[
\frac{\ud}{\ud\rho}
\bigl(\vt_{\va}^{\om}(u_{\va};0,\rho)+C\log\rho\bigr)
\geq
-C(|\log\va|+1)
\quad\text{for any }\rho\in\(\va,\f{r_0}{10}\).
\]
Integrating from $\rho$ to $\f{r_0}{10}$ and using \eqref{Euvalogleq}, we obtain
\[
\vt_{\va}^{\om}(u_{\va};0,\rho)
\leq
C(|\log\va|+1)
\quad\text{for any }\rho\in\(\va,\f{r_0}{10}\).
\]
For $\rho\in(0,\va)$, the same bound follows from the local gradient estimates
in Lemmas \ref{Apriori} and \ref{AprioriBoundary}, together with the uniform
$L^\infty$ bound for $u_{\va}$. Combining this estimate near each point of
$\cA$ with \eqref{x0notincA} and the global energy bound \eqref{Euvalogleq}
proves the asserted density estimate.
\end{proof}

\subsection{Lower bound}

The next lemma gives the two-dimensional lower bound used to detect the energy signature of a topological defect line on a small disk orthogonal to the defect.

\begin{lem}\label{LowerBound}
Let $\va\in(0,\f{r}{2})$ and $u\in H^1(B_r^2,\R^m)$. Assume that
$g=u|_{\pa B_r^2}\in H^1(\pa B_r^2,\R^m)$ and that
$\dist(g,\cN)<\delta_{\cN}$ on $\pa B_r^2$. Then
$\sg:=[\Pi_{\cN}\circ g]_{\cN}$ is well-defined and
\[
E_{\va}(u,B_r^2)+CrE_{\va}(g,\pa B_r^2)
\geq
|\sg|_*\log\f{r}{\va}-C,
\]
where $C>0$ depends only on $f$ and $\cN$.
\end{lem}

\begin{proof}
After rescaling $B_r^2$ to $B_1^2$ and replacing $\va$ by $\f{\va}{r}$, the result
follows from \cite[Corollary~6.8]{MRS21} and Remark \ref{remequvalence}.
\end{proof}

\subsection{Upper bound}

We also need a matching upper construction on two-dimensional disks. The next
lemma provides this competitor with the optimal logarithmic leading order.

\begin{lem}\label{upperboundleast}
Let $r>0$. There exists a constant $C>0$, depending only on $f$ and $\cN$, such
that the following holds. For any $\va\in(0,r)$ and any
$g\in H^1(\pa B_r^2,\cN)$, there exists
$u_{\va}\in H^1(B_r^2,\R^m)$ with $u_{\va}|_{\pa B_r^2}=g$ and
\be
E_{\va}(u_{\va},B_r^2)
\leq
|\sg|_*\log\f{r}{\va}
+C\bigl(r\|\na_{\top}g\|_{L^2(\pa B_r^2)}^2+1\bigr),
\label{upperboundleast1}
\ee
where $\sg:=[g]_{\cN}\in[\Ss^1,\cN]$.
\end{lem}

\begin{proof}
This follows from \cite[Lemma~2.12]{GFW26}.
\end{proof}

The following notation will be used for cylindrical neighborhoods of line
segments. Given positive numbers $L$ and $r$, set
\be
\Lda_{r,L}:=B_r^2\times(-L,L),
\quad
\Ga_{r,L}:=\pa B_r^2\times(-L,L).
\label{LdarLGarL}
\ee
For $u\in H^1(\Ga_{r,L},\cN)$, the analysis in \cite[Section~2.2]{Can17}
implies that, for a.e. $t\in(-L,L)$, the slice $u(\cdot,t)$ belongs to
$C^0(\pa B_r^2,\cN)$, and the homotopy class of this slice is independent of
$t$ on a full-measure set. We denote this common class by $[u]_{\cN}$.

The next lemma extends the disk construction in Lemma \ref{upperboundleast} to
cylinders.

\begin{lem}\label{cylinderex}
There exists a constant $C>0$, depending only on $f$ and $\cN$, such that the
following holds. For any $\va\in(0,r)$ and any
$g\in H^1(\Ga_{r,L},\cN)$ with $\sg=[g]_{\cN}\in[\Ss^1,\cN]$, there exists
$u_{\va}\in H^1(\Lda_{r,L},\R^m)$ that satisfy $u_{\va}=g$ on $\Ga_{r,L}$ such
that
\be
E_{\va}(u_{\va},\Lda_{r,L})
\leq
CL\left(\f{L}{r}+\f{r}{L}\right)
\|\na_{\top}g\|_{L^2(\Ga_{r,L})}^2
+2|\sg|_*L\log\f{r}{\va}
+CL,
\label{cyl:bulk}
\ee
and, for each $z\in\{\pm L\}$,
\be
E_{\va}(u_{\va},B_r^2\times\{z\})
\leq
C\left(\f{L}{r}+\f{r}{L}\right)
\|\na_{\top}g\|_{L^2(\Ga_{r,L})}^2
+|\sg|_*\log\f{r}{\va}
+C.
\label{cyl:slice}
\ee
\end{lem}

\begin{proof}
The proof follows from the argument of \cite[Lemma~29]{Can17}. We give the main
steps for completeness. By scaling, it suffices to consider the case $r=1$.
By Fubini's theorem, there exists $L_0\in[-\f{L}{4},\f{L}{4}]$ such that
\be
\|\na_{\top}g\|_{L^2(\pa B_1^2\times\{L_0\})}^2
\leq
8L^{-1}\|\na_{\top}g\|_{L^2(\Ga_{1,L})}^2
\label{PPL8}
\ee
and $[g(\cdot,L_0)]_{\cN}=[g]_{\cN}$.

For notational simplicity, we present the construction when $L_0=0$. The
general case is obtained by replacing $0$ with $L_0$; since
$L_0\in[-\f{L}{4},\f{L}{4}]$, the constants remain unchanged. We use polar coordinates
$(\rho,\theta,z)\in[0,1]\times\Ss^1\times[-L,L]$ in $\Lda_{1,L}$ and define $ \wt{u}:(\ol{B}_1^2\backslash B_{\f{1}{2}}^2)\times[-L,L]\to\cN $ by
\[
\wt{u}(\rho,\theta,z)=
\begin{cases}
g(1,\theta,z'(\rho,z)),
&\rho\in[\rho_0(z),1],\ |z|\leq L,\\
g(1,\theta,0),
&\rho\in[\tfrac{1}{2},\rho_0(z)),\ |z|\leq L,
\end{cases}
\]
where
\[
z'(\rho,z):=2L\sgn(z)(\rho-1)+z,
\quad
\rho_0(z):=1-\frac{|z|}{2L}.
\]
A direct computation gives
\begin{align*}
\|\na\wt{u}\|_{L^2((B_1^2\backslash B_{\f{1}{2}}^2)\times[0,L])}^2&\leq
(4L^2+1)\int_{\Ss^1}\int_0^L
\bigl(|\pa_zg|^2+|\pa_{\theta}g|^2\bigr)(1,\theta,\xi)
\ud\xi\ud\HH^1(\theta) \\
&\quad
+(\log 2)L\int_{\Ss^1}
|\pa_{\theta}g|^2(1,\theta,0)\ud\HH^1(\theta).
\end{align*}
Combining this with the analogous estimate on $z\in[-L,0]$ and using
\eqref{PPL8}, we obtain
\be
\|\na\wt{u}\|_{L^2((B_1^2\backslash B_{\frac{1}{2}}^2)\times[-L,L])}^2
\leq
CL(L+L^{-1})\|\na_{\top}g\|_{L^2(\Ga_{1,L})}^2.
\label{grad:annulus}
\ee

We now fill the inner cylinder. Applying Lemma \ref{upperboundleast} to the
boundary datum $\theta\mapsto g(1,\theta,0)$ on $\pa B_{\frac{1}{2}}^2$, we find
$\wt{u}_{\va}\in H^1(B_{\frac{1}{2}}^2,\R^m)$ such that
$\wt{u}_{\va}(\frac{1}{2},\theta)=g(1,\theta,0)$ and
\be
E_{\va}(\wt{u}_{\va},B_{\frac{1}{2}}^2)
\leq
\frac{C}{L}\|\na_{\top}g\|_{L^2(\Ga_{1,L})}^2
+|\sg|_*|\log\va|
+C,
\label{inner:energy}
\ee
where $\sg=[g]_{\cN}$. Define
\[
u_{\va}(\rho,\theta,z)=
\begin{cases}
\wt{u}(\rho,\theta,z),
&\rho\in(\tfrac{1}{2},1],\\
\wt{u}_{\va}(\rho,\theta),
&\rho\in(0,\tfrac{1}{2}].
\end{cases}
\]
Then $u_{\va}=g$ on $\Ga_{1,L}$. The estimates \eqref{cyl:bulk} and
\eqref{cyl:slice} follow from \eqref{grad:annulus}, \eqref{inner:energy}, and
the scaling back to radius $r$.
\end{proof}

\subsection{Compactness of minimizers with bounded energy}

We finish this section with a compactness statement for local minimizers whose
energies remain uniformly bounded as $\va\to0^+$.

\begin{prop}\label{interiorcompactnesslem}
Let $U\subset\R^3$ be a bounded domain. Assume that
$\{u_{\va}\}_{\va\in(0,1)}$ is a family of local minimizers of
\eqref{GLfunctional} such that
\be
E_{\va}(u_{\va},U)+\|u_{\va}\|_{L^{\ift}(U)}\leq M.
\label{assumptionbounduniform}
\ee
Then there exist a sequence $\va_i\to0^+$ and a map
$u_*\in H^1(U,\cN)$, which is a local minimizer of the Dirichlet energy
\eqref{Dirichlet}, such that the following properties hold.
\begin{enumerate}[label=$(\theenumi)$]

\item $u_{\va_i}\to u_*$ strongly in $H_{\loc}^1(U,\R^m)$.

\item The singular set of $u_*$, defined by
\[
\sing(u_*):=
\{x\in U:u_*\text{ is not continuous in }B_r(x)
\text{ for any }r>0\},
\]
is locally discrete in $U$, and
$u_{\va_i}\to u_*$ in $C^0(U\backslash\sing(u_*),\R^m)$.

\end{enumerate}
\end{prop}

\begin{proof}
The uniform bound \eqref{assumptionbounduniform} implies that there exist a
subsequence $\va_i\to0^+$ and a map $u_*\in H^1(U,\R^m)$ such that
$u_{\va_i}\wc u_*$ weakly in $H^1(U,\R^m)$. After passing to a further
subsequence, we may also assume that $u_{\va_i}\to u_*$ a.e. in $U$. Since
$f\geq0$, Fatou's lemma gives
\[
0\leq\int_U f(u_*)\leq\liminf_{i\to+\ift}\int_U f(u_{\va_i})
\leq\liminf_{i\to+\ift}\va_i^2E_{\va_i}(u_{\va_i},U)
=0.
\]
Hence $u_*\in H^1(U,\cN)$. It follows from \cite[Proposition~B.1]{CL22} that
$u_*$ is a local minimizer of \eqref{Dirichlet} and that
$u_{\va_i}\to u_*$ strongly in $H_{\loc}^1(U,\R^m)$. The local discreteness of
$\sing(u_*)$ follows from \cite[Theorem~II]{SU82} and
\cite[Theorem~1.5]{NV17}. Finally, the convergence
$u_{\va_i}\to u_*$ in $C^0(U\backslash\sing(u_*),\R^m)$ follows from
\cite[Theorem~1.2]{CL22}.
\end{proof}

\section{Luckhaus-type results}\label{SectionLuc}

\subsection{Luckhaus-type lemmas on the ball}

The following lemma is a Ginzburg--Landau analog of the classical 
interpolation construction of Luckhaus \cite{Luc88}. 

\begin{lem}\label{Luckhauslemma}
Let $ M>0 $ and let $\{u_{\ol{\va}}\}_{\ol{\va}\in(0,1)}\subset H^1(\pa B_1,\R^m)$ be a family of maps satisfying $ \|u_{\ol{\va}}\|_{L^{\ift}(\pa B_1)}\leq M $. Then there exist $\delta\in(0,1)$ and $C>0$, depending only on $f,M $, and 
$\cN$, such that the following properties hold. For any $\ol{\va}\in(0,\delta)$, if
\begin{align}
E_{\ol{\va}}(u_{\ol{\va}},\pa B_1)\leq\delta|\log\ol{\va}|,\label{log1}
\end{align}
then there exist $v_{\ol{\va}}\in H^1(\pa B_1,\cN)$ and 
$w_{\ol{\va}}\in H^1(B_1\backslash B_{1-h(\ol{\va})},\R^m)$ such that
\begin{gather}
w_{\ol{\va}}(x)=u_{\ol{\va}}(x)\text{ and }
w_{\ol{\va}}((1-h(\ol{\va}))x)=v_{\ol{\va}}(x)
\text{ for }\HH^2\text{-a.e. }x\in\pa B_1,\label{interpoL}\\
\f{1}{2}\int_{\pa B_1}|\na_{\top}v_{\ol{\va}}|^2\ud\HH^2
\leq CE_{\ol{\va}}(u_{\ol{\va}},\pa B_1),\label{VboundL}\\
E_{\ol{\va}}(w_{\ol{\va}},B_1\backslash B_{1-h(\ol{\va})})
\leq Ch(\ol{\va})E_{\ol{\va}}(u_{\ol{\va}},\pa B_1),\label{WboundL}\\
\|u_{\ol{\va}}-v_{\ol{\va}}\|_{L^2(\pa B_1)}
\leq Ch^{\f{1}{2}}(\ol{\va})(E_{\ol{\va}}(u_{\ol{\va}},\pa B_1))^{\f{1}{2}},
\label{UolvaVolvami}
\end{gather}
where $h(\ol{\va}):=\ol{\va}^{\f{1}{2}}|\log\ol{\va}|$.
\end{lem}

\begin{proof}
The proof follows essentially the same argument as 
\cite[Proposition~33 and Corollary~34]{Can17}; we sketch the key steps and omit the details. Set $h:=h(\ol{\va})$ for brevity. Following the construction in \cite[Lemma~1]{Luc88}, we establish a partition of $\pa B_1$ 
as
\[
\pa B_1=\bigsqcup_{j=0}^{2}\cR_j,\quad\cR_j=\bigsqcup_{i=1}^{k_j}c_{ij},
\]
where $\bigsqcup$ denotes disjoint union. Moreover, there exists an absolute constant $C_0>0$ such that the following properties hold.
\begin{itemize}
\item For any $j\in\{1,2\}$, the closure $\ol{c}_{ij}$ of each $j$-cell is bi-Lipschitz equivalent to $\ol{B}_h^j$: there exists a bi-Lipschitz homeomorphism $\phi_{ij}:\ol{c}_{ij}\to\ol{B}_h^j$ whose restriction to $c_{ij}$ satisfies
\[
\|\na\phi_{ij}\|_{L^{\ift}(c_{ij})}
+\|\na(\phi_{ij}^{-1})\|_{L^{\ift}(B_h^j)}\leq C_0.
\]
\item $ \#\{\ell\in\Z\cap[1,k_2]:c_{i1}\subset\pa c_{\ell,2}\}\leq C_0 $ for any $i\in\Z\cap[1,k_1]$.
\item By Fubini's theorem,
\[
\begin{aligned}
E_{\ol{\va}}(u_{\ol{\va}},\cR_1)
&\leq C_0h^{-1}E_{\ol{\va}}(u_{\ol{\va}},\pa B_1),\\
\int_{\cR_1}f(u_{\ol{\va}})\ud\HH^1
&\leq C_0h^{-1}\int_{\pa B_1}f(u_{\ol{\va}})\ud\HH^2.
\end{aligned}
\]
\end{itemize}
This partition of $\pa B_1$ induces a partition of $B_1\backslash B_{1-h}$ by
\be
B_1\backslash B_{1-h}=\bigsqcup_{j=0}^{2}\widehat{\cR}_j,\quad
\widehat{\cR}_j=\bigsqcup_{i=1}^{k_j}\wh{c}_{ij},\quad
\wh{c}_{ij}=\left\{x\in B_1\backslash B_{1-h}:\frac{x}{|x|}\in c_{ij}
\right\}.\label{induceddecomposition}
\ee
Using the above partition and the choice of $h$, for $\delta_f>0$ given by Lemma~\ref{Nfproperty}, there is $\delta\in(0,1)$ such that
\[
\sup_{\ol{\va}\in(0,\delta)}\sup_{x\in\cR_1}\dist(u_{\ol{\va}}(x),\cN)
<\delta_f.
\]
We set $v_{\ol{\va}}:=\Pi_{\cN}\circ u_{\ol{\va}}$ on $\cR_1$, so that for 
any $x\in\cR_1$,
\[
v_{\ol{\va}}(x)=\Pi_{\cN}(u_{\ol{\va}}(x)),\quad
|v_{\ol{\va}}(x)-u_{\ol{\va}}(x)|<\delta_f.
\]
Applying Lemma \ref{LowerBound} and \eqref{c0geq}, and choosing $\delta\in(0,1)$ in \eqref{log1} sufficiently small, we find that for any $2$-cell $c_{i2}$ with $i\in[1,k_2]$, the restriction $v_{\ol{\va}}|_{\pa c_{i2}}$ is null-homotopic. We then apply the Lemma~\ref{ExtensionLemma1} to extend $v_{\ol{\va}}|_{\pa c_{i2}}$ to each $c_{i2}$. This completes the construction of $v_{\ol{\va}}$ on $\pa B_1$.

Given the induced decomposition \eqref{induceddecomposition}, we define $w_{\ol{\va}}$ as follows.
\begin{itemize}
\item On $\pa B_1\cup\pa B_{1-h}$, set $w_{\ol{\va}}=u_{\ol{\va}}$ and 
$w_{\ol{\va}}((1-h)\cdot)=v_{\ol{\va}}$, respectively.
\item For $x\in\wh{\cR}_1$,
\[
w_{\ol{\va}}(x):=\(1-\f{1-|x|}{h}\)u_{\ol{\va}}\(\f{x}{|x|}\)
+\f{1-|x|}{h}v_{\ol{\va}}\(\f{x}{|x|}\).
\]
\item For each $3$-cell $\wh{c}_{i2}\subset\wh{\cR}_2$ with $i\in[1,k_2]$, 
extend $w_{\ol{\va}}$ homogeneously from $\pa\wh{c}_{i2}$ by setting
\[
w_{\ol{\va}}(x):=w_{\ol{\va}}\circ(\Phi_{i2})^{-1}
\circ\(\f{h\Phi_{i2}(x)}{|\Phi_{i2}(x)|}\),\quad x\in \wh{c}_{i2}\backslash\{\Phi_{i2}^{-1}(0)\},
\]
where $\Phi_{i2}:\ol{\wh{c}}_{i2}\to\ol{B}_h^3$ is a bi-Lipschitz homeomorphism whose restriction to $\wh{c}_{i2}$ satisfies
\[
\|\na\Phi_{i2}\|_{L^{\ift}(\wh{c}_{i2})}
+\|\na(\Phi_{i2}^{-1})\|_{L^{\ift}(B_h^3)}\leq C
\]
for an absolute constant $C>0$.
\end{itemize}
Given the construction of $v_{\ol{\va}}$ and $w_{\ol{\va}}$, the properties \eqref{interpoL}--\eqref{UolvaVolvami} follow by direct computations.
\end{proof}

The following is the original interpolation lemma of Luckhaus, which we state as a reference.

\begin{lem}[\cite{Luc88}, Lemma~1]\label{Luckhaus1}
For any $\beta\in(\f{1}{2},1)$, there exists $C>0$, depending only on 
$\beta$, such that the following holds. For any $\lda\in(0,\f{1}{2})$, 
$\sg\in(0,1)$, and $u,v\in H^1(\pa B_1,\cN)$, there exists 
$w\in H^1(B_1\backslash B_{1-\lda},\R^m)$ satisfying
\begin{gather*}
w(x)=u(x)\text{ and }w((1-\lda)x)=v(x)\quad
\text{for }\HH^2\text{-a.e. }x\in\pa B_1,\\
\sup_{x\in B_1\backslash B_{1-\lda}}\dist(w(x),\cN)
\leq C\sg^{1-\beta}\lda^{-\f{1}{2}}K^{\f{1}{2}},\\
\int_{B_1\backslash B_{1-\lda}}|\na w|^2\ud x
\leq C\lda(1+\sg^2\lda^{-2})K,
\end{gather*}
where
\[
K:=\int_{\pa B_1}\(|\na_{\top}u|^2+|\na_{\top}v|^2
+\f{|u-v|^2}{\sg^2}\)\ud\HH^2.
\]
\end{lem}

The next proposition combines Lemma \ref{Luckhauslemma} and 
Lemma \ref{Luckhaus1} to construct an interpolation map on a thin shell that connects a general map $u_{\ol{\va}}$ to an $\cN$-valued map $v_{\ol{\va}}$, with an energy bound uniform in $\ol{\va}$.

\begin{prop}\label{Luckhaus11}
Let $\{u_{\ol{\va}}\}_{\ol{\va}\in(0,1)}\subset H^1(\pa B_1,\R^m)$ and $\{v_{\ol{\va}}\}_{\ol{\va}\in(0,1)}\subset H^1(\pa B_1,\cN)$ be families of maps. Assume that there exists $M>0$ such that
\begin{gather}
\|u_{\ol{\va}}\|_{L^{\ift}(\pa B_1)}\leq M,\label{Uolvabound}\\
\int_{\pa B_1}\(|\na_{\top}u_{\ol{\va}}|^2+\f{1}{\ol{\va}^2}f(u_{\ol{\va}})
+|\na_{\top}v_{\ol{\va}}|^2
+\f{|u_{\ol{\va}}-v_{\ol{\va}}|^2}{\sg_{\ol{\va}}^2}\)\ud\HH^2
\leq M\label{BoundC1Uva}
\end{gather}
for any $\ol{\va}\in(0,1)$, where $\sg_{\ol{\va}}\to 0^+$ as 
$\ol{\va}\to 0^+$. Define
\[
\nu_{\ol{\va}}:=h(\ol{\va})+((h(\ol{\va}))^{\f{1}{2}}+\sg_{\ol{\va}})^{\f{1}{4}}(1-h(\ol{\va})),
\]
where $h(\ol{\va})$ is the same as in Lemma~\ref{Luckhauslemma}.

Then there exist $\delta\in(0,1)$, depending only on $f,M$, and $\cN$, and a family of interpolation maps $ \{\vp_{\ol{\va}}\}_{\ol{\va}\in(0,\delta)}\subset 
H^1(B_1\backslash B_{1-\nu_{\ol{\va}}},\R^m)$ such that
\begin{gather*}
\vp_{\ol{\va}}(x)=u_{\ol{\va}}(x)\text{ and }
\vp_{\ol{\va}}((1-h(\ol{\va}))x)=v_{\ol{\va}}(x)
\quad\text{for }\HH^2\text{-a.e. }x\in\pa B_1,\\
E_{\ol{\va}}(\vp_{\ol{\va}},B_1\backslash B_{1-\nu_{\ol{\va}}})\leq C\nu_{\ol{\va}},
\end{gather*}
where $C>0$ depends only on $f,M$, and $\cN$.
\end{prop}

\begin{proof}
Set $h:=h(\ol{\va})$ for brevity. For $\eta\in(0,1)$, choose $\delta=\delta(\eta,M)\in(0,1)$ such that $\eta|\log\delta|\geq M$. The bound \eqref{BoundC1Uva} gives 
$E_{\ol{\va}}(u_{\ol{\va}},\pa B_1)\leq \eta|\log\ol{\va}|$ for any 
$\ol{\va}\in(0,\delta)$. In view of \eqref{Uolvabound}, we apply 
Lemma~\ref{Luckhauslemma}, choosing a sufficiently small $\eta=\eta(f,M,\cN)\in(0,1)$, to obtain 
$u_{\ol{\va},1}\in H^1(\pa B_1,\cN)$ and 
$\vp_{\ol{\va},1}\in H^1(B_1\backslash B_{1-h},\R^m)$ such that
\begin{gather*}
\vp_{\ol{\va},1}(x)=u_{\ol{\va}}(x)\text{ and }
\vp_{\ol{\va},1}((1-h)x)=u_{\ol{\va},1}(x)
\quad\text{for }\HH^2\text{-a.e. }x\in\pa B_1,\\
\int_{\pa B_1}|\na_{\top}u_{\ol{\va},1}|^2\ud\HH^2
\leq CE_{\ol{\va}}(u_{\ol{\va}},\pa B_1)\leq C,\\
E_{\ol{\va}}(\vp_{\ol{\va},1},B_1\backslash B_{1-h})
\leq ChE_{\ol{\va}}(u_{\ol{\va}},\pa B_1)\leq Ch.
\end{gather*}
Moreover, \eqref{UolvaVolvami} gives 
$\|u_{\ol{\va},1}-v_{\ol{\va}}\|_{L^2(\pa B_1)}\leq C\sg_{\ol{\va},1}$, 
where $\sg_{\ol{\va},1}:=h^{\f{1}{2}}(\ol{\va})+\sg_{\ol{\va}}$. Together 
with \eqref{BoundC1Uva}, this yields
\[
\int_{\pa B_1}\(|\na_{\top}u_{\ol{\va},1}|^2+|\na_{\top}v_{\ol{\va}}|^2
+\f{|u_{\ol{\va},1}-v_{\ol{\va}}|^2}{\sg_{\ol{\va},1}^2}\)\ud\HH^2\leq C.
\]
We apply Lemma~\ref{Luckhaus1} with $\beta=\f{3}{4}$, $\sg=\sg_{\ol{\va},1}$, 
and $\lda=\sg_{\ol{\va},1}^{\f{1}{4}}$. By scaling, there exists 
$\vp_{\ol{\va},0}\in H^1(B_{1-h}\backslash B_{1-\nu_{\ol{\va}}},\R^m)$ 
satisfying
\begin{gather*}
\int_{B_{1-h}\backslash B_{1-\nu_{\ol{\va}}}}|\na\vp_{\ol{\va},0}|^2
\leq C\sg_{\ol{\va},1}^{\f{1}{4}}(1-h),\quad
\sup_{B_{1-h}\backslash B_{1-\nu_{\ol{\va}}}}
\dist(\vp_{\ol{\va},0}(x),\cN)\leq C\sg_{\ol{\va},1}^{\f{1}{8}}.
\end{gather*}
Since $\sg_{\ol{\va},1}\to 0^+$ as $\ol{\va}\to 0^+$, we may shrink 
$\delta$ further to ensure that $\Pi_{\cN}\circ\vp_{\ol{\va},0}$ is 
well-defined for any $\ol{\va}\in(0,\delta)$. We then define
\[
\vp_{\ol{\va}}:=\left\{\begin{aligned}
&\vp_{\ol{\va},1}&&\text{if }x\in B_1\backslash B_{1-h},\\
&\Pi_{\cN}\circ\vp_{\ol{\va},0}&&\text{if }x\in B_{1-h}\backslash 
B_{1-\nu_{\ol{\va}}}.
\end{aligned}\right.
\]
By construction, $\vp_{\ol{\va}}$ satisfies the stated boundary conditions 
and the energy bound.
\end{proof}

\subsection{Luckhaus-type lemmas on the cylinder} For any $ r,L>0 $, recall that the cylinder $ \Lda_{r,L} $and its lateral surface $ \Ga_{r,L} $ are defined in \eqref{LdarLGarL}. The following lemma is the cylindrical form of Lemma~\ref{Luckhauslemma}.

\begin{lem}\label{Luckhauscylinder2}
Let $\delta\in(0,\f{1}{2}]$. Assume that for $\ol{\va}\in(0,1)$, the map $g_{\delta,\ol{\va}}\in H^1(\Ga_{\delta,1},\R^m)$ satisfies
\[
\|g_{\delta,\ol{\va}}\|_{L^{\ift}(\Ga_{\delta,1})}\leq M
\]
for some $M>0$. There exists $\eta\in(0,1)$, depending only on $f,M$, and $\cN$, such that the following holds. If $\ol{\va}\in(0,\eta\delta)$ and
\[
E_{\ol{\va}}(g_{\delta,\ol{\va}},\Ga_{\delta,1})\leq\eta\log\f{\delta}{\ol{\va}},
\]
then there exist $z_-\in(-1+\f{\delta}{2},-1+\delta)$, 
$z_+\in(1-\delta,1-\f{\delta}{2})$, 
$v_{\delta,\ol{\va}}\in H^1(\pa B_{\delta}^2\times(z_-,z_+),\cN)$, 
$\vp_{\delta,\ol{\va}}\in H^1((B_{\delta}^2\backslash B_{\f{\delta}{2}}^2)
\times(z_-,z_+),\R^m)$, and $C>0$ depending only on $f,M$, and $\cN$, 
such that the following assertions hold.
\begin{enumerate}[label=$(\theenumi)$]
\item For $\HH^2$-a.e. $x=(y,z)\in\pa B_{\delta}^2\times(z_-,z_+)$,
\be
\vp_{\delta,\ol{\va}}(x)=g_{\delta,\ol{\va}}(x)\quad\text{and}\quad
\vp_{\delta,\ol{\va}}\(\f{y}{2},z\)=v_{\delta,\ol{\va}}(y,z).
\label{propertyinterpolationcyl}
\ee
\item $v_{\delta,\ol{\va}}$ and $\vp_{\delta,\ol{\va}}$ satisfy
\begin{align}
E_{\ol{\va}}(\vp_{\delta,\ol{\va}},(B_{\delta}^2\backslash B_{\f{\delta}{2}}^2)
\times(z_-,z_+))&\leq C\delta E_{\ol{\va}}(g_{\delta,\ol{\va}},\Ga_{\delta,1}),
\label{Wvaest1}\\
E_{\ol{\va}}(v_{\delta,\ol{\va}},\pa B_{\delta}^2\times(z_-,z_+))
&\leq CE_{\ol{\va}}(g_{\delta,\ol{\va}},\Ga_{\delta,1}).\label{Vvaes1}
\end{align}
\item The restriction $\vp_{\delta,\ol{\va}}^{\pm}:=\vp_{\delta,\ol{\va}}|_{
(B_{\delta}^2\backslash B_{\f{\delta}{2}}^2)\times\{z_{\pm}\}}$ belongs to 
$H^1((B_{\delta}^2\backslash B_{\f{\delta}{2}}^2)\times\{z_{\pm}\},\R^m)$ 
and satisfies
\be
E_{\ol{\va}}(\vp_{\delta,\ol{\va}}^{\pm},(B_{\delta}^2\backslash 
B_{\f{\delta}{2}}^2)\times\{z_{\pm}\})
\leq CE_{\ol{\va}}(g_{\delta,\ol{\va}},\Ga_{\delta,1}).\label{WpluesminusC11}
\ee
\end{enumerate}
\end{lem}

\begin{proof}
This result is analogous to \cite[Lemma~58]{Can17}; see also 
\cite[Lemma~7.6]{WZ24}; we only sketch the proof. Set $\wh{\va}:=\f{\ol{\va}}{\delta}$ and $g_{\wh{\va}}(x):=g_{\delta,\ol{\va}}(\delta x)$, 
so that we can work on the rescaled domain $\Ga_{1,\delta^{-1}}$. By 
Fubini's theorem, there exist $z_+'\in(\delta^{-1}-1,\delta^{-1}-\f{1}{2})$ 
and $z_-'\in(-\delta^{-1}+\f{1}{2},-\delta^{-1}+1)$ such that
\[
E_{\wh{\va}}(g_{\wh{\va}},\pa B_1^2\times(z_-',z_+'))
\leq 4E_{\wh{\va}}(g_{\wh{\va}},\Ga_{1,\delta^{-1}}).
\]
We construct a cell decomposition of $\pa B_1^2\times[z_-',z_+']$:
\[
\pa B_1^2\times[z_-',z_+']=\bigsqcup_{j=0}^{2}\cR_j,\quad
\cR_j:=\bigsqcup_{i=1}^{k_j}c_{ij},
\]
where there exists an absolute constant $C_0>0$ such that the following properties hold.
\begin{itemize}
\item For any $j\in\{1,2\}$, the closure $\ol{c}_{ij}$ of each $j$-cell is bi-Lipschitz equivalent to $\ol{B}_1^j$: there exists a bi-Lipschitz homeomorphism $\phi_{ij}:\ol{c}_{ij}\to\ol{B}_1^j$ whose restriction to $c_{ij}$ satisfies
\[
\|\na\phi_{ij}\|_{L^{\ift}(c_{ij})}
+\|\na(\phi_{ij}^{-1})\|_{L^{\ift}(B_1^j)}\leq C_0.
\]
\item $ \#\{\ell\in\Z\cap[1,k_2]:c_{i1}\subset\pa c_{\ell,2}\}\leq C_0 $ for any $i\in\Z\cap[1,k_1]$.
\item We have
\begin{align*}
E_{\wh{\va}}(g_{\wh{\va}},\cR_1)
&\leq C_0 E_{\wh{\va}}(g_{\wh{\va}},\pa B_1^2\times(z_-',z_+')),\\
\int_{\cR_1}f(g_{\wh{\va}})\ud\HH^1
&\leq C_0\int_{\pa B_1^2\times(z_-',z_+')}f(g_{\wh{\va}})\ud\HH^2.
\end{align*}
\end{itemize}
This decomposition of $\pa B_1^2\times[z_-',z_+']$ induces a partition of 
$(B_1^2\backslash B_{\f{1}{2}}^2)\times[z_-',z_+']$ by
\begin{align*}
(B_1^2\backslash B_{\f{1}{2}}^2)\times[z_-',z_+']
=\bigsqcup_{j=0}^{2}\widehat{\cR}_j,\quad
\widehat{\cR}_j=\bigsqcup_{i=1}^{k_j}\wh{c}_{ij},\quad
\wh{c}_{ij}=\left\{(y,z)\in(B_1^2\backslash B_{\f{1}{2}}^2)\times[z_-',z_+']:
\frac{y}{|y|}\in c_{ij}\right\}.
\end{align*}
Applying the same construction as in the proof of Lemma~\ref{Luckhauslemma}, 
we obtain
\[
v_{\wh{\va}}\in H^1(\pa B_1^2\times(z_-',z_+'),\cN),\quad
\vp_{\wh{\va}}\in H^1((B_1^2\backslash B_{\f{1}{2}}^2)\times(z_-',z_+'),\R^m),
\]
such that the following properties hold.
\begin{itemize}
\item For $\HH^2$-a.e. $x=(y,z)\in\pa B_1^2\times(z_-',z_+')$,
\[
\vp_{\wh{\va}}(x)=g_{\wh{\va}}(x)\quad\text{and}\quad
\vp_{\wh{\va}}\(\f{y}{2},z\)=v_{\wh{\va}}(x).
\]
\item $v_{\wh{\va}}$ and $\vp_{\wh{\va}}$ satisfy
\begin{align}
E_{\wh{\va}}(v_{\wh{\va}},\pa B_{\f{1}{2}}^2\times(z_-',z_+'))
&\leq CE_{\wh{\va}}(g_{\wh{\va}},\Ga_{1,\delta^{-1}}),\label{Vvaescylinder}\\
E_{\wh{\va}}(\vp_{\wh{\va}},(B_1^2\backslash B_{\f{1}{2}}^2)\times(z_-',z_+'))
&\leq CE_{\wh{\va}}(g_{\wh{\va}},\Ga_{1,\delta^{-1}}).\label{Wvaestcylinder}
\end{align}
\item The restriction $\vp_{\wh{\va}}^{\pm}:=\vp_{\wh{\va}}|_{(B_1^2\backslash 
B_{\f{1}{2}}^2)\times\{z_{\pm}'\}}$ belongs to 
$H^1((B_1^2\backslash B_{\f{1}{2}}^2)\times\{z_{\pm}'\},\R^m)$ and satisfies
\be
E_{\wh{\va}}(\vp_{\wh{\va}}^{\pm},(B_1^2\backslash B_{\f{1}{2}}^2)
\times\{z_{\pm}'\})\leq CE_{\wh{\va}}(g_{\wh{\va}},\Ga_{1,\delta^{-1}}).
\label{WpluesminusCcylinder}
\ee
\end{itemize}
We now set $v_{\delta,\ol{\va}}(x):=v_{\wh{\va}}(\f{x}{\delta})$, 
$\vp_{\delta,\ol{\va}}(x):=\vp_{\wh{\va}}(\f{x}{\delta})$, and 
$z_{\pm}:=\delta z_{\pm}'$. Then \eqref{Vvaes1}, \eqref{Wvaest1}, and 
\eqref{WpluesminusC11} follow from \eqref{Vvaescylinder}, 
\eqref{Wvaestcylinder}, and \eqref{WpluesminusCcylinder}, respectively, 
by a change of variables.
\end{proof}

The next lemma treats the case in which $g_{\delta,\ol{\va}}$ 
is already uniformly close to $\cN$, so that a direct projection argument suffices in place of the cell decomposition above.

\begin{lem}\label{closeprojlem}
Let $\delta\in(0,\f{1}{2}]$. Assume that for $\ol{\va}\in(0,1)$, 
$g_{\delta,\ol{\va}}\in H^1(\Ga_{\delta,1},\R^m)$. There exists 
$\eta\in(0,1)$, depending only on $\cN$, such that the following holds. If $\ol{\va}\in(0,\delta)$ and
\be
\|\dist(g_{\delta,\ol{\va}},\cN)\|_{L^{\ift}(\Ga_{\delta,1})}<\eta,
\label{smallprojectionuse}
\ee
then there exist $z_-\in(-1+\f{\delta}{2},-1+\delta)$, $z_+\in(1-\delta,1-\f{\delta}{2})$, $v_{\delta,\ol{\va}}\in H^1(\pa B_{\delta}^2\times(z_-,z_+),\cN)$, $\vp_{\delta,\ol{\va}}\in H^1((B_{\delta}^2\backslash B_{\f{\delta}{2}}^2)
\times(z_-,z_+),\R^m)$, and $C>0$ depending only on $f,\cN$, such 
that \eqref{propertyinterpolationcyl}--\eqref{WpluesminusC11} hold.
\end{lem}

\begin{proof}
By Fubini's theorem, there exist $z_+\in(1-\delta,1-\f{\delta}{2})$ and $z_-\in(-1+\f{\delta}{2},-1+\delta)$ such that
\be
E_{\ol{\va}}(g_{\delta,\ol{\va}},\pa B_{\delta}^2\times\{z_{\pm}\})
\leq 4\delta^{-1}E_{\ol{\va}}(g_{\delta,\ol{\va}},\Ga_{\delta,1}).
\label{gdeltavaFubini}
\ee
Choosing $\eta=\eta(\cN)\in(0,1)$ sufficiently small in 
\eqref{smallprojectionuse}, the nearest-point projection $\Pi_{\cN}$ given 
by Lemma~\ref{Nfproperty} is well-defined on the image 
$g_{\delta,\ol{\va}}(\Ga_{\delta,1})$. We set 
$v_{\delta,\ol{\va}}:=\Pi_{\cN}\circ g_{\delta,\ol{\va}}$, so that for any 
$x\in\pa B_{\delta}^2\times(z_-,z_+)$,
\be
v_{\delta,\ol{\va}}(x)=\Pi_{\cN}(g_{\delta,\ol{\va}}(x)),\quad
|v_{\delta,\ol{\va}}(x)-g_{\delta,\ol{\va}}(x)|<\delta_f,\label{leqdeltaf1}
\ee
where $\delta_f>0$ is given by Lemma~\ref{Nfproperty}. It follows that
\[
E_{\ol{\va}}(v_{\delta,\ol{\va}},\pa B_{\delta}^2\times(z_-,z_+))
\leq CE_{\ol{\va}}(g_{\delta,\ol{\va}},\Ga_{\delta,1}),
\]
which is \eqref{Vvaes1}. For $x=(y,z)\in(B_{\delta}^2\backslash 
B_{\f{\delta}{2}}^2)\times(z_-,z_+)$, we define
\[
\vp_{\delta,\ol{\va}}(x):=\(1-\f{2(\delta-|y|)}{\delta}\)
g_{\delta,\ol{\va}}\(\f{\delta y}{|y|},z\)
+\f{2(\delta-|y|)}{\delta}
v_{\delta,\ol{\va}}\(\f{\delta y}{|y|},z\).
\]
Using \eqref{fBconvex} and \eqref{leqdeltaf1}, we obtain
\[
f(\vp_{\delta,\ol{\va}}(x))
\leq C\(1-\f{2(\delta-|y|)}{\delta}\)^2
f\(g_{\delta,\ol{\va}}\(\f{\delta y}{|y|},z\)\).
\]
Together with a direct calculation, this gives
\[
E_{\ol{\va}}(\vp_{\delta,\ol{\va}},(B_{\delta}^2\backslash B_{\f{\delta}{2}}^2)
\times(z_-,z_+))\leq C\delta E_{\ol{\va}}(g_{\delta,\ol{\va}},\Ga_{\delta,1}),
\]
which is \eqref{Wvaest1}. For restrictions 
$\vp_{\delta,\ol{\va}}^{\pm}:=\vp_{\delta,\ol{\va}}|_{(B_{\delta}^2\backslash 
B_{\f{\delta}{2}}^2)\times\{z_{\pm}\}}$, it follows from \eqref{gdeltavaFubini} 
that
\[
E_{\ol{\va}}(\vp_{\delta,\ol{\va}}^{\pm},(B_{\delta}^2\backslash 
B_{\f{\delta}{2}}^2)\times\{z_{\pm}\})
\leq CE_{\ol{\va}}(g_{\delta,\ol{\va}},\Ga_{\delta,1}),
\]
which is \eqref{WpluesminusC11}.
\end{proof}

\section{The clearing-out property}\label{SectionClear}

The following two propositions establish the clearing-out property for local minimizers of \eqref{GLfunctional}, first in the interior and then near the boundary.

\begin{prop}\label{InteriorClearingout}
Let $r>0$ and $x_0\in\R^3$. Assume that $u\in H^1(B_{2r}(x_0),\R^m)$ is a local minimizer of \eqref{GLfunctional} and satisfies
$\|u\|_{L^{\ift}(B_{2r}(x_0))}\leq M$. There exists $\delta>0$, depending only on $f$, $M$, and $\cN$, such that if $\va\in(0,\delta r)$ and
\[
\theta_{\va}(u;x_0,r)\leq\delta\log\f{r}{\va},
\]
then $\theta_{\va}(u;x_0,\f{r}{2})\leq C$, where $C>0$ depends only on
$f$, $M$, and $\cN$.
\end{prop}

\begin{proof}
The proof is analogous to \cite[Proposition~8]{Can17}; we sketch the key steps. By translation, we may assume $x_0=0$. Fix $\va\in(0,\f{r}{4})$ and define
\[
D^{\va}:=\left\{\rho\in\(\f{r}{2},r\):E_{\va}(u,\pa B_{\rho})
\leq 4\delta\log\f{r}{\va}\right\},
\]
where $\delta>0$ will be chosen below. By Fubini's theorem,
$\HH^1(D^{\va})\geq\f{r}{4}$. Choosing $\delta\in(0,1)$ sufficiently
small and then choosing $\delta'\in(0,1)$ accordingly, we have, for any
$\va\in(0,\delta'r)$ and any $\rho\in D^{\va}$,
\[
E_{\va}(u,\pa B_{\rho})\leq 8\delta\log\f{\rho}{\va}.
\]
We claim that if $\delta=\delta(f,M,\cN)\in(0,1)$ is sufficiently small, then
\be
E_{\va}(u,B_{\rho})\leq Cr\left((E_{\va}(u,\pa B_{\rho}))^{\f{1}{2}}+1\right)
\label{Iteene}
\ee
for any $\va\in(0,\delta'r)$ and any $\rho\in D^{\va}$, where
$C=C(f,M,\cN)>0$. Assuming \eqref{Iteene}, and applying
\cite[Lemma~46]{Can17}, we obtain $\theta_{\va}(u;0,\f{r}{2})\leq C$.
The proposition follows after replacing $\delta$ by $\min\{\delta',\delta\}$.

It remains to prove \eqref{Iteene}. Set $\ol{\va}:=\f{\va}{\rho}$ and
$u_{\ol{\va}}(x):=u(\rho x)$. Applying Lemma~\ref{Luckhauslemma}, we
construct $v_{\ol{\va}}\in H^1(\pa B_1,\cN)$ and
$\vp_{\ol{\va}}\in H^1(B_1\backslash B_{1-h(\ol{\va})},\R^m)$ that satisfy
\eqref{interpoL}, \eqref{VboundL}, and \eqref{WboundL}, where
$h(\ol{\va})=\ol{\va}^{\f{1}{2}}|\log\ol{\va}|$. By
Lemma~\ref{ExtensionLemma2} and \eqref{VboundL}, there exists
$v_{\ol{\va},1}\in H^1(B_1,\cN)$ with
$v_{\ol{\va},1}|_{\pa B_1}=v_{\ol{\va}}$ and
\be
\int_{B_1}|\na v_{\ol{\va},1}|^2\leq
C(E_{\ol{\va}}(u_{\ol{\va}},\pa B_1))^{\f{1}{2}}.
\label{12E}
\ee
Define $w_{\ol{\va}}:B_1\to\R^m$ by
\[
w_{\ol{\va}}(x):=\left\{\begin{aligned}
&\vp_{\ol{\va}}(x)&&\text{if }x\in B_1\backslash B_{1-h(\ol{\va})},\\
&v_{\ol{\va},1}\(\f{x}{1-h(\ol{\va})}\)&&\text{if }x\in B_{1-h(\ol{\va})}.
\end{aligned}\right.
\]
From \eqref{VboundL}, \eqref{WboundL}, and \eqref{12E}, we obtain
\be
E_{\ol{\va}}(w_{\ol{\va}},B_1)
\leq C\left((E_{\ol{\va}}(u_{\ol{\va}},\pa B_1))^{\f{1}{2}}+1\right).
\label{Wolinequality}
\ee
Set $w_{\va}(x):=w_{\ol{\va}}(\f{x}{\rho})$ for $x\in B_{\rho}$. Rescaling
\eqref{Wolinequality} gives
\[
E_{\va}(w_{\va},B_{\rho})
\leq Cr\left((E_{\va}(u,\pa B_{\rho}))^{\f{1}{2}}+1\right).
\]
Since $w_{\va}|_{\pa B_{\rho}}=u|_{\pa B_{\rho}}$, the minimizing property of
$u$ gives \eqref{Iteene}.
\end{proof}

The next proposition extends the clearing-out estimate to the boundary
setting, where the domain is $C^{2,1}$ and the minimizer has prescribed
boundary values.

\begin{prop}\label{BoundaryClearingout}
Let $U\subset\R^3$ be a bounded $C^{2,1}$ domain with parameters $M_{U,2}$ and $r_{U,2}$. Let $x_0\in\pa U$ and $r\in(0,r_{U,2})$. Let
$g\in H^1(T_{2r}^U(x_0),\R^m)$ satisfy
\be
\|g\|_{L^{\ift}(T_{2r}^U(x_0))}+E_{\va}(g,T_{2r}^U(x_0))\leq M.
\label{Boundaryclearingoutestimate1}
\ee
Assume that $u\in H^1(U_{2r}(x_0),\R^m)$ is a local minimizer of
\eqref{GLfunctional} in $U_{2r}(x_0)$, up to $T_{2r}^U(x_0)$, and that
$\|u\|_{L^{\ift}(U_{2r}(x_0))}\leq M$. There exist $\delta\in(0,1)$ and
$C>0$, depending only on $M$, $M_{U,2}$, $r_{U,2}$, $f$, and $\cN$, such
that if $\va\in(0,\delta r)$ and
\be
\theta_{\va}^U(u;x_0,r)<\delta\log\f{r}{\va},
\label{thetaUleqC}
\ee
then $\theta_{\va}^U(u;x_0,\f{r}{2})\leq C$.
\end{prop}

\begin{proof}
The proof is analogous to \cite[Proposition~53]{Can17}; we sketch the key
steps. By translation, we may assume $x_0=0$. Define
\[
F_{\va}(\rho):=\int_0^{\rho}E_{\va}(u,U\cap\pa B_s)\ud s,
\]
so that $F_{\va}'(\rho)=E_{\va}(u,U\cap\pa B_{\rho})$ for $\HH^1$-a.e.
$\rho\in(0,r)\subset(0,r_{U,2})$. Set
\[
D^{\va}:=\left\{\rho\in\(\f{r}{2},r\):E_{\va}(u,U\cap\pa B_{\rho})
\leq 4\delta\log\f{r}{\va}\right\}.
\]
By Fubini's theorem, $\HH^1(D^{\va})\geq\f{r}{4}$. For any $\rho\in D^{\va}$,
\eqref{Boundaryclearingoutestimate1} gives
\[
\begin{aligned}
E_{\va}(u,\pa(U_{\rho}))
&=E_{\va}(g,T_{\rho}^U)+E_{\va}(u,U\cap\pa B_{\rho})\leq C+4\delta\log\f{r}{\va}
\leq 8\delta\log\f{\rho}{\va}
\end{aligned}
\]
for any $\va\in(0,\delta'r)$, with $\delta'\in(0,1)$ sufficiently small.
Since $U_{\rho}$ is bi-Lipschitz equivalent to a ball, with constants
controlled by $M_{U,2}$ and $r_{U,2}$, we may repeat the argument of
Proposition~\ref{InteriorClearingout} to obtain
\[
F_{\va}(\rho)\leq Cr\left((E_{\va}(u,\pa(U_{\rho})))^{\f{1}{2}}+1\right)
\leq Cr\left((F_{\va}'(\rho))^{\f{1}{2}}+1\right).
\]
Applying \cite[Lemma~46]{Can17} to this inequality yields
$\theta_{\va}^U(u;0,\f{r}{2})\leq C$. The proposition follows after replacing
$\delta$ with $\min\{\delta',\delta\}$.
\end{proof}

\section{Partial regularity theory}\label{SectionPartial}

\subsection{Interior partial regularity}

We now collect the interior partial regularity properties of local minimizers of \eqref{GLfunctional}. The first lemma gives an $L^\infty$ bound on $e_{\va}(u)$ under the assumption that $u$ is already uniformly close to $\cN$.

\begin{lem}\label{lemusepar1}
Let $M,r>0$ and $x_0\in\R^3$. Assume that
$u\in H^1(B_{2r}(x_0),\R^m)$ is a local minimizer of \eqref{GLfunctional}
that satisfies $\|u\|_{L^{\ift}(B_{2r}(x_0),\R^m)}\leq M$. There exists
$\delta\in(0,1)$, depending only on $f$ and $\cN$, such that if
$\va\in(0,r)$, $\|\dist(u,\cN)\|_{L^{\ift}(B_{2r}(x_0))}<\delta$, and
$\theta_{\va}(u;x_0,2r)<\delta$, then
\[
r^2\|e_{\va}(u)\|_{L^{\ift}(B_r(x_0))}\leq C,
\]
where $C>0$ depends only on $f$, $M$, and $\cN$.
\end{lem}

\begin{proof}
When $\|\dist(u,\cN)\|_{L^{\ift}(B_{2r}(x_0))}<\delta$ with
$\delta=\delta(f,\cN)\in(0,1)$ sufficiently small, \cite[Lemma~3.2]{CLR18}
implies the Bochner-type inequality
\be
-\Delta e_{\va}(u)\leq C(e_{\va}(u))^2.
\label{Bochner1}
\ee
The proof then follows from the argument of \cite[Lemma~7]{MZ10}, using
\eqref{Bochner1} together with the monotonicity formula \eqref{Mo11}.
\end{proof}

The second lemma shows that a small density forces $u$ to be uniformly close
to $\cN$.

\begin{lem}\label{lemusepar11}
Let $M,r>0$ and $x_0\in\R^3$. Assume that
$u\in H^1(B_{2r}(x_0),\R^m)$ is a local minimizer of \eqref{GLfunctional}
that satisfies $\|u\|_{L^{\ift}(B_{2r}(x_0),\R^m)}\leq M$. For any $\eta>0$,
there exists $\delta\in(0,1)$, depending only on $\eta$, $f$, $M$, and
$\cN$, such that if $\va\in(0,\delta r)$ and
$\theta_{\va}(u;x_0,2r)<\delta$, then
$\|\dist(u,\cN)\|_{L^{\ift}(B_r(x_0))}<\eta$.
\end{lem}

\begin{proof}
By scaling and translation, we may assume $r=1$ and $x_0=0$. Fix
$y_0\in B_1$ and set $\al_0:=f(u(y_0))$. By Lemma~\ref{Apriori},
\[
0\leq\al_0\leq f(u(y))+\f{C_0}{\va}|y-y_0|,\quad y\in B_1(y_0),
\]
where $C_0>0$ depends only on $f$, $M$, and $\cN$. Fix
$\rho\in(0,\f{1}{2})$. The monotonicity formula \eqref{Mo11} gives
\be
\f{\rho^2}{\va^2}\left(\al_0-\f{C_0\rho}{\va}\right)
\leq\f{C}{\rho}\int_{B_{\rho}(y_0)}\f{f(u)}{\va^2}\ud y
\leq C_1\theta_{\va}(u;0,2)<C_1\delta,
\label{leqdeltause}
\ee
where $C_1>0$ depends only on $f$, $M$, and $\cN$. Set
$\rho:=\f{\al_0\va}{2C_0}$. Since $\va\in(0,\delta)$, we have
$\rho\in(0,\f{1}{2})$ for $\delta=\delta(f,M,\cN)\in(0,1)$ sufficiently
small. From \eqref{leqdeltause}, it follows that $ 0\leq\al_0\leq 2C_1^{\f{1}{3}}C_0^{\f{2}{3}}\delta^{\f{1}{3}} $. Since $\|u\|_{L^{\ift}}\leq M$ and $\cN$ is compact, we may choose $\delta\in(0,1)$ smaller, depending on $\eta$, $f$, $M$, and $\cN$, so that $f(u(y_0))<\delta$ implies $\dist(u(y_0),\cN)<\eta$. As $y_0\in B_1$ is arbitrary, the proof is complete.
\end{proof}

Combining the preceding two lemmas gives the basic interior
$\varepsilon$-regularity estimate.

\begin{prop}\label{InteriorPartialRegularity1}
Let $M,r>0$ and $x_0\in\R^3$. Assume that
$u\in H^1(B_{2r}(x_0),\R^m)$ is a local minimizer of \eqref{GLfunctional}
in $B_{2r}(x_0)$ and satisfies
$\|u\|_{L^{\ift}(B_{2r}(x_0),\R^m)}\leq M$. There exists
$\delta\in(0,1)$, depending only on $f$, $M$, and $\cN$, such that if
$\theta_{\va}(u;x_0,2r)<\delta$ and $\va\in(0,\delta r)$, then
\be
r^2\|e_{\va}(u)\|_{L^{\ift}(B_r(x_0))}\leq 1.
\label{leq1partial}
\ee
\end{prop}

\begin{proof}
By scaling and translation, we may assume $r=1$ and $x_0=0$. Choosing
$\delta=\delta(f,M,\cN)\in(0,1)$ sufficiently small and applying
Lemmas~\ref{lemusepar1} and \ref{lemusepar11}, we find that if
$\va\in(0,\delta)$ and $\theta_{\va}(u;0,2)<\delta$, then
\be
\|e_{\va}(u)\|_{L^{\ift}(B_{\f{3}{2}})}\leq C,
\label{evapre1}
\ee
where $C=C(f,M,\cN)>0$. By Lemma~\ref{lemusepar11} and \eqref{Bochner1}, we have
\[
-\Delta e_{\va}(u)\leq C(e_{\va}(u))^2\leq Ce_{\va}(u),
\]
where the second inequality follows from~\eqref{evapre1}. Harnack's inequality
\cite[Theorem~4.1]{HL11} gives
\[
\|e_{\va}(u)\|_{L^{\ift}(B_1)}
\leq C\int_{B_{\f{3}{2}}}e_{\va}(u)\ud x<C\delta.
\]
The estimate \eqref{leq1partial} follows by choosing $\delta$ smaller if
necessary.
\end{proof}

The next proposition shows that, once the energy density is bounded
pointwise, the potential term $f(u)$ decays at an improved rate as
$\va\to 0^+$.

\begin{prop}\label{improvedpotential}
Let $M,r>0$, $\va\in(0,r)$, and $x_0\in\R^3$. Assume that
$u\in C^{\ift}(B_{2r}(x_0),\R^m)$ solves \eqref{EL} in $B_{2r}(x_0)$ and satisfies
\be
r^2\|e_{\va}(u)\|_{L^{\ift}(B_{2r}(x_0))}\leq M.
\label{nablaufirstesti}
\ee
There exists $\delta\in(0,1)$, depending only on $f$, $M$, and $\cN$, such
that if $\va\in(0,\delta r)$, then
\[
\|f(u)\|_{L^{\ift}(B_r(x_0))}\leq C\va^4r^{-4},
\]
where $C>0$ depends only on $f$, $M$, and $\cN$.
\end{prop}

\begin{proof}
For $\va\in(0,\delta r)$, the bound \eqref{nablaufirstesti} gives
\[
\|f(u)\|_{L^{\ift}(B_{2r}(x_0))}\leq M\va^2r^{-2}\leq M\delta^2.
\]
Choosing $\delta=\delta(f,M,\cN)\in(0,1)$ sufficiently small, we obtain
$\dist(u,\cN)<\delta_f$ in $B_{2r}(x_0)$, where $\delta_f>0$ is given by
Lemma~\ref{Nfproperty}. A direct computation gives
\begin{align*}
\va^2\Delta(f(u))
&=\va^2D^2f(u)[\na u,\na u]+\va^2Df(u):\Delta u\\
&\stackrel{\eqref{EL}}{=}\va^2D^2f(u)[\na u,\na u]+|Df(u)|^2.
\end{align*}
Using \eqref{mfMfestimates} and \eqref{Dfyestimate}, and the fact that
$\dist(u,\cN)<\delta_f$ in $B_{2r}(x_0)$, we obtain
\be
\va^2\Delta(f(u))
\geq\va^2D^2f(u)[\na u,\na u]+C^{-1}f(u).
\label{vaDeltafu}
\ee
Since $f$ is smooth and $\cN$ is bounded, the mean value theorem gives
\[
|D^2f(u)-D^2f(\Pi_{\cN}u)|\leq C|u-\Pi_{\cN}u|.
\]
Together with \eqref{A4}, \eqref{mfMfestimates}, and \eqref{vaDeltafu}, and
then by the Cauchy--Schwarz inequality, this yields
\[
\va^2\Delta(f(u))
\geq -C\va^2(f(u))^{\f{1}{2}}|\na u|^2+C^{-1}f(u)
\geq C^{-1}f(u)-C\va^4|\na u|^4
\stackrel{\eqref{nablaufirstesti}}{\geq}C^{-1}f(u)-C\va^4r^{-4}.
\]
The conclusion follows from \cite[Lemma~6]{NZ13}.
\end{proof}

\subsection{Boundary partial regularity}

In this subsection, we prove boundary partial regularity estimates parallel
to Proposition~\ref{InteriorPartialRegularity1} and
Proposition~\ref{improvedpotential}. We first state the key technical lemma
and then deduce the boundary estimate from it.

\begin{lem}\label{BoundaryPartialRegularity1}
Let $U\subset\R^3$ be a bounded $C^{2,1}$ domain with parameters $M_{U,2}$ and $r_{U,2}$. Let $x_0\in\pa U$ and $r\in(0,r_{U,2})$. Let
$g\in C^2(T_{2r}^U(x_0),\cN)$. Assume that
$u\in H^1(U_{2r}(x_0),\R^m)$ is a local minimizer of \eqref{GLfunctional}
in $U_{2r}(x_0)$ up to the boundary $T_{2r}^U(x_0)$, with $u=g$ on
$T_{2r}^U(x_0)$, and
\begin{gather*}
r\|(|\na_{\top}g|+r|D_{\top}^2g|)\|_{L^{\ift}(T_{2r}^U(x_0))}\leq M,\\
\|u\|_{L^{\ift}(U_{2r}(x_0))}+\theta_{\va}^U(u;x_0,2r)\leq M.
\end{gather*}
There exist $\delta\in(0,1)$ and $C>0$, depending only on
$M$, $M_{U,2}$, $r_{U,2}$, $f$, and $\cN$, such that if
$\va\in(0,\delta r)$, $r\in(0,\delta)$,
\begin{gather*}
r\|\na_{\top}g\|_{L^{\ift}(T_{2r}^U(x_0))}<\delta,\\
\vt_{\va}^U\(u;x_0,\f{r}{10}\)-\vt_{\va}^U\(u;x_0,\f{r}{20}\)<\delta,
\end{gather*}
then
\[
r^2\|e_{\va}(u)\|_{L^{\ift}(U_{\f{r}{40}}(x_0))}\leq C.
\]
\end{lem}

\begin{proof}
By translation, we may assume $x_0=0$. Throughout the proof, $C$ denotes a
positive constant depending only on $M$, $M_{U,2}$, $r_{U,2}$, $f$, and
$\cN$, and it may change from line to line. We argue by contradiction.
Suppose that the conclusion fails. Then there exist sequences
$\{u_i\}_{i\in\N}\subset H^1(U_{2r_i},\R^m)$ of local minimizers of
\eqref{GLfunctional} up to the boundary $T_{2r_i}^U$, and
$\{g_i\}_{i\in\N}\subset C^2(T_{2r_i}^U,\cN)$, such that the following
properties hold.
\begin{itemize}
\item $u_i=g_i$ on $T_{2r_i}^U$, and
\be
\begin{gathered}
r_i\|(|\na_{\top}g_i|+r_i|D_{\top}^2g_i|)\|_{L^{\ift}(T_{2r_i}^U)}\leq M,\\
\|u_i\|_{L^{\ift}(U_{2r_i})}+\theta_{\va_i}(u_i;0,2r_i)\leq M.
\end{gathered}
\label{energyassumptionui}
\ee
\item For a sequence $\delta_i\to 0^+$, with $\va_i\in(0,\delta_ir_i)$ and
$r_i\in(0,\delta_i)$, we have
\begin{gather}
r_i\|\na_{\top}g_i\|_{L^{\ift}(T_{2r_i}^U)}<\delta_i,
\label{gideltai}\\
\vt_{\va_i}^U\(u_i;0,\f{r_i}{10}\)
-\vt_{\va_i}^U\(u_i;0,\f{r_i}{20}\)<\delta_i.
\label{vtvaileq}
\end{gather}
\item The energy density blows up at the claimed scale:
\be
\lim_{i\to+\ift}r_i^2\|e_{\va_i}(u_i)\|_{L^{\ift}(U_{\f{r_i}{40}})}
=+\ift.
\label{Ur40infty}
\ee
\end{itemize}

Define the rescaled quantities
$\wt{u}_i:=u_i(r_i\cdot)$, $\wt{g}_i:=g_i(r_i\cdot)$, and
$\wt{U}:=r_i^{-1}U$. For sufficiently large $i$, since $r_i\in(0,\delta_i)$,
\be
\begin{gathered}
0<M_{\wt{U},2}\leq M_{U,2}r_i<M_{U,2}\delta_i,\\
r_{\wt{U},2}=r_i^{-1}r_{U,2}>\delta_i^{-1}r_{U,2}>10.
\end{gathered}
\label{MwtUrwtU}
\ee
Since the rescaled flattening radius satisfies $r_{\wt{U},2}>10$, Definition \ref{RegularityBoundary} applies to $\wt{U}$ at the fixed scale $2$. Together with $M_{\wt{U},2}\leq M_{U,2}\delta_i$, this gives, for
each $i$, a bi-Lipschitz homeomorphism
$\Phi_i:\ol{B}_2^+\to\ol{\wt{U}\cap B_2}$ whose restriction to $B_2^+$
satisfies
\be
\|\na\Phi_i-\op{Id}_3\|_{L^{\ift}(B_2^+)}
+\|\na(\Phi_i^{-1})-\op{Id}_3\|_{L^{\ift}(\wt{U}\cap B_2)}
\leq C\delta_i.
\label{PhiId}
\ee
By \eqref{gideltai} and \eqref{energyassumptionui}, there exist
$p_0\in\cN$ and $\wt{u}_0\in H^1(B_2^+,\R^m)$ such that, up to a subsequence,
\be
\begin{aligned}
h_i:=\wt{g}_i\circ\Phi_i|_{T_2}
&\to p_0\quad\text{strongly in }C^1(T_2,\R^m),\\
v_i:=\wt{u}_i\circ\Phi_i
&\wc\wt{u}_0\quad\text{weakly in }H^1(B_2^+,\R^m).
\end{aligned}
\label{hiviconvergence}
\ee
Moreover, $v_i\to\wt{u}_0$ a.e.\ in $B_2^+$. By
\eqref{energyassumptionui} and $\f{\va_i}{r_i}\in(0,\delta_i)$,
\[
0\leq\limsup_{i\to+\ift}\int_{B_2^+}f(v_i)
\leq\limsup_{i\to+\ift}C\left(\f{\va_i}{r_i}\right)^2
\leq\limsup_{i\to+\ift}C\delta_i^2=0.
\]
Fatou's lemma gives
\[
0\leq\int_{B_2^+}f(\wt{u}_0)
\leq\liminf_{i\to+\ift}\int_{B_2^+}f(v_i)=0,
\]
and hence $\wt{u}_0\in H^1(B_2^+,\cN)$.

Fix $\lda\in(0,1)$. Applying Fatou's lemma again,
\[
\int_{2\lda}^2\liminf_{i\to+\ift}
E_{\f{\va_i}{r_i}}(v_i,(\pa B_{\rho})^+)\ud\rho
\leq\liminf_{i\to+\ift}
E_{\f{\va_i}{r_i}}(v_i,B_2^+\backslash B_{2\lda}^+)\leq C_0,
\]
where $C_0=C_0(M,M_{U,2},r_{U,2})>0$. Define
\[
D_{\lda}:=\left\{\rho\in(0,2]:
\liminf_{i\to+\ift}E_{\f{\va_i}{r_i}}(v_i,(\pa B_{\rho})^+)
>\f{2C_0}{1-\lda}\right\}.
\]
By Fubini's theorem, $\HH^1(D_{\lda})\leq 1-\lda$. Therefore, there exists
$\rho_0\in(2\lda,2]$ such that
\[
\liminf_{i\to+\ift}E_{\f{\va_i}{r_i}}(v_i,(\pa B_{\rho_0})^+)
\leq\f{2C_0}{1-\lda}.
\]
Together with \eqref{gideltai}, and after passing to a further subsequence,
this gives
\be
\sup_{i\in\N}E_{\f{\va_i}{r_i}}(v_i,\pa(B_{\rho_0}^+))\leq C.
\label{supivairi}
\ee

Let $\Psi_i:\ol{B_1}\to\ol{B}_1^+$ be a bi-Lipschitz homeomorphism whose
restriction to $B_1$ satisfies
\[
\|\na\Psi_i\|_{L^{\ift}(B_1)}
+\|\na(\Psi_i^{-1})\|_{L^{\ift}(B_1^+)}\leq C.
\]
Set $ \wt{v}_i:=v_i(\rho_0\cdot) $, $
\wt{h}_i:=h_i(\rho_0\cdot) $, $ \wt{v}_0:=\wt{u}_0(\rho_0\cdot) $, and $
\ol{\va}_i:=(\rho_0r_i)^{-1}\va_i $.

Let $\wt{w}_0\in H^1(B_1^+,\cN)$ be a minimizer of
\be
\inf\left\{\f{1}{2}\int_{B_1^+}|\na w|^2\ud x:
w\in H^1(B_1^+,\cN),\,w|_{\pa(B_1^+)}=\wt{v}_0\right\}.
\label{w0problem}
\ee
By \eqref{supivairi} and the lower semicontinuity,
\[
\int_{\pa(B_1^+)}|\na_{\top}\wt{v}_0|^2\ud\HH^2
\leq\liminf_{i\to+\ift}E_{\f{\va_i}{r_i}}(\wt{v}_i,\pa(B_1^+))\leq C.
\]
Combining this with \eqref{hiviconvergence} and the compactness of the trace
embedding, we obtain
\[
\int_{\pa(B_1^+)}\left(|\na_{\top}\wt{v}_i|^2
+\f{1}{\ol{\va}_i^2}f(\wt{v}_i)+|\na_{\top}\wt{v}_0|^2
+\f{|\wt{v}_i-\wt{v}_0|^2}{\sg_i^2}\right)\ud\HH^2\leq C,
\]
where $ \sg_i:=\|\wt{v}_i-\wt{v}_0\|_{L^2(\pa(B_1^+))}+i^{-1} $. Applying Proposition~\ref{Luckhaus11} to
$\wt{v}_i\circ\Psi_i$ and $\wt{v}_0\circ\Psi_i$, we obtain
$\nu_i\to 0^+$ and a sequence
$\{\vp_i\}_{i\in\N}\subset H^1(B_1\backslash B_{1-\nu_i},\R^m)$ such that
\be
\begin{gathered}
\vp_i(y)=\wt{v}_i\circ\Psi_i(y)\text{ and }
\vp_i((1-\nu_i)y)=\wt{v}_0\circ\Psi_i(y)
\quad\text{for }\HH^2\text{-a.e. }y\in\pa B_1,\\
E_{\ol{\va}_i}(\vp_i,B_1\backslash B_{1-\nu_i})\leq C\nu_i.
\end{gathered}
\label{vpiw0}
\ee
Define
\[
w_i(y):=\left\{\begin{aligned}
&\vp_i(y)&&\text{if }y\in B_1\backslash B_{1-\nu_i},\\
&\wt{w}_0\circ\Psi_i\(\f{y}{1-\nu_i}\)&&\text{if }y\in B_{1-\nu_i}.
\end{aligned}\right.
\]
Using the minimality of $u_i$, the flattening estimate \eqref{PhiId}, and
\eqref{vpiw0}, we obtain
\be
\begin{aligned}
\f{1}{2}\int_{B_1^+}|\na\wt{v}_0|^2\ud x
&\leq\liminf_{i\to+\ift}\f{1}{2}\int_{B_1^+}|\na\wt{v}_i|^2\ud x\leq\limsup_{i\to+\ift}\f{1}{2}E_{\ol{\va}_i}(\wt{v}_i,B_1^+)\\
&\leq\limsup_{i\to+\ift}\f{1}{2}E_{\ol{\va}_i}(w_i\circ\Psi_i^{-1},B_1^+)\leq\f{1}{2}\int_{B_1^+}|\na\wt{w}_0|^2\ud x\leq\f{1}{2}\int_{B_1^+}|\na\wt{v}_0|^2\ud x.
\end{aligned}
\label{longchain}
\ee
Here, the third inequality follows from the minimality of $u_i$ after flattening, the
fourth uses \eqref{vpiw0} and $\nu_i\to 0^+$, and the fifth uses the
minimality of $\wt{w}_0$ in \eqref{w0problem}. Hence, all inequalities in
\eqref{longchain} are equalities. Since $\lda\in(0,1)$ was arbitrary,
$\wt{u}_0\in H^1(B_2^+,\cN)$ is a local minimizer up to the boundary
$T_1$, with $\wt{u}_0=p_0$ on $T_1$, and
\be
v_i\to\wt{u}_0\quad\text{strongly in }H^1(B_{\f{3}{2}}^+,\R^m).
\label{viconvergence}
\ee

By \eqref{phiproperty0}, \eqref{Monotone11}, \eqref{gideltai},
\eqref{vtvaileq}, and \eqref{MwtUrwtU},
\begin{align*}
&Cr_i^{-3}\int_{U_{\f{9r_i}{20}}}|y\cdot\na u_i|^2\ud y\leq\int_{\f{r_i}{20}}^{\f{r_i}{10}}
\left(-\f{2}{\rho^2}\int_U\left|\f{y}{\rho}\cdot\na u_i\right|^2\overline{\phi}_{0,\rho}
+\f{2}{\va_i^2\rho^2}\int_Uf(u_i)\phi_{0,\rho}\right)\ud\rho\\
&\quad\leq\vt_{\va_i}^U\(u_i;0,\f{r_i}{10}\)
-\vt_{\va_i}^U\(u_i;0,\f{r_i}{20}\)
+C\left(\f{\delta_i}{r_i}(\delta_i+M^{\f{1}{2}})+(\delta_i^2+M)\right)r_i
\to 0^+
\end{align*}
as $i\to+\ift$. Together with \eqref{PhiId} and \eqref{viconvergence}, this
gives
\[
\lim_{i\to+\ift}\int_{B_{\f{1}{4}}^+}|y\cdot\na v_i|^2\ud y=0.
\]
Thus, $\wt{u}_0$ is $0$-homogeneous in $B_{\f{1}{4}}^+$, that is,
$\wt{u}_0(tx)=\wt{u}_0(x)$ for any $t>0$ with
$x,tx\in B_{\f{1}{4}}^+$. Combined with the minimizing property of
$\wt{u}_0$, \cite[Theorem~2.4.3]{LW08} shows that $\wt{u}_0$ is constant.
By \eqref{viconvergence},
\[
\theta_{\va_i}^U\(u_i;0,\f{r_i}{4}\)
=\theta_{\f{\va_i}{r_i}}^{\wt{U}}\(\wt{u}_i;0,\f{1}{4}\)\to 0^+.
\]
It follows from \eqref{Monotone11b} and \eqref{energyassumptionui} that, for
any ball $B_{\rho}(x)\subset B_{\f{r_i}{8}}$,
\[
\theta_{\va_i}^U(u_i;x,\rho)
\leq C\delta_i+C\theta_{\va_i}^U\(u_i;0,\f{r_i}{4}\)\to 0^+
\]
as $i\to+\ift$. By \cite[Proposition~3.1]{CLR18},
\[
r_i^2\|e_{\va_i}(u_i)\|_{L^{\ift}(U_{\f{r_i}{40}})}\leq C,
\]
contradicting \eqref{Ur40infty}.
\end{proof}

The final proposition gives the boundary version of the pointwise energy
estimate and the improved potential bound.

\begin{prop}\label{BoundaryPartialRegularity1prop}
Let $M>0$ and let $U\subset\R^3$ be a bounded $C^{2,1}$ domain with
parameters $M_{U,2}$ and $r_{U,2}$. Let $x_0\in\pa U$ and
$r\in(0,r_{U,2})$. Let $g\in C^2(T_{2r}^U(x_0),\cN)$ satisfy
\[
r\|(|\na_{\top}g|+r|D_{\top}^2g|)\|_{L^{\ift}(T_{2r}^U(x_0))}\leq M.
\]
Assume that $u\in H^1(U_{2r}(x_0),\R^m)$ is a local minimizer of
\eqref{GLfunctional} in $U_{2r}(x_0)$ up to the boundary
$T_{2r}^U(x_0)$, with $u=g$ on $T_{2r}^U(x_0)$ and
\be
\|u\|_{L^{\ift}(U_{2r}(x_0))}+\theta_{\va}^U(u;x_0,2r)\leq M.
\label{thetavaassenb}
\ee
There exist $\delta\in(0,1)$ and $C>0$, depending only on
$M$, $M_{U,2}$, $r_{U,2}$, $f$, and $\cN$, such that if
$\va\in(0,\delta r)$, then
\begin{gather*}
(\delta r)^2\|e_{\va}(u)\|_{L^{\ift}(U_{\delta r}(x_0))}\leq C,\\
\|f(u)\|_{L^{\ift}(U_{\delta r}(x_0))}\leq C\va^4r^{-4}.
\end{gather*}
\end{prop}

\begin{proof}
By translation, we may assume $x_0=0$. Let $\eta>0$ be chosen later. For
$\delta_0=\delta_0(M,\eta)\in(0,\f{1}{100})$ sufficiently small,
\be
(\delta_0r)\|\na_{\top}g\|_{L^{\ift}(T_{2\delta_0r}^U)}
\leq M\delta_0<\eta.
\label{delta1rnabla}
\ee
Let $s\in(0,\delta_0r)$. The monotonicity formula \eqref{Monotone11} implies
that $\vt_{\va}^U(u;0,s)+C_0\f{s}{r}$ is non-decreasing in $s$, where
$C_0=C_0(M,M_{U,2},r_{U,2})>0$. By \eqref{thetavaassenb},
\[
\sum_{j=1}^N\left(\vt_{\va}^U\(u;0,\f{\delta_0r}{10\cdot 2^j}\)
-\vt_{\va}^U\(u;0,\f{\delta_0r}{10\cdot 2^{j+1}}\)\right)
\leq\vt_{\va}^U\(u;0,\f{\delta_0r}{20}\)\leq C.
\]
By a pigeonhole argument, there exists $j_0\in[1,N]$ such that
\[
\vt_{\va}^U\(u;0,\f{\delta_0r}{10\cdot 2^{j_0}}\)
-\vt_{\va}^U\(u;0,\f{\delta_0r}{10\cdot 2^{j_0+1}}\)
\leq\f{C}{N}.
\]
Taking $N:=\lceil\log\delta_0^{-1}\rceil$, we find
$s_0\in(\delta_0^2r,\delta_0r]$ such that
\be
\vt_{\va}^U\(u;0,\f{s_0}{10}\)
-\vt_{\va}^U\(u;0,\f{s_0}{20}\)
<\f{C}{|\log\delta_0|}<\eta
\label{decoeta}
\ee
provided $\delta_0=\delta_0(f,M,M_{U,2},r_{U,2},\cN)\in(0,\f{1}{100})$ is
sufficiently small. Choose $\va\in(0,\eta\delta_0^2r)$. Then, for
$\delta_0$ sufficiently small, we have $\va\in(0,\eta s_0)$ and
$s_0\in(0,\eta)$. By \eqref{delta1rnabla},
\[
s_0\|\na_{\top}g\|_{L^{\ift}(T_{2s_0}^U)}<\eta.
\]
Choosing $\eta=\eta(f,M,M_{U,2},r_{U,2},\cN)\in(0,1)$ sufficiently small
and applying Lemma~\ref{BoundaryPartialRegularity1}, we obtain
\[
s_0^2\|e_{\va}(u)\|_{L^{\ift}(U_{\f{s_0}{40}})}\leq C.
\]
Set $\delta_1:=\f{\eta\delta_0^2}{80}$. For any
$\va\in(0,\delta_1r)$, the inclusion $U_{2\delta_1r}\subset U_{\f{s_0}{40}}$
gives
\[
(\delta_1r)^2\|e_{\va}(u)\|_{L^{\ift}(U_{2\delta_1r})}\leq C.
\]
In $U_{2\delta_1r}$ we have
$0\leq f(u)\leq C\va^2(\delta_1r)^{-2}$. By Lemma~\ref{Nfproperty}, there
exists $\delta_2\in(0,1)$ such that, for $\va\in(0,\delta_2r)$, $ \|\dist(u,\cN)\|_{L^{\ift}(U_{2\delta_1r})}<\delta_f $. Repeating the argument of Proposition~\ref{improvedpotential}, we obtain
\[
\va^2\Delta(f(u))\geq C^{-1}f(u)-C\va^4r^{-4}\quad\text{in }U_{2\delta_1r}.
\]
It follows from \cite[Lemma~6]{NZ13} that
\[
\|f(u)\|_{L^{\ift}(U_{\delta_1r})}\leq C\va^4r^{-4}.
\]
The proposition follows after setting $\delta:=\min\{\delta_1,\delta_2\}$.
\end{proof}

\section{The singular set}\label{SectionSingular}

Throughout this section, we retain the notation and assumptions of
Theorem~\ref{globalminimizersproperties}. Specifically, $\om\subset\R^3$ is a
bounded $C^{2,1}$ domain whose closure $\ol{\om}$ is strongly convex at each
point of $\pa\om$. The family
$\{g_{\va}\}_{\va\in(0,1)}\subset C^2(\pa\om,\R^m)$ is a class of suitable
boundary data with respect to $\cA=\{a_i\}_{i=1}^k\subset\om$ and $M_0>0$,
and there exists $g\in C^2(\pa\om\backslash\cA,\cN)$ such that
\[
g_{\va}\to g\quad\text{in }C_{\loc}^2(\pa\om\backslash\cA,\R^m)\quad\text{as }\va\to 0^+.
\]
The family $\{u_{\va}\}_{\va\in(0,1)}$ consists of global minimizers of
\eqref{GLfunctional} subject to the boundary condition $u_{\va}=g_{\va}$ on
$\pa\om$.

By Proposition~\ref{locallogestimate}, there exists $M>0$, depending only on
$\cA$, $f$, $M_0$, $\cN$, and $\om$, such that
\be
E_{\va}(u_{\va},\om)\leq M(|\log\va|+1)
\label{assumptionbound}
\ee
and
\be
\|u_{\va}\|_{L^{\ift}(\om)}\leq M.
\label{assumptionbound1}
\ee

For each $\va\in(0,1)$, define the non-negative Radon measure $\mu_{\va}$ on
$\ol{\om}$ associated with $u_{\va}$ by
\[
\mu_{\va}(U):=\f{1}{|\log\va|}\int_{U\cap\om}e_{\va}(u_{\va})\ud x
\]
for any Borel set $U\subset\ol{\om}$. By \eqref{assumptionbound},
\be
\sup_{\va\in(0,\f{1}{2})}\mu_{\va}(\ol{\om})
\leq M\(1+\f{1}{\log 2}\).
\label{UniformboundMeasure}
\ee
The uniform bound \eqref{UniformboundMeasure} implies that there exist a
non-negative Radon measure $\mu_*\in(C^0(\ol{\om}))'$ and a sequence
$\va_i\to 0^+$ such that
\be
\mu_{\va_i}\wc^*\mu_*\quad\text{weakly${}^*$ in }(C^0(\ol{\om}))'.
\label{munconve}
\ee
Let $S_*:=\supp(\mu_*)$ denote the support of $\mu_*$ in $\ol{\om}$. Then
$S_*$ is a closed subset of $\ol{\om}$.

\subsection{Interior properties of the singular set}

\subsubsection{Preliminary properties}

We first record a clearing-out consequence for the limiting measure. It shows
that a sufficiently small normalized mass excludes points of the singular set
at a smaller scale.

\begin{lem}\label{smallmu0re}
There exists $\delta>0$, depending only on $f$, $M$, and $\cN$, such that
for any $x\in\om$ and any $r\in(0,\dist(x,\pa\om))$, if
$\mu_*(\ol{B}_r(x))<\delta r$, then $\mu_*(B_{\f{r}{2}}(x))=0$. In
particular, $B_{\f{r}{2}}(x)\subset\om\backslash S_*$.
\end{lem}

\begin{proof}
This is a direct consequence of Proposition~\ref{InteriorClearingout}; see
\cite[Lemma~49]{Can17} for the same argument.
\end{proof}

The next lemma passes the interior monotonicity formula to the limiting
measure.

\begin{lem}\label{monomu0}
For any $x\in\om$, the map
\[
r\in(0,\dist(x,\pa\om))\mapsto\f{\mu_*(\ol{B}_r(x))}{2r}
\]
is non-decreasing.
\end{lem}

\begin{proof}
This follows from the monotonicity formula \eqref{Mo11} by passing to the
limit along the sequence $\va_i\to 0^+$.
\end{proof}

In view of Lemma~\ref{monomu0}, the following density is well-defined. We
define $\Theta^1(\mu_*,\cdot):\om\to\R$ by
\be
\Theta^1(\mu_*,x):=\lim_{r\to 0^+}\f{\mu_*(\ol{B}_r(x))}{2r},
\quad x\in\om,
\label{Thetadensity}
\ee
and call $\Theta^1(\mu_*,x)$ the $1$-dimensional density of $\mu_*$ at $x$.

The following lemma identifies the interior singular set with the set of
positive density points and gives a uniform lower bound for the density on
that set.

\begin{lem}\label{Thetamu0eta0}
There exists $\delta>0$, depending only on $f$, $M$, and $\cN$, such that
\[
S_*\cap\om
=\{x\in\om:\Theta^1(\mu_*,x)>0\}
=\{x\in\om:\Theta^1(\mu_*,x)\geq\delta\}.
\]
\end{lem}

\begin{proof}
This follows from Lemma~\ref{smallmu0re} by the argument of
\cite[Lemma~50]{Can17}.
\end{proof}

The next proposition describes the convergence of the minimizers away from
$S_*$ and gives the compactness required below.

\begin{prop}\label{propustar}
For any open set $U\subset\subset\om\backslash S_*$,
\be
\limsup_{i\to+\ift}E_{\va_i}(u_{\va_i},U)<+\ift.
\label{limsupleqC}
\ee
Moreover, there exists $u_*\in H_{\loc}^1(\om\backslash S_*,\cN)$ such that,
up to a subsequence, $u_{\va_i}\to u_*$ strongly in
$H_{\loc}^1(\om\backslash S_*,\R^m)$, and the following properties hold.
\begin{enumerate}[label=$(\theenumi)$]
\item $u_*$ is a local minimizer of \eqref{Dirichlet}.
\item Let $S_0:=\sing(u_*)$. Then $S_0$ is locally finite in
$\om\backslash S_*$, and
\[
u_{\va_i}\to u_*
\quad\text{in }C_{\loc}^0(\om\backslash(S_*\cup S_0),\R^m).
\]
\end{enumerate}
\end{prop}

\begin{proof}
The bound \eqref{limsupleqC} follows from Proposition~\ref{InteriorClearingout}.
The convergence of $u_{\va_i}$ to $u_*$ follows from
Proposition~\ref{interiorcompactnesslem}. The two stated properties are then
standard consequences of the partial regularity theory for minimizing
$\cN$-valued harmonic maps.
\end{proof}

\subsubsection{Further structure of the singular set}

We now formulate the varifold structure associated with the limiting measure.
Let $\bG(3,1)$ denote the Grassmann manifold of $1$-dimensional subspaces of
$\R^3$. Identifying each line with the corresponding orthogonal projection
matrix, we write
\[
\bG(3,1)=\{A\in\MM_3(\R):A^2=A,\, A^{\T}=A,\, \rank(A)=1\}.
\]
Set $\bG_3(\om):=\om\times\bG(3,1)$. A $1$-dimensional varifold in $\om$ is
a non-negative Radon measure on $\bG_3(\om)$.

Let $S\subset\om$ be a countably $\HH^1$-rectifiable set, and let
$\mu\in(C^0(\ol{\om}))'$ be a non-negative Radon measure such that
$\mu\llcorner\om=\theta_S\HH^1\llcorner S$, where $\theta_S$ is locally
$\HH^1$-integrable on $S$. For $\HH^1$-a.e. $x\in S$, let
$T_xS\in\bG(3,1)$ denote the approximate tangent line to $S$ at $x$. For
$\mu$-a.e. $x\in\om$, let $A(x)\in\bG(3,1)$ be the orthogonal projection onto
$T_xS$. The varifold associated with $\mu\llcorner\om$ is defined by
\[
V:=(\op{Id},A)_{\#}(\mu\llcorner\om),
\]
that is,
$V(E):=\mu(\{x\in\om:(x,A(x))\in E\})$ for any Borel set
$E\subset\bG_3(\om)$. We say that $V$ is stationary if
\[
\int_{\om}A_{ij}(x)\pa_i\xi_j(x)\ud\mu(x)=0
\quad\text{for any }\xi\in C_0^1(\om,\R^3).
\]

With this terminology, the following lemma gives the rectifiability of
$S_*\cap\om$ and the stationarity of the associated varifold.

\begin{lem}\label{stationvari}
The following properties hold.
\begin{enumerate}[label=$(\theenumi)$]
\item $S_*\cap\om$ is countably $\HH^1$-rectifiable,
$\HH^1(S_*\cap\om)<+\ift$, and
\[
\mu_*\llcorner\om=\Theta^1(\mu_*,\cdot)\HH^1\llcorner(S_*\cap\om).
\]
\item For $\mu_*$-a.e. $x\in\om$, let $A_*(x)\in\bG(3,1)$ be the orthogonal
projection onto $T_xS_*$. The varifold
$V_0:=(\op{Id},A_*)_{\#}(\mu_*\llcorner\om)$ is stationary.
\end{enumerate}
\end{lem}

\begin{proof}
The proof is based on the arguments of
\cite[Propositions~51~and~55]{Can17}.
\end{proof}

The next lemma is the main cylindrical estimate. It shows that the energy of
a minimizer in a thin cylinder is determined, up to a controlled error, by
the free homotopy class carried by the lateral boundary data.
The high-dimensional counterpart in \cite{GFW26} treats dimensions $n\geq4$,
where $B_{\delta}^2\times\pa B_1^{n-2}$ is connected. In the present
three-dimensional setting, this part reduces to the two caps
$B_{\delta}^2\times\{\pm1\}$, which must be handled separately. For this
reason and for completeness, we give the proof in the form needed below.

\begin{lem}\label{chooseseveralcases}
Let $\delta\in(0,\f{1}{2}]$. Assume that, for each
$\ol{\va}\in(0,1)$, the boundary datum
$g_{\delta,\ol{\va}}$ is defined on $\pa\Lda_{\delta,1}$, its restriction to
$\Ga_{\delta,1}$ belongs to $H^1(\Ga_{\delta,1},\R^m)$, and
\be
\begin{aligned}
\|g_{\delta,\ol{\va}}\|_{L^{\ift}(\pa\Lda_{\delta,1},\R^m)}
&\leq M,\\
E_{\ol{\va}}(g_{\delta,\ol{\va}},B_{\delta}^2\times\{\pm 1\})
&\leq M\log\f{\delta}{\ol{\va}},
\end{aligned}
\label{GvaboundaryH13M}
\ee
for some $M>0$. There exists $\eta\in(0,1)$, depending only on $f$, $M$, and
$\cN$, such that the following properties hold.
\begin{enumerate}[label=$(\theenumi)$]
\item For any $\ol{\va}\in(0,\eta\delta)$, if
\be
E_{\ol{\va}}(g_{\delta,\ol{\va}},\Ga_{\delta,1})
\leq\eta\log\f{\delta}{\ol{\va}},
\label{GvaboundaryH13}
\ee
and if $u_{\delta,\ol{\va}}$ is a minimizer of
$E_{\ol{\va}}(\cdot,\Lda_{\delta,1})$ in
$H^1(\Lda_{\delta,1},\R^m)$ with boundary condition
$u_{\delta,\ol{\va}}=g_{\delta,\ol{\va}}$ on $\pa\Lda_{\delta,1}$, then
there exists $\sg\in[\Ss^1,\cN]$ such that
\be
\left|E_{\ol{\va}}(u_{\delta,\ol{\va}},\Lda_{\delta,1})
-2|\sg|_*\log\f{\delta}{\ol{\va}}\right|
\leq\al(M,\delta,\eta)\log\f{\delta}{\ol{\va}}+C,
\label{EolvaapproximateEstar}
\ee
where $C>0$ depends only on $f$, $M$, and $\cN$, and
\[
\al(M,\delta,\eta)\leq C(\delta+\eta+\delta^{-1}\eta).
\]
\item If, moreover,
\[
\|\dist(g_{\delta,\ol{\va}},\cN)\|_{L^{\ift}(\Ga_{\delta,1})}<\eta,
\]
then the homotopy class $\sg$ in \eqref{EolvaapproximateEstar} satisfies
\be
\sg=[\Pi_{\cN}\circ g_{\delta,\ol{\va}}]_{\cN},
\label{sgfirstdefine}
\ee
where $\Pi_{\cN}$ is the nearest-point projection given by
Lemma~\ref{Nfproperty}.
\end{enumerate}
\end{lem}

\begin{proof}
Applying Lemma~\ref{Luckhauscylinder2}, we find
$\eta=\eta(f,M,\cN)\in(0,1)$ and points
\[
z_-\in\(-1+\f{\delta}{2},-1+\delta\),\quad
z_+\in\(1-\delta,1-\f{\delta}{2}\),
\]
together with maps
\[
v_{\delta,\ol{\va}}\in H^1(\pa B_{\delta}^2\times(z_-,z_+),\cN)
\]
and
\[
\vp_{\delta,\ol{\va}}\in
H^1((B_{\delta}^2\backslash B_{\f{\delta}{2}}^2)\times(z_-,z_+),\R^m),
\]
whose restrictions
\[
\vp_{\delta,\ol{\va}}^{\pm}:=
\vp_{\delta,\ol{\va}}|_{(B_{\delta}^2\backslash B_{\f{\delta}{2}}^2)
\times\{z_{\pm}\}}
\]
belong to
$H^1((B_{\delta}^2\backslash B_{\f{\delta}{2}}^2)\times\{z_{\pm}\},\R^m)$.
These maps satisfy \eqref{Wvaest1}, \eqref{Vvaes1}, and
\eqref{WpluesminusC11}. The homotopy class
\be
\sg:=[v_{\delta,\ol{\va}}]_{\cN}\in[\Ss^1,\cN]
\label{sgdefine1}
\ee
is well-defined by the analysis in \cite[Section~2.2]{Can17}.

\smallskip
\noindent\textbf{Step~1. Upper bound for
$E_{\ol{\va}}(u_{\delta,\ol{\va}},\Lda_{\delta,1})$.}
Define
\begin{align*}
\Lda_{\delta}^-&:=B_{\delta}^2\times(-1,z_-),&
\Lda_{\delta}^+&:=B_{\delta}^2\times(z_+,1),\\
\Lda_{\delta}^0&:=B_{\f{\delta}{2}}^2\times(z_-,z_+),&
E_{\delta}&:=(B_{\delta}^2\backslash B_{\f{\delta}{2}}^2)\times(z_-,z_+).
\end{align*}
Since
$v_{\delta,\ol{\va}}\in H^1(\pa B_{\delta}^2\times(z_-,z_+),\cN)$, we
apply Lemma~\ref{cylinderex}, after scaling the radial variable by a factor
$2$, to obtain
$\wh{u}_{\delta,\ol{\va}}^{(0)}\in H^1(\Lda_{\delta}^0,\R^m)$ such that
\[
\wh{u}_{\delta,\ol{\va}}^{(0)}(y,z)
=v_{\delta,\ol{\va}}(2y,z)
\quad\text{for }\HH^2\text{-a.e. }(y,z)\in
\pa B_{\f{\delta}{2}}^2\times(z_-,z_+),
\]
and
\begin{gather}
E_{\ol{\va}}(\wh{u}_{\delta,\ol{\va}}^{(0)},\Lda_{\delta}^0)
\leq(2|\sg|_*+C(\delta^{-1}+\delta)\eta)
\log\f{\delta}{\ol{\va}}+C,
\label{inteU0hat}\\
E_{\ol{\va}}(\wh{u}_{\delta,\ol{\va}}^{(0)},
B_{\f{\delta}{2}}^2\times\{z_{\pm}\})
\leq(|\sg|_*+C(\delta^{-1}+\delta)\eta)
\log\f{\delta}{\ol{\va}}+C.
\label{zplusmiunsU0}
\end{gather}
Define $g_{\delta,\ol{\va}}^{\pm}:\pa\Lda_{\delta}^{\pm}\to\R^m$ by
\[
g_{\delta,\ol{\va}}^{\pm}(x):=\left\{\begin{aligned}
&g_{\delta,\ol{\va}}(x)
&&\text{if }x\in\pa\Lda_{\delta}^{\pm}\cap\pa\Lda_{\delta,1},\\
&\vp_{\delta,\ol{\va}}^{\pm}(x)
&&\text{if }x\in(B_{\delta}^2\backslash B_{\f{\delta}{2}}^2)
\times\{z_{\pm}\},\\
&\wh{u}_{\delta,\ol{\va}}^{(0)}(x)
&&\text{if }x\in\ol{B}_{\f{\delta}{2}}^2\times\{z_{\pm}\}.
\end{aligned}\right.
\]
Then $g_{\delta,\ol{\va}}^{\pm}\in H^1(\pa\Lda_{\delta}^{\pm},\R^m)$. Using
\eqref{WpluesminusC11}, \eqref{GvaboundaryH13M}, \eqref{GvaboundaryH13}, and
\eqref{zplusmiunsU0}, we obtain
\be
\begin{aligned}
E_{\ol{\va}}(g_{\delta,\ol{\va}}^{\pm},\pa\Lda_{\delta}^{\pm})
&\leq E_{\ol{\va}}(\wh{u}_{\delta,\ol{\va}}^{(0)},
B_{\f{\delta}{2}}^2\times\{z_{\pm}\})+E_{\ol{\va}}(\vp_{\delta,\ol{\va}}^{\pm},
(B_{\delta}^2\backslash B_{\f{\delta}{2}}^2)\times\{z_{\pm}\})\\
&\quad+E_{\ol{\va}}(g_{\delta,\ol{\va}},
\pa\Lda_{\delta}^{\pm}\cap\Ga_{\delta,1})
+E_{\ol{\va}}(g_{\delta,\ol{\va}},B_{\delta}^2\times\{\pm1\})\\
&\leq C(1+(1+\delta^{-1}+\delta)\eta)
\log\f{\delta}{\ol{\va}}+C.
\end{aligned}
\label{Gvadeltaplusminus}
\ee
There exist bi-Lipschitz homeomorphisms
$\Phi^{\pm}:\ol{\Lda}_{\delta}^{\pm}\to\ol{B}_{\delta}^3$ whose
bi-Lipschitz constants are independent of $\delta$. We define
\[
\wh{u}_{\delta,\ol{\va}}^{\pm}(x)
:=g_{\delta,\ol{\va}}^{\pm}\circ(\Phi^{\pm})^{-1}
\circ\(\f{\delta\Phi^{\pm}(x)}{|\Phi^{\pm}(x)|}\)
\quad\text{for }x\in\Lda_{\delta}^{\pm}\backslash(\Phi^{\pm})^{-1}(0).
\]
From \eqref{Gvadeltaplusminus},
\be
E_{\ol{\va}}(\wh{u}_{\delta,\ol{\va}}^{\pm},\Lda_{\delta}^{\pm})
\leq C\delta E_{\ol{\va}}(g_{\delta,\ol{\va}}^{\pm},\pa\Lda_{\delta}^{\pm})
\leq C(\delta+(\delta+\delta^2+1)\eta)
\log\f{\delta}{\ol{\va}}+C.
\label{U0plusminusinmte}
\ee
We now define the global competitor
\[
\wh{u}_{\delta,\ol{\va}}(x):=\left\{\begin{aligned}
&\wh{u}_{\delta,\ol{\va}}^{(0)}(x)
&&\text{if }x\in\Lda_{\delta}^0,\\
&\vp_{\delta,\ol{\va}}(x)
&&\text{if }x\in E_{\delta},\\
&\wh{u}_{\delta,\ol{\va}}^{\pm}(x)
&&\text{if }x\in\Lda_{\delta}^{\pm}.
\end{aligned}\right.
\]
By \eqref{Wvaest1}, \eqref{inteU0hat}, and \eqref{U0plusminusinmte},
\[
E_{\ol{\va}}(\wh{u}_{\delta,\ol{\va}},\Lda_{\delta,1})
\leq(2|\sg|_*+C(\delta+\eta+\delta^{-1}\eta))
\log\f{\delta}{\ol{\va}}+C.
\]
The upper bound follows from the minimality of $u_{\delta,\ol{\va}}$:
\[
E_{\ol{\va}}(u_{\delta,\ol{\va}},\Lda_{\delta,1})
\leq E_{\ol{\va}}(\wh{u}_{\delta,\ol{\va}},\Lda_{\delta,1}).
\]

\smallskip
\noindent\textbf{Step~2. Lower bound for
$E_{\ol{\va}}(u_{\delta,\ol{\va}},\Lda_{\delta,1})$.}
In cylindrical coordinates
$(\rho,\theta,z)\in[0,\delta]\times\Ss^1\times[-1,1]$, define
\[
\wh{u}_{\delta,\ol{\va}}^{(1)}(\rho,\theta,z):=
\left\{\begin{aligned}
&u_{\delta,\ol{\va}}(2\rho,\theta,z)
&&\text{if }\rho\in[0,\tfrac{\delta}{2})\text{ and }z\in(z_-,z_+),\\
&\vp_{\delta,\ol{\va}}(\tfrac{3\delta}{2}-\rho,\theta,z)
&&\text{if }\rho\in[\tfrac{\delta}{2},\delta)\text{ and }z\in(z_-,z_+).
\end{aligned}\right.
\]
By construction,
$\wh{u}_{\delta,\ol{\va}}^{(1)}\in
H^1(B_{\delta}^2\times(z_-,z_+),\R^m)$ and
\be
\wh{u}_{\delta,\ol{\va}}^{(1)}(x)=v_{\delta,\ol{\va}}(x)
\quad\text{for }\HH^2\text{-a.e. }x\in
\pa B_{\delta}^2\times(z_-,z_+).
\label{boundaryconUV2}
\ee
A direct computation, together with \eqref{Wvaest1}, gives
\be
\begin{aligned}
E_{\ol{\va}}(\wh{u}_{\delta,\ol{\va}}^{(1)},
B_{\delta}^2\times(z_-,z_+))
&\leq E_{\ol{\va}}(u_{\delta,\ol{\va}},\Lda_{\delta,1})+E_{\ol{\va}}(\vp_{\delta,\ol{\va}},
(B_{\delta}^2\backslash B_{\f{\delta}{2}}^2)\times(z_-,z_+))\\
&\leq E_{\ol{\va}}(u_{\delta,\ol{\va}},\Lda_{\delta,1})
+C\eta\log\f{\delta}{\ol{\va}}.
\end{aligned}
\label{U1Udeltaminus}
\ee
By Fubini's theorem, Lemma~\ref{LowerBound}, and \eqref{boundaryconUV2},
\[
E_{\ol{\va}}(\wh{u}_{\delta,\ol{\va}}^{(1)},B_{\delta}^2\times\{z\})
+C\delta E_{\ol{\va}}(v_{\delta,\ol{\va}},
\pa B_{\delta}^2\times\{z\})
\geq|\sg|_*\log\f{\delta}{\ol{\va}}-C
\]
for a.e. $z\in(z_-,z_+)$. Integrating over $z\in(z_-,z_+)$ and using
$z_+-z_-\geq2-\delta$, we obtain
\[
E_{\ol{\va}}(\wh{u}_{\delta,\ol{\va}}^{(1)},
B_{\delta}^2\times(z_-,z_+))
+C\delta\int_{\pa B_{\delta}^2\times(z_-,z_+)}
|\na v_{\delta,\ol{\va}}|^2\ud\HH^2
\geq(2|\sg|_*-C\delta)\log\f{\delta}{\ol{\va}}-C.
\]
Combining this estimate with \eqref{Vvaes1}, \eqref{GvaboundaryH13}, and
\eqref{U1Udeltaminus}, we get
\[
E_{\ol{\va}}(u_{\delta,\ol{\va}},\Lda_{\delta,1})
\geq(2|\sg|_*-C(\delta+\eta))\log\f{\delta}{\ol{\va}}-C.
\]
This proves the lower bound and completes the proof of part~$(1)$.

For part~$(2)$, the additional assumption
$\|\dist(g_{\delta,\ol{\va}},\cN)\|_{L^{\ift}(\Ga_{\delta,1})}<\eta$ allows
us to apply Lemma~\ref{closeprojlem} and choose
$v_{\delta,\ol{\va}}=\Pi_{\cN}\circ g_{\delta,\ol{\va}}$ on
$\Ga_{\delta,1}$. Hence \eqref{sgfirstdefine} follows directly from
\eqref{sgdefine1}.
\end{proof}

The preceding cylindrical estimate gives the quantization of the density of
the limiting measure at almost any interior point of the singular set.

\begin{lem}\label{densitydescrete}
For $\HH^1$-a.e. $x\in S_*\cap\om$,
\[
\Theta^1(\mu_*,x)\in\{|\sg|_*:\sg\in[\Ss^1,\cN]\}\backslash\{0\}.
\]
\end{lem}

\begin{proof}
This follows from Lemma~\ref{chooseseveralcases} by the blow-up argument of
\cite[Proposition~59]{Can17}.
\end{proof}

\subsubsection{Conclusions of the interior analysis}

We first record the interior structure of the limiting singular set. The following proposition gives a precise geometric description of $S_*$ near any compact subset of $\om$.

\begin{prop}\label{interiorproperty1}
Let $K\subset\om$ be a compact set such that $S_*\cap K\neq\emptyset$. Then there exists an open set $U\subset\subset\om$ whose connected components are Lipschitz domains and such that the following properties hold.
\begin{enumerate}[label=$(\theenumi)$]
\item There exist closed line segments $\{L_i\}_{i=1}^k$ with $\HH^1(L_i)>0$ for any $i\in\Z\cap[1,k]$ such that
\[
S_*\cap\ol{U}=\bigcup_{i=1}^k L_i.
\]
Moreover, $\op{Int}(L_i)\cap\op{Int}(L_j)=\emptyset$ for any $i\neq j$, where $\op{Int}(L_i)$ denotes the relative interior of $L_i$.
\item If $K$ is connected, then $U$ is connected.
\item $S_*\cap\pa U$ is a finite set. For any $x\in S_*\cap\pa U$, there is a unique $L\in\{L_i\}_{i=1}^k$ such that $x$ is an endpoint of $L$.
\end{enumerate}
\end{prop}

\begin{proof}
The proof follows the argument of \cite[Proposition~5.17]{BSV25}, using Lemma~\ref{stationvari}, Lemma~\ref{densitydescrete}, and the structure theorem for $1$-dimensional stationary varifolds \cite[Theorem~p.89]{AA76}.
\end{proof}

The next proposition identifies the homotopy class carried by each interior segment of $S_*$ and gives the balance condition at the interior branching points.

\begin{prop}\label{interiorproperty11}
Let $u_*$ and $S_0$ be as in Proposition~\ref{propustar}. The following properties hold.
\begin{enumerate}[label=$(\theenumi)$]
\item Let $L\subset S_*$ be a closed line segment, let $x_0\in\op{Int}(L)$, and let $D\subset\om$ be a closed $2$-disk centered at $x_0$ such that
\[
D\cap S_*=\{x_0\}\quad\text{and}\quad\pa D\cap S_0=\emptyset.
\]
Then the homotopy class $\sg:=[u_*|_{\pa D}]_{\cN}\in[\Ss^1,\cN]$ is non-trivial and, for any $y\in\op{Int}(L)$,
\be
\Theta^1(\mu_*,y)=|\sg|_*.\label{Theta1muy}
\ee
\item Let $x_0\in S_*\cap\om$. There exists $r>0$ such that $B_r(x_0)\subset\subset\om$ and
\[
S_*\cap\ol{B}_r(x_0)=\bigcup_{i=1}^k L_i,
\]
where $k\in\N$ and each $L_i$ is a closed line segment with $x_0\in L_i$. Let $v_i\in\Ss^2$ be the unit direction vector of $L_i$ pointing outward from $x_0$. Then
\[
\sum_{i=1}^k |\sg_i|_*v_i=0,
\]
where $\sg_i:=[u_*|_{\pa D_i}]_{\cN}$, and $D_i$ is a closed $2$-disk centered at some $x_i\in\op{Int}(L_i)\cap B_r(x_0)$ and lying in the $2$-plane orthogonal to $L_i$, with
\[
D_i\cap S_*=\{x_i\}\quad\text{and}\quad\pa D_i\cap S_0=\emptyset.
\]
\end{enumerate}
\end{prop}

\begin{proof}
We first prove $(1)$. By a rigid motion, we may assume that
$D=B_{r_0}^2\times\{0\}$ and
$L=\{0^2\}\times[-4h_0,4h_0]$, where $0^2$ denotes the origin of $\R^2$, and that
\be
S_*\cap(D\times[-4h_0,4h_0])=L\quad\text{and}\quad
S_*\cap(\pa D\times[-4h_0,4h_0])=\emptyset.\label{cSlinecSpts1}
\ee
Recalling \eqref{assumptionbound} and \eqref{assumptionbound1}, we have
\be
E_{\va_i}(u_{\va_i},\om)\leq M(|\log\va_i|+1)
\quad\text{and}\quad
\|u_{\va_i}\|_{L^{\ift}(\om)}\leq M.\label{recallEva}
\ee
By Fubini's theorem, there exists $h\in(h_0,2h_0)$ such that, for any $r\in(0,r_0)$ and all sufficiently large $i$,
\be
E_{\va_i}(u_{\va_i},B_r^2\times\{\pm h\})
\leq E_{\va_i}(u_{\va_i},B_{r_0}^2\times\{\pm h\})
\leq C_0\log\f{r_0}{\va_i},\label{uplowh}
\ee
where $C_0=C_0(M,h_0)>0$. Since $S_0$ is locally finite in $\om\backslash S_*$ by Proposition~\ref{propustar}, for $\HH^1$-a.e. $r\in(0,r_0)$,
\be
S_0\cap\ol{\Ga}_{r,h}=\emptyset\quad\text{and}\quad
S_*\cap\ol{\Lda}_{r,h}=\{0^2\}\times[-h,h],\label{cSptscSline}
\ee
where $\Lda_{r,h}=B_r^2\times(-h,h)$ and $\Ga_{r,h}=\pa B_r^2\times(-h,h)$.

Fix such an $r$. Set $\delta:=\f{r}{h}$, $\ol{\va}_i:=\f{\va_i}{h}$, and $u_{\delta,\ol{\va}_i}(x):=u_{\va_i}(hx)$ for $x\in\Lda_{\delta,1}$. By Proposition~\ref{propustar}, there exists a sequence $\eta_i\to 0^+$ such that
\[
E_{\va_i}(u_{\va_i},\Ga_{r,h})\leq\eta_i\log\f{r}{\va_i}.
\]
Together with \eqref{recallEva} and \eqref{uplowh}, this gives
\begin{gather*}
\|u_{\delta,\ol{\va}_i}\|_{L^{\ift}(\pa\Lda_{\delta,1})}\leq M,\\
E_{\ol{\va}_i}(u_{\delta,\ol{\va}_i},\Ga_{\delta,1})
\leq\f{\eta_i}{h_0}\log\f{\delta}{\ol{\va}_i},\\
E_{\ol{\va}_i}(u_{\delta,\ol{\va}_i},B_{\delta}^2\times\{\pm 1\})
\leq C_0\log\f{\delta}{\ol{\va}_i}.
\end{gather*}
Applying Lemma~\ref{chooseseveralcases}, we obtain a class $\sg\in[\Ss^1,\cN]$ such that, for all sufficiently large $i$,
\be
\left|E_{\ol{\va}_i}(u_{\delta,\ol{\va}_i},\Lda_{\delta,1})
-2|\sg|_*\log\f{\delta}{\ol{\va}_i}\right|
\leq\wt{C}(\delta+\eta_i+\delta^{-1}\eta_i)\log\f{\delta}{\ol{\va}_i}+\wt{C},
\label{wechooseuse}
\ee
where $\wt{C}>0$ depends only on $f,M,\cN$, and $h_0$.

Since $r$ satisfies \eqref{cSptscSline}, we have $u_{\va_i}\to u_*$ uniformly near $\ol{\Ga}_{r,h}$, and $u_*$ takes values in $\cN$ there. Lemma~\ref{closeprojlem} gives
\be
\sg=[\Pi_{\cN}\circ u_{\va_i}|_{\pa B_r^2\times\{0\}}]_{\cN}
=[u_*|_{\pa B_r^2\times\{0\}}]_{\cN}
=[v_{\delta,\ol{\va}_i}]_{\cN}.\label{rHequal}
\ee
Since $u_*\in C_{\loc}^{\ift}(\om\backslash(S_*\cup S_0),\cN)$, conditions \eqref{cSlinecSpts1} and \eqref{cSptscSline} imply that $u_*$ gives a homotopy between the loops $u_*|_{\pa B_r^2\times\{0\}}$ and $u_*|_{\pa B_{r_0}^2\times\{0\}}$. Hence \eqref{rHequal} yields $ \sg=[u_*|_{\pa B_{r_0}^2\times\{0\}}]_{\cN} $. Recalling that $\delta=\f{r}{h}$ and $\eta_i\to 0^+$, for any $\al_0>0$ we may fix $h$ and then choose $r>0$ sufficiently small so that, for all sufficiently large $i$,
\[
\left|E_{\ol{\va}_i}(u_{\delta,\ol{\va}_i},\Lda_{\delta,1})
-2|\sg|_*\log\f{\delta}{\ol{\va}_i}\right|
\leq\al_0\log\f{\delta}{\ol{\va}_i}.
\]
Rescaling back to the original variables and dividing by the length $2h$ of the cylinder in the axial direction gives
\[
\f{1}{2h}\lim_{i\to+\ift}
\f{E_{\va_i}(u_{\va_i},\Lda_{r,h})}{\log\f{1}{\va_i}}
=|\sg|_*.
\]
Letting $r\to 0^+$ along admissible radii gives \eqref{Theta1muy} at $x_0$. Since $x_0\in S_*$, the density is positive; hence $|\sg|_*>0$, and $\sg$ is non-trivial. The same argument, applied after translating the cylinder along $L$, gives \eqref{Theta1muy} for any $y\in\op{Int}(L)$.

Part~$(2)$ follows by the same argument as \cite[Corollary~5.19]{BSV25}, using the stationarity of $\mu_*\llcorner\om$ established in Lemma~\ref{stationvari}.
\end{proof}

\subsection{The singular set near the boundary}

\subsubsection{Preliminary results near the boundary}

We next prove the boundary estimates needed to control the first variation of the limiting stress tensor near $\pa\om$. The first estimate bounds the full boundary gradient by the interior energy and the tangential boundary energy.

\begin{prop}\label{EstimatesOntheboundary}
Let $U\subset\R^3$ be a bounded $C^{2,1}$ domain with parameters $M_{U,2}$ and $r_{U,2}$. Let $x_0\in\pa U$ and $r\in(0,r_{U,2})$. Assume that $u\in C^{\ift}(U_r(x_0),\R^m)\cap C^1(\ol{U}_r(x_0),\R^m)$ satisfies \eqref{EL} and that $u(x)\in\cN$ for any $x\in T_r^U(x_0)$. Then, for any relatively open subset $V\subset\subset T_r^U(x_0)$, there exists $C>0$, depending only on $V$, $M_{U,2}$, and $r_{U,2}$, such that
\be
\int_V|\na u|^2\ud\HH^2
\leq C\left(\int_{U_r(x_0)}e_{\va}(u)\ud x
+\int_{T_r^U(x_0)}|\na_{\pa U}u|^2\ud\HH^2\right).\label{nablauK}
\ee
\end{prop}

The proof depends on two elementary identities for the stress tensor. We first record the integration-by-parts formula.

\begin{lem}\label{integrationbyparts}
Let $U\subset\R^3$ be a bounded $C^{2,1}$ domain with parameters $M_{U,2}$ and $r_{U,2}$. Let $x_0\in\pa U$ and $r\in(0,r_{U,2})$. Assume that $u\in C^{\ift}(U_r(x_0),\R^m)\cap C^1(\ol{U}_r(x_0),\R^m)$ satisfies \eqref{EL}. Then, for any $\xi\in C_0^1(B_r(x_0),\R^3)$,
\[
\int_{U_r(x_0)}\<T_{\va}(u),D\xi\>\ud x
=\int_{T_r^U(x_0)}T_{\va}(u)[\nu,\xi]\ud\HH^2,
\]
where $T_{\va}(u)$ is defined by \eqref{stress}, $\nu$ is the unit outer normal to $\pa U$, and
\[
T_{\va}(u)[\nu,\xi]:=T_{\va}^{ij}(u)\nu_i\xi_j,
\quad
\<T_{\va}(u),D\xi\>:=T_{\va}^{ij}(u)\pa_i\xi_j.
\]
\end{lem}

\begin{proof}
The identity follows from \eqref{stressidentity} and integration by parts. The compact support of $\xi$ in $B_r(x_0)$ removes the contribution from the spherical part of $\pa U_r(x_0)$.
\end{proof}

The next identity expresses the normal derivative on the boundary in terms of the normal component of the stress tensor.

\begin{lem}\label{partialnuestimates}
Under the same assumptions as Proposition~\ref{EstimatesOntheboundary},
\[
|\pa_{\nu}u|^2=-2T_{\va}(u)[\nu,\nu]+|\na_{\pa U}u|^2
\quad\text{on }T_r^U(x_0),
\]
where $\nu$ is the unit outer normal to $\pa U$.
\end{lem}

\begin{proof}
On $T_r^U(x_0)$, we have $u\in\cN$, and hence $f(u)=0$. Since
$|\na u|^2=|\na_{\pa U}u|^2+|\pa_{\nu}u|^2$, the definition \eqref{stress} gives
\[
2T_{\va}(u)[\nu,\nu]
=|\na u|^2+\f{2}{\va^2}f(u)-2|\pa_{\nu}u|^2
=|\na_{\pa U}u|^2-|\pa_{\nu}u|^2.
\]
This is the desired identity.
\end{proof}

\begin{proof}[Proof of Proposition~\ref{EstimatesOntheboundary}]
Choose $\theta\in C^1(\pa U,[0,1])$ such that $\theta\equiv 1$ on $V$ and $\theta\equiv 0$ on $\pa U\backslash B_r(x_0)$. Let $\xi\in C_0^1(B_r(x_0),\R^3)$ satisfy $\xi=\theta\nu$ on $\pa U$. By the $C^{2,1}$ regularity of $\pa U$, we may choose $\xi$ so that
$\|D\xi\|_{L^{\ift}(U_r(x_0))}\leq C$, where $C=C(V,M_{U,2},r_{U,2})>0$. Lemma~\ref{integrationbyparts} gives
\[
\int_{T_r^U(x_0)}\theta T_{\va}(u)[\nu,\nu]\ud\HH^2
=\int_{U_r(x_0)}\<T_{\va}(u),D\xi\>\ud x.
\]
Since $|T_{\va}(u)|\leq Ce_{\va}(u)$, we obtain
\be
\left|\int_{T_r^U(x_0)}\theta T_{\va}(u)[\nu,\nu]\ud\HH^2\right|
\leq C\int_{U_r(x_0)}e_{\va}(u)\ud x.\label{thetaTvau}
\ee
Set
\[
E:=\{x\in T_r^U(x_0):|\pa_{\nu}u(x)|^2\geq|\na_{\pa U}u(x)|^2\}.
\]
By Lemma~\ref{partialnuestimates}, $-T_{\va}(u)[\nu,\nu]\geq 0$ on $E$, and hence
\be
-\int_{V\cap E}T_{\va}(u)[\nu,\nu]\ud\HH^2
\leq -\int_E\theta T_{\va}(u)[\nu,\nu]\ud\HH^2.\label{KcapEleq}
\ee
On $T_r^U(x_0)\backslash E$, Lemma~\ref{partialnuestimates} gives
\be
0\leq T_{\va}(u)[\nu,\nu]\leq\f{1}{2}|\na_{\pa U}u|^2.\label{backE}
\ee
Using \eqref{thetaTvau} and \eqref{backE}, we have
\[
-\int_E\theta T_{\va}(u)[\nu,\nu]\ud\HH^2
\leq C\int_{U_r(x_0)}e_{\va}(u)\ud x
+\f{1}{2}\int_{T_r^U(x_0)}|\na_{\pa U}u|^2\ud\HH^2.
\]
We now compute, using Lemma~\ref{partialnuestimates} and then \eqref{KcapEleq}--\eqref{backE},
\begin{align*}
\int_V|\na u|^2\ud\HH^2
&=\int_V(|\na_{\pa U}u|^2+|\pa_{\nu}u|^2)\ud\HH^2\\
&=2\int_V|\na_{\pa U}u|^2\ud\HH^2
  -2\int_VT_{\va}(u)[\nu,\nu]\ud\HH^2\\
&\leq -2\int_{V\cap E}T_{\va}(u)[\nu,\nu]\ud\HH^2
  +2\int_V|\na_{\pa U}u|^2\ud\HH^2\\
&\leq C\int_{U_r(x_0)}e_{\va}(u)\ud x
  +C\int_{T_r^U(x_0)}|\na_{\pa U}u|^2\ud\HH^2.
\end{align*}
This proves \eqref{nablauK}.
\end{proof}

The next two lemmas estimate $\op{div}T_{\va}$ in the functional sense. They will be used to reach the weak${}^*$ limit of normalized stress tensors.

\begin{lem}\label{xiLinfty}
Under the same assumptions as Proposition~\ref{EstimatesOntheboundary}, for any $\xi\in C_0^1(B_r(x_0),\R^3)$,
\[
\left|\int_{U_r(x_0)}\<T_{\va}(u),D\xi\>\ud x\right|
\leq C\|\xi\|_{L^{\ift}(T_r^U(x_0),\R^3)}
\int_{T_r^U(x_0)}|\na u|^2\ud\HH^2,
\]
where $C>0$ is an absolute constant.
\end{lem}

\begin{proof}
By Lemma~\ref{integrationbyparts},
\[
\left|\int_{U_r(x_0)}\<T_{\va}(u),D\xi\>\ud x\right|
\leq \int_{T_r^U(x_0)}|T_{\va}(u)[\nu,\xi]|\ud\HH^2.
\]
Since $u\in\cN$ on $T_r^U(x_0)$, we have $f(u)=0$ there and therefore
$|T_{\va}(u)[\nu,\xi]|\leq C|\xi||\na u|^2$ on $T_r^U(x_0)$. The estimate follows directly.
\end{proof}

\begin{lem}\label{lowerboundxi}
Under the same assumptions as Proposition~\ref{EstimatesOntheboundary}, for any $\xi\in C_0^1(B_r(x_0),\R^3)$ with $\xi\cdot\nu\leq 0$ on $T_r^U(x_0)$,
\[
\int_{U_r(x_0)}\<T_{\va}(u),D\xi\>\ud x
\geq -\|\xi\|_{L^{\ift}(T_r^U(x_0),\R^3)}
\int_{T_r^U(x_0)}\left(\f{1}{2}|\na_{\pa U}u|^2
+|\na_{\pa U}u||\pa_{\nu}u|\right)\ud\HH^2,
\]
where $\nu\in C^1(\pa U,\Ss^2)$ is the unit outer normal to $\pa U$.
\end{lem}

\begin{proof}
On $T_r^U(x_0)$, we decompose
\[
T_{\va}(u)[\nu,\xi]
=(\nu\cdot\xi)T_{\va}(u)[\nu,\nu]
-\bigl[(\xi-(\nu\cdot\xi)\nu)\cdot\na_{\pa U}u\bigr]\cdot\pa_{\nu}u.
\]
By Lemma~\ref{partialnuestimates},
\[
(\nu\cdot\xi)T_{\va}(u)[\nu,\nu]
=\f{1}{2}(\nu\cdot\xi)(|\na_{\pa U}u|^2-|\pa_{\nu}u|^2).
\]
Since $\nu\cdot\xi\leq 0$ and
$|\na_{\pa U}u|^2-|\pa_{\nu}u|^2\leq |\na_{\pa U}u|^2$, we obtain
\[
(\nu\cdot\xi)T_{\va}(u)[\nu,\nu]
\geq -\f{1}{2}\|\xi\|_{L^{\ift}(T_r^U(x_0),\R^3)}|\na_{\pa U}u|^2.
\]
The tangential part satisfies
\[
\left|\bigl[(\xi-(\nu\cdot\xi)\nu)\cdot\na_{\pa U}u\bigr]\cdot\pa_{\nu}u\right|
\leq \|\xi\|_{L^{\ift}(T_r^U(x_0),\R^3)}
|\na_{\pa U}u||\pa_{\nu}u|.
\]
Combining the two bounds and integrating over $T_r^U(x_0)$ by Lemma~\ref{integrationbyparts} completes the proof.
\end{proof}

We can now pass to the weak${}^*$ limit of the normalized stress tensors and retain the boundary sign information needed later.

\begin{prop}\label{divTstar}
Let $U\subset\R^3$ be a bounded $C^{2,1}$ domain with parameters $M_{U,2}$ and $r_{U,2}$. Let $x_0\in\pa U$ and $r\in(0,\f{r_{U,2}}{2})$. Assume that $\{u_{\va}\}_{\va\in(0,1)}\subset C^{\ift}(U_{2r}(x_0),\R^m)\cap C^1(\ol{U}_{2r}(x_0),\R^m)$ is a family of solutions of \eqref{EL} with $u_{\va}(x)\in\cN$ for any $x\in T_{2r}^U(x_0)$. Suppose that
\begin{gather*}
\sup_{\va\in(0,1)}\f{1}{|\log\va|}\int_{U_{2r}(x_0)}e_{\va}(u_{\va})\ud x
\leq M_1,\\
\sup_{\va\in(0,1)}\f{1}{|\log\va|}
\int_{T_{2r}^U(x_0)}|\na_{\pa U}u_{\va}|^2\ud\HH^2\leq M_2.
\end{gather*}
Then there exists a subsequence $\va_i\to 0^+$ such that
\[
\f{T_{\va_i}(u_{\va_i})}{|\log\va_i|}\ud x
\wc^*T_*\quad\text{weakly${}^*$ in }
(C^0(\ol{U}_{2r}(x_0),\R^3\otimes\R^3))',
\]
and the following properties hold.
\begin{enumerate}[label=$(\theenumi)$]
\item For any $\xi\in C_0^1(B_r(x_0),\R^3)$,
\be
\left|\int_{U_r(x_0)}\<T_*,D\xi\>\ud x\right|
\leq C\|\xi\|_{L^{\ift}(T_r^U(x_0),\R^3)},\label{Tstar1}
\ee
where $C>0$ depends only on $M_{U,2}$, $r_{U,2}$, $M_1$, and $M_2$.
\item If, in addition,
\be
\lim_{i\to+\ift}\f{1}{|\log\va_i|}
\int_{T_{2r}^U(x_0)}|\na_{\pa U}u_{\va_i}|^2\ud\HH^2=0,\label{Tstar0property}
\ee
then, for any $\xi\in C_0^1(B_r(x_0),\R^3)$ with $\xi\cdot\nu\leq 0$ on $T_r^U(x_0)$,
\be
\int_{U_r(x_0)}\<T_*,D\xi\>\ud x\geq 0.\label{Tstar2}
\ee
\end{enumerate}
\end{prop}

\begin{proof}
Since $|T_{\va}(u_{\va})|\leq Ce_{\va}(u_{\va})$, the first energy bound gives a uniform bound on the total variation of
$\f{T_{\va}(u_{\va})}{|\log\va|}\ud x$ in $\ol{U}_{2r}(x_0)$. Thus, after passing to a subsequence, these tensor-valued Radon measures converge weakly${}^*$ to a limit $T_*$. It remains to prove \eqref{Tstar1} and \eqref{Tstar2}.

\textit{Proof of \eqref{Tstar1}.} Since $r<\f{r_{U,2}}{2}$, we may apply Proposition~\ref{EstimatesOntheboundary} with radius $2r$ and with $V=T_r^U(x_0)$. This yields
\[
\int_{T_r^U(x_0)}|\na u_{\va_i}|^2\ud\HH^2
\leq C\left(\int_{U_{2r}(x_0)}e_{\va_i}(u_{\va_i})\ud x
+\int_{T_{2r}^U(x_0)}|\na_{\pa U}u_{\va_i}|^2\ud\HH^2\right).
\]
By Lemma~\ref{xiLinfty},
\[
\left|\int_{U_r(x_0)}\<T_{\va_i}(u_{\va_i}),D\xi\>\ud x\right|
\leq C(M_1+M_2)|\log\va_i|
\|\xi\|_{L^{\ift}(T_r^U(x_0),\R^3)}.
\]
Dividing by $|\log\va_i|$ and passing to the limit $i\to+\ift$ gives \eqref{Tstar1}.

\textit{Proof of \eqref{Tstar2}.} The preceding boundary estimate also gives
\[
\sup_i\f{1}{|\log\va_i|}
\int_{T_r^U(x_0)}|\pa_{\nu}u_{\va_i}|^2\ud\HH^2
\leq C(M_1+M_2).
\]
Combining this bound with \eqref{Tstar0property} and the Cauchy--Schwarz inequality, we obtain
\[
\lim_{i\to+\ift}\f{1}{|\log\va_i|}
\int_{T_r^U(x_0)}
\left(|\na_{\pa U}u_{\va_i}|^2
+|\na_{\pa U}u_{\va_i}||\pa_{\nu}u_{\va_i}|\right)
\ud\HH^2=0.
\]
Dividing the inequality in Lemma~\ref{lowerboundxi} by $|\log\va_i|$ and passing to the limit gives \eqref{Tstar2}.
\end{proof}

\subsubsection{Properties of the singular set up to the boundary}

We finally collect the structural properties of the limiting singular set up to the boundary.

\begin{prop}\label{boundarySstarbehave}
Under the same assumptions as Theorem~\ref{globalminimizersproperties}, the following properties hold.
\begin{enumerate}[label=$(\theenumi)$]
\item\label{boundary1} $S_*$ is contained in $\op{Conv}(\cA)$, the convex hull of $\cA$, and $S_*\cap\pa\om=\cA$.
\item\label{boundary11} Let $u_*$ and $S_0$ be as in Proposition~\ref{propustar}. There exist finitely many closed line segments $\{L_i\}_{i=1}^n\subset\ol{\om}$ such that
\[
S_*=\bigcup_{i=1}^n L_i.
\]
\item\label{boundary10} Let $i\in\Z\cap[1,n]$, and let $x$ be an endpoint of $L_i\subset S_*$. Then exactly one of the following alternatives holds.
\begin{enumerate}[label=$(\op{\alph*})$]
\item\label{boundary10a} $x$ is a branching point of $S_*$ lying in $\om$.
\item\label{boundary10b} $x\in\cA$. Moreover,
\be
[g|_{\pa D_x}]_{\cN}=[u_*|_{\pa D}]_{\cN},\label{gDxeq}
\ee
where $D_x$ is a closed geodesic disk on $\pa\om$ centered at $x$ with $D_x\cap\cA=\{x\}$, and $D\subset\om$ is a closed $2$-disk centered at a point of $\op{Int}(L_i)$ such that $ \pa D\cap(\op{Conv}(\cA)\cup S_0)=\emptyset $.
\end{enumerate}
\end{enumerate}
\end{prop}

\begin{proof}[Proof of Proposition~\ref{boundarySstarbehave}\ref{boundary1}]
Throughout the proof, $C>0$ denotes a constant depending only on
$M_0,\om,f$, and $\cN$, and it can change from line to line.

\smallskip
\noindent\textbf{Part 1. Proof that $S_*\cap\pa\om\subset\cA$.}

Fix $x_0\in\pa\om\backslash\cA$. We choose
$r\in(0,\f{1}{10})$ such that $\cA\cap\ol{B}_r(x_0)=\emptyset$. Since the
boundary data are suitable and no point of $\cA$ lies in $\ol{B}_r(x_0)$, we
have
\be
g_{\va_i}\in\cN\quad\text{on }T_r^{\om}(x_0)
\label{gvaNawayA}
\ee
and
\be
\|g_{\va_i}\|_{L^{\ift}(T_r^{\om}(x_0))}
+r\|\na_{\top}g_{\va_i}\|_{L^{\ift}(T_r^{\om}(x_0))}\leq C,
\label{gvaassuse}
\ee
with $C$ independent of $i$. In particular,
\be
\lim_{i\to+\ift}\f{1}{|\log\va_i|}
\int_{T_r^{\om}(x_0)}|\na_{\top}g_{\va_i}|^2\ud\HH^2=0.
\label{tangentialenergyawayA}
\ee

We first pass to the weak limit of the normalized stress tensors. A direct
calculation gives
\begin{align}
\left|\f{T_{\va_i}(u_{\va_i})}{|\log\va_i|}\right|
&\leq C\f{e_{\va_i}(u_{\va_i})}{|\log\va_i|},
\label{Tva1}\\
\tr\left(\f{T_{\va_i}(u_{\va_i})}{|\log\va_i|}\right)
&=\f{1}{|\log\va_i|}
\left(3e_{\va_i}(u_{\va_i})-|\na u_{\va_i}|^2\right)
\geq\f{e_{\va_i}(u_{\va_i})}{|\log\va_i|},
\label{Tva11}
\end{align}
and, for any unit vector $v\in\Ss^2$,
\be
\f{T_{\va_i}^{jk}(u_{\va_i})v_jv_k}{|\log\va_i|}
=\f{e_{\va_i}(u_{\va_i})-|v\cdot\na u_{\va_i}|^2}{|\log\va_i|}
\leq\f{e_{\va_i}(u_{\va_i})}{|\log\va_i|}.
\label{Tva111}
\ee
By \eqref{assumptionbound} and \eqref{Tva1}, after passing to a subsequence,
there exists a Radon tensor-valued measure $T_*$ such that
\[
\f{T_{\va_i}(u_{\va_i})}{|\log\va_i|}\ud x
\wc^*T_*
\quad\text{weakly${}^*$ in }
(C^0(\ol{\om}\cap\ol{B}_r(x_0),\R^3\otimes\R^3))'.
\]
We restrict both $T_*$ and $\mu_*$ to $B_r(x_0)$ and keep the same notation:
\[
T_*:=T_*\llcorner B_r(x_0),\quad
\mu_*:=\mu_*\llcorner B_r(x_0).
\]

We next record the density lower bound near this boundary point. By
\eqref{Monotone11} and \eqref{assumptionbound}, for any
$x\in\ol{\om}\cap B_r(x_0)$ and any
$0<s<t<\f{r-|x-x_0|}{10}$,
\[
\f{\theta_{\va_i}^{\om}(u_{\va_i};x,s)}{|\log\va_i|}
\leq
\f{\theta_{\va_i}^{\om}(u_{\va_i};x,t)}{|\log\va_i|}
+C(t-s).
\]
Passing to the limit $i\to+\ift$ gives
\[
\f{\mu_*(\ol{B}_s(x))}{2s}
\leq
\f{\mu_*(\ol{B}_t(x))}{2t}+C(t-s).
\]
Hence
$\Theta^1(\mu_*,x)$ exists for any $x\in\ol{\om}\cap B_r(x_0)$.

We claim that there exists $\delta>0$, depending only on the fixed data, such
that
\be
S_*\cap B_r(x_0)
=\{x\in\ol{\om}\cap B_r(x_0):\Theta^1(\mu_*,x)>0\}
=\{x\in\ol{\om}\cap B_r(x_0):\Theta^1(\mu_*,x)\geq\delta\}.
\label{claimSstar}
\ee
For interior points, this follows from Lemma~\ref{Thetamu0eta0}. It remains
to consider points on $T_r^{\om}(x_0)$. If
$x\in T_r^{\om}(x_0)$ satisfies $\Theta^1(\mu_*,x)>0$ but $x\notin S_*$,
then $\mu_*(B_{\rho}(x))=0$ for some $\rho>0$, which contradicts the
positivity of the density. Hence
\be
\{x\in\ol{\om}\cap B_r(x_0):\Theta^1(\mu_*,x)>0\}
\subset S_*\cap B_r(x_0).
\label{inclusion}
\ee
Conversely, suppose that $x\in S_*\cap B_r(x_0)$ and
$\Theta^1(\mu_*,x)<\delta$, where $\delta\in(0,1)$ will be chosen below.
Then there exists
$\rho\in(0,\f{1}{10}(r-|x-x_0|))$ such that, for all sufficiently large $i$,
\[
\theta_{\va_i}^{\om}(u_{\va_i};x,\rho)
<\delta\log\f{\rho}{\va_i}.
\]
Using \eqref{gvaassuse} and Proposition~\ref{BoundaryClearingout}, and then
choosing $\delta$ sufficiently small, we obtain
\be
\theta_{\va_i}^{\om}\left(u_{\va_i};x,\f{\rho}{2}\right)\leq C,
\label{thetarhohalf}
\ee
with $C$ independent of $i$. Dividing \eqref{thetarhohalf} by
$|\log\va_i|$ and letting $i\to+\ift$ give
$\mu_*(B_{\f{\rho}{4}}(x))=0$, contradicting $x\in S_*$. Thus,
\[
S_*\cap B_r(x_0)\subset
\{x\in\ol{\om}\cap B_r(x_0):\Theta^1(\mu_*,x)\geq\delta\}.
\]
Together with \eqref{inclusion}, this proves \eqref{claimSstar}.

By \eqref{Tva1} and \eqref{munconve}, we have $|T_*|\leq C\mu_*$. Hence
$T_*$ is absolutely continuous with respect to $\mu_*$. By the
Radon--Nikodym theorem, there exists a $\mu_*$-measurable map
$F_*:\ol{\om}\cap B_r(x_0)\to\R^3\otimes\R^3$ such that $ T_*=F_*\mu_* $ and
\be
\int_{\ol{\om}\cap B_r(x_0)}|F_*|\ud\mu_*<+\ift.
\label{Tstarmustar}
\ee
Passing to the limit in the symmetry of $T_{\va_i}(u_{\va_i})$ and in
\eqref{Tva11}--\eqref{Tva111}, we obtain, for $\mu_*$-a.e.
$x\in\ol{\om}\cap B_r(x_0)$,
\be
F_*(x)^{\T}=F_*(x),
\quad
\tr F_*(x)\geq 1,
\quad
\lda_j(F_*(x))\leq 1
\quad\text{for }j=1,2,3,
\label{Foengenva}
\ee
where $\lda_1(F_*)\leq\lda_2(F_*)\leq\lda_3(F_*)$ denote the eigenvalues of
$F_*$. In particular, $F_*(x)\neq0$ and
\[
\rank\f{\ud T_*}{\ud|T_*|}\geq1
\quad\text{for }|T_*|\text{-a.e. }x\in B_r(x_0).
\]

By Proposition~\ref{divTstar}, \eqref{gvaNawayA}, and
\eqref{tangentialenergyawayA}, the limit tensor satisfies
\[
\int_{\om_r(x_0)}\<T_*,D\xi\>\ud x\geq0
\]
for any $\xi\in C_0^1(B_r(x_0),\R^3)$ with
$\xi\cdot\nu\leq0$ on $T_r^{\om}(x_0)$, where $\nu$ is the unit outer normal
to $\pa\om$.

We now apply the rectifiability criterion in
\cite[Proposition~3.1]{ADHR19}, together with \eqref{claimSstar}. It gives a
$1$-rectifiable set $R\subset\ol{\om}\cap B_r(x_0)$ and a Borel map
$\lda_*:R\to\R^3\otimes\R^3$ such that $R=S_*\cap B_r(x_0)$ up to an
$\HH^1$-null set and, for $\HH^1$-a.e. $x\in R$, $\lda_*(x)$ is the
orthogonal projection onto the approximate tangent line $T_xR$. The
associated rectifiable varifold is
\[
V_*(\ud A,\ud x):=\delta_{\lda_*(x)}(\ud A)\otimes\mu_*(\ud x).
\]
The preceding first-variation inequality shows that $V_*$ satisfies the
hypotheses of \cite[Theorem~1]{Whi10} in
$\om_r(x_0)\cup T_r^{\om}(x_0)$. Since $\ol{\om}$ is strongly convex at
$x_0$, that theorem implies $x_0\notin S_*$. As
$x_0\in\pa\om\backslash\cA$ was arbitrary, we conclude that $ S_*\cap\pa\om\subset\cA $.

\smallskip
\noindent\textbf{Part 2. Proof that $S_*\subset\op{Conv}(\cA)$.}

Suppose, for contradiction, that
$S_*\cap(\om\backslash\op{Conv}(\cA))\neq\emptyset$. Define
\be
F:=\left\{x\in S_*:
\dist(x,\op{Conv}(\cA))
=\max_{y\in S_*}\dist(y,\op{Conv}(\cA))>0\right\}.\label{DefF}
\ee
Since $S_*$ is compact, $F\neq\emptyset$. By Part~1, every boundary point of
$S_*$ lies in $\cA$. Hence, the above maximum is positive and $ F\subset\om\backslash\op{Conv}(\cA) $. Moreover, $F$ is a compact subset of $\om$.

Fix $x_0\in F$. By Proposition~\ref{interiorproperty11}, there exist
$r>0$ and closed line segments $\{L_i\}_{i=1}^k\subset S_*$ emanating from
$x_0$, with unit outward direction vectors $v_i\in\Ss^2$, such that
\[
S_*\cap B_r(x_0)=\bigcup_{i=1}^kL_i,
\quad
\sum_{i=1}^k\lda_iv_i=0,
\]
where $\lda_i>0$ for any $i\in\Z\cap[1,k]$. Let $p(x_0)\in\op{Conv}(\cA)$ be
the nearest point to $x_0$, and set
\[
n_0:=\f{x_0-p(x_0)}{|x_0-p(x_0)|}\in\Ss^2.
\]
Then
\[
\op{Conv}(\cA)\subset H^-:=
\{y:(y-p(x_0))\cdot n_0\leq0\},
\]
while $x_0$ lies on the parallel plane
$\{y:(y-p(x_0))\cdot n_0=|x_0-p(x_0)|\}$.

For any $i$, if $v_i\cdot n_0>0$, then for all sufficiently small $t>0$,
\[
\dist(x_0+tv_i,\op{Conv}(\cA))
\geq \dist(x_0+tv_i,H^-)
>\dist(x_0,\op{Conv}(\cA)),
\]
which contradicts the definition of $F$. Therefore, $v_i\cdot n_0\leq0$ for
all $i$. If $v_{i_0}\cdot n_0<0$ for some $i_0$, then
\[
0=\left(\sum_{i=1}^k\lda_iv_i\right)\cdot n_0
=\sum_{i=1}^k\lda_i(v_i\cdot n_0)<0.
\]
Again, a contradiction. Hence $ v_i\cdot n_0=0 $ for any $ i\in\Z\cap[1,k] $. It follows that for any $i$ and any
$t\in[0,\HH^1(L_i)]$,
\[
\dist(x_0+tv_i,\op{Conv}(\cA))
\geq \dist(x_0+tv_i,H^-)
=\dist(x_0,\op{Conv}(\cA)).
\]
By the maximality of $x_0$, every point in each segment $L_i$ belongs to
$F$.

We now apply the same argument at an arbitrary point of $F$. Let $z\in F$.
Since $F\subset\om\backslash\op{Conv}(\cA)$, Proposition~\ref{interiorproperty11}
applies to $z$. Repeating the preceding argument with $z$ in place of $x_0$,
we find $\rho_z>0$ such that $ S_*\cap B_{\rho_z}(z)\subset F $. Thus, $F$ is relatively open in $S_*$. Since $F$ is defined as the level set
in which the continuous function $x\mapsto\dist(x,\op{Conv}(\cA))$ reaches its
maximum on $S_*$, it is also relatively close in $S_*$.

Since $S_*$ is connected, this implies $F=S_*$. This is impossible: according
to Part~1, the boundary points of $S_*$ lie in $\cA\subset\op{Conv}(\cA)$; this contradicts the definition of $ F $ in \eqref{DefF}. Therefore, $ S_*\subset\op{Conv}(\cA) $.

\smallskip
\noindent\textbf{Part 3. Proof that $\cA\subset S_*\cap\pa\om$.}

Let $a\in\cA$. Suppose, for contradiction, that $a\notin S_*$. Since $S_*$
is closed and since we have already proved
$S_*\cap\pa\om\subset\cA$ and $S_*\subset\op{Conv}(\cA)$, there exists
$r\in(0,\f{1}{10})$ such that
\be
S_*\cap\ol{B}_{2r}(a)=\emptyset,
\quad
\cA\cap B_{2r}(a)=\{a\}.
\label{aproperty}
\ee

Choose $\rho\in(\f{r}{2},r)$ such that the geodesic circle
\[
\gamma_{a,\rho}:=\pa D_a^{\rho}\subset\pa\om
\]
is smooth, where $D_a^{\rho}\subset\pa\om$ is the geodesic disk centered at
$a$ with radius $\rho$. By the definition of suitable boundary data at
$a$, the free homotopy class
\be
\sg_a:=[g|_{\gamma_{a,\rho}}]_{\cN}\in[\Ss^1,\cN]
\label{sigmadefineboundarya}
\ee
is non-trivial.

For $\eta>0$ small, define the boundary collar annulus
\[
U_{\eta,r}:=
\left\{x\in\om\cap(B_r(a)\backslash B_{\f{r}{2}}(a)):
\dist(x,\pa\om)<\eta\right\}.
\]
By \eqref{aproperty}, Proposition~\ref{BoundaryClearingout}, and a standard
covering argument, we can choose $\eta>0$ sufficiently small so that $ E_{\va_i}(u_{\va_i},U_{\eta,r})\leq C $ for all sufficiently large $i$, with $C$ independent of $i$. Applying Proposition~\ref{BoundaryPartialRegularity1prop}, and decreasing $\eta$ and
$r$ if necessary, we obtain
\[
\|e_{\va_i}(u_{\va_i})\|_{L^{\ift}(U_{\eta,r})}\leq C
\]
for all sufficiently large $i$. Consequently,
\[
u_*\in C^{\ift}(U_{\eta,r},\cN)\cap C^0(\ol{U}_{\eta,r},\cN),
\]
and the trace of $u_*$ on
$\pa\om\cap(B_r(a)\backslash B_{\f{r}{2}}(a))$ is $g$.

Let $\nu_{\op{in}}$ denote the unit inner normal to $\pa\om$. After decreasing
$\eta>0$, if necessary, the normal collar map
\[
\Psi:\pa\om\times[0,\eta]\to\ol{\om},
\quad
\Psi(y,s):=y+s\nu_{\op{in}}(y),
\]
is a $C^1$ diffeomorphism onto its image. For $s\in(0,\eta)$, set
\[
\gamma_{a,\rho}^s:=\Psi(\gamma_{a,\rho},s),
\quad
D_{a,\rho}^s:=\Psi(D_a^\rho,s).
\]
Here, $D_{a,\rho}^s$ is an embedded two-dimensional disk in $\om$ and $ \pa D_{a,\rho}^s=\gamma_{a,\rho}^s $.

Since $S_*\cap\ol{B}_{2r}(a)=\emptyset$ and $S_0$ is locally finite in
$\om\backslash S_*$, we may choose $s\in(0,\eta)$ such that
\[
D_{a,\rho}^s\subset\om\cap B_{2r}(a),
\quad
D_{a,\rho}^s\cap S_0=\emptyset.
\]
The map $ (t,y)\in[0,1]\times\gamma_{a,\rho}\mapsto\Psi(y,ts) $ gives a homotopy in $U_{\eta,r}$ between $\gamma_{a,\rho}$ and $\gamma_{a,\rho}^s$. Since $u_*$ is continuous up to the boundary in $\ol{U}_{\eta,r}$ and its trace on $\pa\om\cap(B_r(a)\backslash B_{\f{r}{2}}(a))$ is $g$, this homotopy gives
\[
[u_*|_{\gamma_{a,\rho}^s}]_{\cN}
=[g|_{\gamma_{a,\rho}}]_{\cN}
=\sg_a.
\]
In particular, $[u_*|_{\gamma_{a,\rho}^s}]_{\cN}$ is non-trivial by
\eqref{sigmadefineboundarya}.

On the other hand, $u_*$ is smooth and $\cN$-valued on
$D_{a,\rho}^s$, because $D_{a,\rho}^s\cap(S_*\cup S_0)=\emptyset$. Hence,
the loop $u_*|_{\gamma_{a,\rho}^s}$ extends to a $\cN$-valued map on the
disk $D_{a,\rho}^s$. Therefore, this loop is freely null-homotopic in
$\cN$, which contradicts the non-triviality of $\sg_a$. Consequently,
$a\in S_*$.

Combining this with Part~1, we obtain $ S_*\cap\pa\om=\cA $. 
Together with Part~2, this proves Proposition~\ref{boundarySstarbehave}\ref{boundary1}.
\end{proof}

\begin{proof}[Proof of Proposition~\ref{boundarySstarbehave}\ref{boundary11}]
Fix $x_0\in\cA$. After translation and rotation, we assume that
$x_0=0$ and that $\om$ is locally represented near $0$ by the model ball
\[
B_*:=\{y\in\R^3:y_1^2+y_2^2+(y_3+1)^2<1\}.
\]
All estimates in the following are invariant under the fixed local $C^{2,1}$ change of
coordinates. The constants may depend on $\om$, $f$, $\cN$, and the local
chart, but not on $i$.

For any $r_0>0$ sufficiently small, the global energy bound gives
\[
\limsup_{i\to+\ift}\f{1}{|\log\va_i|}
\int_{\om_{r_0}}e_{\va_i}(u_{\va_i})\ud x<+\ift.
\]
By Fubini's theorem, after passing to a subsequence and using Fatou's lemma,
\[
\limsup_{i\to+\ift}\f{1}{|\log\va_i|}
\int_{\pa B_r\cap\om}e_{\va_i}(u_{\va_i})\ud\HH^2<+\ift
\quad\text{for }\HH^1\text{-a.e. }r\in(0,r_0).
\]
We choose $r_0>0$ so small that $\cA\cap B_{r_0}=\{0\}$. Recalling
Proposition~\ref{boundarySstarbehave}\ref{boundary1}, we have
$S_*\cap\pa\om=\cA$ and $S_*\subset\op{Conv}(\cA)$. Hence, after decreasing
$r_0$ if necessary, $\pa B_{r_0}$ does not contain a branching point of $S_*$.

By Lemma~\ref{stationvari},
\[
\mu_*\llcorner\om=\Theta^1(\mu_*,x)\HH^1\llcorner(S_*\cap\om).
\]
For $r\in(0,r_0)$ such that $\pa B_r$ contains no branching point of $S_*$,
define
\[
h(r):=\int_{S_*\cap\pa B_r}\Theta^1(\mu_*,x)\ud\HH^0(x).
\]
By Lemma~\ref{densitydescrete} and inclusion 
$S_*\subset\op{Conv}(\cA)$, the function $h$ is bounded from below by a
positive constant and takes values in a discrete subset of $(0,+\ift)$.
Moreover, by Proposition~\ref{interiorproperty1}, the branching points of
$S_*$ are locally finite in $\om$. Hence, the radii at which
$\pa B_r$ contains a branching point form a discrete subset of $(0,r_0]$.
It follows that $h$ is piecewise constant on $(0,r_0]$. More precisely,
after decreasing $r_0$ if necessary, there exist $k_0\in\Z$, $k_0\geq 0$ and a
partition
\[
b_0=r_0,\quad
(0,r_0]=\bigcup_{j=0}^{k_0}(a_j,b_j],
\quad
a_j=b_{j+1}\quad\text{for }j\in\Z\cap[0,k_0-1],
\]
such that, for each $j\in\Z\cap[0,k_0]$, the following properties hold:
\begin{itemize}
\item $h$ is constant on $(a_j,b_j)$;
\item $\pa B_r$ contains no branching point of $S_*$ for every
$r\in(a_j,b_j)$;
\item if $a_j>0$, then $\pa B_{a_j}$ contains at least one branching point
of $S_*$.
\end{itemize}

Since $\ol{\om}$ is strongly convex at every boundary point and
$S_*\subset\op{Conv}(\cA)$, we have
\[
\mu_*(\pa(\om_r))=0
\quad\text{for every }r\in(0,r_0].
\]
Combining this with \eqref{munconve}, for each $j\in\Z\cap[0,k_0]$ we obtain
\[
\lim_{i\to+\ift}
\int_{a_j}^{b_j}
\left(
\f{1}{|\log\va_i|}
\int_{\pa B_r\cap\om}e_{\va_i}(u_{\va_i})\ud\HH^2
-h(r)
\right)\ud r=0.
\]
After passing to a further subsequence, Fatou's lemma gives
\be
\int_{a_j}^{b_j}
\left(
\limsup_{i\to+\ift}
\f{1}{|\log\va_i|}
\int_{\pa B_r\cap\om}e_{\va_i}(u_{\va_i})\ud\HH^2
-h(r)
\right)\ud r\leq0.
\label{sibileq}
\ee
Therefore, for each $j=\Z\cap[0,k_0]$, we claim that there exists
$\delta_j\in(a_j,b_j)$ such that
\be
\limsup_{i\to+\ift}
\f{1}{|\log\va_i|}
\int_{\pa B_{\delta_j}\cap\om}e_{\va_i}(u_{\va_i})\ud\HH^2
\leq h(\delta_j).
\label{geqMr}
\ee
Indeed, if the limsup equals $h(r)$ for $\HH^1$-a.e. $r$, then \eqref{sibileq}
forces equality a.e., and any such $r$ works. Otherwise, any $r$ in the
positive-measure set where the limsup is strictly below $h(r)$ serves.

Fix $\delta=\delta_j$ as in \eqref{geqMr} so that some $j$ will be chosen later.
We now construct a competitor $v_i$ for $u_{\va_i}$ in $B_\delta\cap\om$,
with $v_i=u_{\va_i}$ on $\om\backslash B_\delta$, and prove
\be
\limsup_{i\to+\ift}\f{1}{|\log\va_i|}
E_{\va_i}(v_i,B_\delta\cap\om)=0
\label{iiftviestimate}
\ee
away from the conical part where $v_i$ is defined by radial extension.

\medskip
\noindent\textbf{Step~1. Construction on the conical region.}
Define
\[
K_*:=\left\{y\in B_*\cap B_\delta:
\f{\delta y}{|y|}\in\pa B_\delta\cap\ol{B}_*\right\},
\quad
\wt{K}_*:=\pa K_*\backslash\pa B_\delta .
\]
Thus, $\wt{K}_*$ is the lateral interface between the conical region $K_*$
and the remaining part of $B_*\cap B_\delta$. We set $v_i=u_{\va_i}$ on
$\om\backslash B_\delta$, and for $x\in K_*\backslash B_{\va_i}$ define
\[
v_i(x):=u_{\va_i}\left(\f{\delta x}{|x|}\right).
\]
On the interface $\wt{K}_*$ this induces the boundary value
\[
\wt{g}_{\va_i}(x):=g_{\va_i}\left(\f{\delta x}{|x|}\right).
\]

\medskip
\noindent\textbf{Step~2. Dyadic decomposition of the boundary layer.}
Set
\[
V_0:=(\ol{B}_*\cap(B_\delta\backslash B_{\va_i}))\backslash K_*,
\quad
V_1:=\ol{B}_*\cap B_{\va_i}.
\]
For $r>0$, define
\[
E_r:=(B_*\cap B_r)\backslash K_*,
\quad
\wt{E}_r:=\pa B_r\cap E_r.
\]
Let
\[
N_0:=\left\lfloor\log_2\f{\delta}{\va_i}\right\rfloor-10,
\quad
F_j:=E_{2^{j+1}\va_i}\backslash\ol{E}_{2^j\va_i}
\quad\text{for }j\in\Z\cap[1,N_0].
\]
Then
\[
V_0=(V_0\backslash E_{2^{N_0+1}\va_i})\cup
\bigcup_{j=1}^{N_0}F_j.
\]

For $x\in F_j$, let $\theta(x)\in[0,2\pi)$ be the azimuthal angle in the
cylindrical coordinates $(\rho,\theta,z)$ of $\R^3$, and define
\be
z_j(x):=\f{|x|}{2^j\va_i}\in[1,2],\quad
d_0(x):=\dist(x,K_*),\quad
d_1(x):=\dist(x,\pa B_*).
\label{d0d1def}
\ee
We also set
\[
r_j(x):=1+\f{d_0(x)}{d_0(x)+d_1(x)}\in[1,2].
\]
Define $\Phi_j:\ol{F}_j\to(\ol{B}_2^2\backslash B_1^2)\times[1,2]$ by
\be
\Phi_j(x):=(r_j(x)\cos\theta(x),r_j(x)\sin\theta(x),z_j(x)).
\label{representationPhi}
\ee
We claim that, for $\delta\in(0,1)$ sufficiently small, $\Phi_j$ is
bi-Lipschitz on the closed set $\ol{F}_j$ and satisfies
\be
\|\na\Phi_j\|_{L^{\ift}(F_j)}
\leq C(2^j\va_i)^{-1},
\quad
\|\na(\Phi_j^{-1})\|_{L^{\ift}((B_2^2\backslash B_1^2)\times[1,2])}
\leq C(2^j\va_i),
\label{Phijdef}
\ee
where $C>0$ may depend on $\delta$ but is independent of $j$ and $i$.

We verify \eqref{Phijdef}. The functions $d_0$ and $d_1$ are Lipschitz, and
the following computations are understood at points of differentiability. On
$F_j$, one has $|\na d_0|=1$ and $d_1\leq C(2^j\va_i)\delta$. Moreover, the
angle between the ray through $x$ and the interface $\wt{K}_*$ is bounded
from below by $C^{-1}\delta$, and hence
\be
d_0(x)+d_1(x)\geq C^{-1}(2^j\va_i)\delta
\quad\text{for any }x\in F_j.
\label{d0d1lowerbound}
\ee
For $\delta$ small,
\be
|\na(d_0+d_1)|^2
=2+2\na d_0\cdot\na d_1
=2-2\na d_0\cdot\f{x+(0,0,1)}{|x+(0,0,1)|}
\leq C\delta.
\label{nad0nad1}
\ee
Since
\[
r_j=2-\f{d_1}{d_0+d_1},
\]
we have
\be
\na r_j=-\f{\na d_0}{d_0+d_1}
+\f{d_0(\na d_0+\na d_1)}{(d_0+d_1)^2}.
\label{nablarjrepresent}
\ee
For $x\in F_j$, the geometry of the model ball gives
$|x_3|\leq C(x_1^2+x_2^2)^{\f{1}{2}}$, and therefore
\be
|\na\theta|=\f{1}{(x_1^2+x_2^2)^{\f{1}{2}}}
\leq\f{C}{2^j\va_i},
\quad\na z_j=\f{x}{2^j\va_i|x|}.
\label{nathetaj}
\ee
The upper bound in \eqref{Phijdef} follows from
\eqref{d0d1lowerbound}--\eqref{nathetaj}. For the inverse bound, we use
\eqref{d0d1lowerbound} to obtain
\be
\left|\det\left(
\f{\na\theta}{|\na\theta|},\f{x}{|x|},
\f{2^j\va_i\na d_1}{d_0+d_1}\right)\right|
\geq\f{1}{C\delta}
\left|\det\left(
\f{\na\theta}{|\na\theta|},\f{x}{|x|},
\f{x+(0,0,1)}{|x+(0,0,1)|}\right)\right|
\geq\f{1}{C\delta},
\label{Cdelta}
\ee
while \eqref{d0d1lowerbound} and \eqref{nad0nad1} give
\be
\left|\det\left(
\f{\na\theta}{|\na\theta|},\f{x}{|x|},
\f{2^j\va_i d_0(\na d_0+\na d_1)}{(d_0+d_1)^2}\right)\right|
\leq C\delta^{-\f{1}{2}}.
\label{Cdelta1}
\ee
The constants in \eqref{d0d1lowerbound}--\eqref{Cdelta1} are independent of
$\delta$. Thus, for $\delta$ small, \eqref{Cdelta} and \eqref{Cdelta1} imply
\[
|\det(2^j\va_i\na r_j,2^j\va_i\na\theta,2^j\va_i\na z_j)|
\geq C^{-1}.
\]
Consequently,
\[
|\det(\na\Phi_j)|\geq \f{1}{C(2^j\va_i)^3},
\quad
\|\na(\Phi_j^{-1})\|_{L^{\ift}}
\leq\f{C\|\na\Phi_j\|_{L^{\ift}}^2}{|\det(\na\Phi_j)|}
\leq C(2^j\va_i),
\]
which completes the proof of \eqref{Phijdef}.

Under $\Phi_j$, the four boundary components of $F_j$ correspond to
\begin{align*}
\wt{E}_{2^{j+1}\va_i}
&\longleftrightarrow
(\ol{B}_2^2\backslash B_1^2)\times\{2\}
&&(\text{top annulus}),\\
\wt{E}_{2^j\va_i}
&\longleftrightarrow
(\ol{B}_2^2\backslash B_1^2)\times\{1\}
&&(\text{bottom annulus}),\\
F_j\cap\pa B_*
&\longleftrightarrow
\pa B_2^2\times[1,2]
&&(\text{outer lateral face}),\\
F_j\cap\wt{K}_*
&\longleftrightarrow
\pa B_1^2\times[1,2]
&&(\text{inner lateral face}).
\end{align*}

\medskip
\noindent\textbf{Step~3. Inductive construction of $v_i$ in $V_0$.}
We construct $v_i$ on the shells $F_j$, proceeding from $j=N_0$ down to
$j=1$. The induction maintains the following properties.
\begin{itemize}
\item \emph{Shell energy}: $E_{\va_i}(v_i,F_j)\leq C(2^j\va_i)$;
\item \emph{Bottom face gradient}:
\be
\|\na_\top(v_i|_{\wt{E}_{2^j\va_i}})\|_{L^2(\wt{E}_{2^j\va_i})}
\leq C_1.
\label{inductive-invariant}
\ee
\end{itemize}
Here $C>0$ and $C_1>0$ depend only on $\delta,f,\cN$, and the constant
$C_0$ in \eqref{Giestimate1} below, but not on $j$ or $i$. The independence
of $C_1$ from $j$ is the essential point. 

\smallskip
\textit{Base case.}
For $x\in V_0\backslash E_{2^{N_0+1}\va_i}$, let
$y\in B_\delta\cap\pa B_*$ be the unique point collinear with $0$ and $x$.
Set
\[
v_i(x):=g_{\va_i}(y).
\]
This defines $v_i$ on $V_0\backslash E_{2^{N_0+1}\va_i}$ and, in particular,
on $\wt{E}_{2^{N_0+1}\va_i}$. By the regularity of $g_{\va_i}$ from
\eqref{gvaassuse}, the tangential gradient of
$v_i|_{\wt{E}_{2^{N_0+1}\va_i}}$ is bounded in $L^2$ by a constant
independent of $i$. Hence, the inductive invariant
\eqref{inductive-invariant} holds at the initial level $j=N_0+1$, after
enlarging $C_1$ if necessary.

\smallskip
\textit{Inductive step at level $j$ $(N_0\geq j\geq 1)$.}
Suppose that $v_i$ has already been constructed on all shells $F_{j'}$ with
$j'>j$, as well as on $V_0\backslash E_{2^{N_0+1}\va_i}$. In particular, the
top annulus $\wt{E}_{2^{j+1}\va_i}$ of the shell $F_j$ already carries the
trace of $v_i$, and the inductive invariant gives
\[
\|\na_\top(v_i|_{\wt{E}_{2^{j+1}\va_i}})\|_{L^2(\wt{E}_{2^{j+1}\va_i})}
\leq C_1 .
\]
The two lateral pieces also carry prescribed data: on $F_j\cap\pa B_*$ the
data come from $g_{\va_i}$, while on $F_j\cap\wt{K}_*$ they come from
\[
\wt{g}_{\va_i}(x):=g_{\va_i}\left(\f{\delta x}{|x|}\right).
\]
These lateral data satisfy the scale-independent $L^\infty$ gradient bound
that will be recorded in \eqref{Giestimate1}. The goal at the present level
is to construct $v_i$ on $F_j$ so that
\[
E_{\va_i}(v_i,F_j)\leq C(2^j\va_i)
\quad\text{and}\quad
\|\na_\top(v_i|_{\wt{E}_{2^j\va_i}})\|_{L^2(\wt{E}_{2^j\va_i})}\leq C_1,
\]
with the same constants independent of $j$ and $i$. We carry this out in the
following sub-steps.

\smallskip
\noindent\textbf{(3a) Data on the standard annular cylinder.}
Via $\Phi_j$, the data on the inner and outer lateral faces induce boundary
values $G_i$ on $\pa B_1^2\times[1,2]$ and $\pa B_2^2\times[1,2]$. By the
boundary regularity of $g_{\va_i}$ and \eqref{Phijdef},
\be
\|\na_\top G_i\|_{L^{\ift}(\pa B_1^2\times[1,2])}
+\|\na_\top G_i\|_{L^{\ift}(\pa B_2^2\times[1,2])}
\leq C_0,
\label{Giestimate1}
\ee
where $C_0$ is independent of $j$ and $i$. The top annulus
$(\ol{B}_2^2\backslash B_1^2)\times\{2\}$ receives its data from
$v_i|_{\wt{E}_{2^{j+1}\va_i}}$ through $\Phi_j$, and its tangential gradient
has $L^2$-norm bounded by $C_1$ by the inductive hypothesis.

\smallskip
\noindent\textbf{(3b) Unfolding of the annular cylinder.}
We cut the closed annular cylinder
$(\ol{B}_2^2\backslash B_1^2)\times[1,2]$ along one radial half-plane and use
the map $\Phi$ shown in Figure~\ref{fig:cylinder-map} to unfold the resulting
closed cylinder onto a rectangular box $\mathcal{B}$ with corners
$A_1,A_2,A_2',A_1',B_1,B_2,B_2',B_1'$.

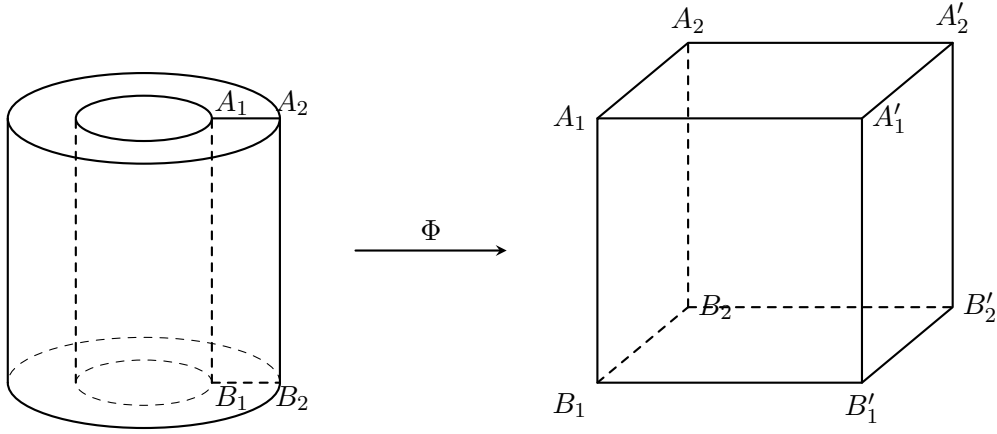
\begin{figure}[htbp]
\centering
\begin{tikzpicture}[>=stealth, line cap=round, line join=round]

    \begin{scope}[shift={(0,0)}]
        \draw[thick] (1.8, 0) arc (0:-180:1.8 and 0.6);
        \draw[dashed, thin] (1.8, 0) arc (0:180:1.8 and 0.6);
        \draw[dashed, thin] (0,0) ellipse (0.9 and 0.3);

        \draw[thick] (-1.8, 0) -- (-1.8, 3.5);
        \draw[thick] (1.8, 0) -- (1.8, 3.5);
        \draw[dashed, thick] (-0.9, 0) -- (-0.9, 3.5);
        \draw[dashed, thick] (0.9, 0) -- (0.9, 3.5);

        \draw[thick] (0, 3.5) ellipse (1.8 and 0.6);
        \draw[thick] (0, 3.5) ellipse (0.9 and 0.3);

        \draw[thick] (0.9, 3.5) -- (1.8, 3.5);
        \draw[dashed, thick] (0.9, 0) -- (1.8, 0);

        \node[right] at (0.8, 3.7) {$A_1$};
        \node[right] at (1.6, 3.7) {$A_2$};
        \node[right] at (0.8, -0.2) {$B_1$};
        \node[right] at (1.6, -0.2) {$B_2$};

        \draw[->, thick] (2.8, 1.75) -- (4.8, 1.75);
        \node[above] at (3.8, 1.75) {$\Phi$};
    \end{scope}

    \begin{scope}[shift={(6,0)}]
        \draw[dashed, thick] (1.2, 1.0) -- (4.7, 1.0);
        \draw[dashed, thick] (0, 0) -- (1.2, 1.0);
        \draw[dashed, thick] (1.2, 1.0) -- (1.2, 4.5);

        \draw[thick] (0,0) -- (3.5,0) -- (3.5,3.5) -- (0,3.5) -- cycle;
        \draw[thick] (3.5, 3.5) -- (4.7, 4.5) -- (4.7, 1.0) -- (3.5, 0);
        \draw[thick] (0, 3.5) -- (1.2, 4.5) -- (4.7, 4.5);

        \node[left] at (0, 3.5) {$A_1$};
        \node[right] at (3.5, 3.5) {$A_1'$};
        \node[above] at (1.2, 4.5) {$A_2$};
        \node[above] at (4.7, 4.5) {$A_2'$};

        \node[below left] at (0, 0) {$B_1$};
        \node[below] at (3.5, 0) {$B_1'$};
        \node[right] at (1.2, 1.0) {$B_2$};
        \node[right] at (4.7, 1.0) {$B_2'$};
    \end{scope}

\end{tikzpicture}
\caption{Unfolding of the annular cylinder to a rectangular box. The two vertical faces
$A_1A_2B_2B_1$ and $A_1'A_2'B_2'B_1'$ are the two faces created by the angular cut.}
\label{fig:cylinder-map}
\end{figure}

Under $\Phi$, the inner lateral face corresponds to $A_1A_1'B_1'B_1$, the
outer lateral face corresponds to $A_2A_2'B_2'B_2$, the top face corresponds
to $A_1A_2A_2'A_1'$, and the bottom face corresponds to
$B_1B_2B_2'B_1'$. The two remaining faces,
$A_1A_2B_2B_1$ and $A_1'A_2'B_2'B_1'$, are the two faces introduced by the
angular cut.

Set $ H_i:=v_i\circ\Phi_j^{-1}\circ\Phi^{-1} $ on the faces where the data have already been prescribed. Thus, the data are known on the inner and outer lateral faces, with the $L^\infty$ tangential gradient bound in \eqref{Giestimate1}, and on the top face, where the $L^2$ tangential gradient is bounded by the inductive hypothesis. By Fubini's theorem, we choose the angular cut defining $\Phi$ so that
\[
\|\na_{A_1A_2}H_i\|_{L^2(A_1A_2)}
+\|\na_{A_1'A_2'}H_i\|_{L^2(A_1'A_2')}
\leq C\|\na_\top H_i\|_{L^2(A_1A_2A_2'A_1')}
\leq CC_1 .
\]

\smallskip
\noindent\textbf{(3c) Construction of the lower and cut faces.}
The edge data already prescribed determine two paths in $\cN$:
\[
\mathcal{P}_i:=B_1\op{-}A_1\op{-}A_2\op{-}B_2,
\quad
\mathcal{P}_i':=B_1'\op{-}A_1'\op{-}A_2'\op{-}B_2' .
\]
Here, the notation means that the path is obtained by following the existing
edge data. The $L^2$ bounds on the top edges and the $L^\infty$ bounds on the
lateral edges give
\[
\|\na H_i\|_{L^2(\mathcal{P}_i)}
+\|\na H_i\|_{L^2(\mathcal{P}_i')}
\leq C .
\]
We define $H_i$ on the bottom edges $B_1B_2$ and $B_1'B_2'$ by
reparameterization $\mathcal{P}_i$ and $\mathcal{P}_i'$, respectively. Hence,
\[
\|\na H_i\|_{L^2(B_1B_2)}
+\|\na H_i\|_{L^2(B_1'B_2')}
\leq C,
\]
where $C$ is independent of $j$ and $i$.

We now fill the two cut faces. Consider first the face $A_1A_2B_2B_1$.
With the orientation
\[
B_1\op{-}A_1\op{-}A_2\op{-}B_2\op{-}B_1,
\]
its boundary loop is the concatenation of the path $\mathcal{P}_i$ and the
reverse of the bottom edge $B_1B_2$. Since the bottom edge has just been
chosen as a reparameterization of $\mathcal{P}_i$, this loop is freely
null-homotopic in $\cN$. Therefore, Lemma~\ref{ExtensionLemma1} extends
$H_i$ to $A_1A_2B_2B_1$ and gives
\[
\|\na H_i\|_{L^2(A_1A_2B_2B_1)}\leq C .
\]
The same argument, using $\mathcal{P}_i'$ and the bottom edge $B_1'B_2'$,
extends $H_i$ to the primed cut face $A_1'A_2'B_2'B_1'$ with
\[
\|\na H_i\|_{L^2(A_1'A_2'B_2'B_1')}\leq C .
\]

It remains to fill the bottom face $B_1B_2B_2'B_1'$. We verify that its
boundary loop is null-homotopic. By the construction of the two cut faces,
the edge $B_1B_2$ is homotopic relative to its endpoints to
$B_1$-$A_1$-$A_2$-$B_2$, and the edge
$B_2'B_1'$ is homotopic relative to its endpoints to
$B_2'$-$A_2'$-$A_1'$-$B_1'$. Therefore, the boundary loop
\[
B_1\op{-}B_2\op{-}B_2'\op{-}B_1'\op{-}B_1
\]
is homotopic to
\[
B_1\op{-}A_1\op{-}A_2\op{-}B_2\op{-}B_2'
\op{-}A_2'\op{-}A_1'\op{-}B_1'\op{-}B_1 .
\]
Using the data already prescribed on the outer lateral face, the path
$A_2$-$B_2$-$B_2'$-$A_2'$ is homotopic relative to its
endpoints to the top edge $A_2$-$A_2'$. Similarly, using data on
the inner lateral face, the path
$A_1'$-$B_1'$-$B_1$-$A_1$ is homotopic relative to its
endpoints to the top edge $A_1'$-$A_1$. Hence, the bottom boundary loop
is homotopic in $\cN$ to a conjugate of the top boundary loop
\[
A_1\op{-}A_2\op{-}A_2'\op{-}A_1'\op{-}A_1 .
\]
The top face has already been filled by the inductive construction or by
the initialization at the first level. Thus, its boundary loop is
null-homotopic in $\cN$, and so the boundary loop of
$B_1B_2B_2'B_1'$ is also null-homotopic. Applying
Lemma~\ref{ExtensionLemma1} once again, we extend $H_i$ to the bottom face and
obtain
\be
\|\na H_i\|_{L^2(B_1B_2B_2'B_1')}\leq C.\label{nablaHiB1B}
\ee

\smallskip
\noindent\textbf{(3d) Interior extension through a bi-Lipschitz ball.}
At this point $H_i$ is defined on the whole boundary $\pa\mathcal{B}$ and
takes values in $\cN$. Since the closed box $\mathcal{B}$ is bi-Lipschitz
homeomorphic to the closed ball $\ol{B}_1^3$, fix a bi-Lipschitz map
$\Psi:\ol{B}_1^3\to\mathcal{B}$ with $\Psi(\pa B_1^3)=\pa\mathcal{B}$. Its
bi-Lipschitz constant depends only on the shape of $\mathcal{B}$. Pull back
the boundary data by
\[
\wt{H}_i:=H_i\circ\Psi|_{\pa B_1^3}\in H^1(\pa B_1^3,\cN),
\]
and extend homogeneously to $B_1^3$ by
\[
\wt{H}_i(x):=\wt{H}_i\left(\f{x}{|x|}\right),
\quad x\in B_1^3\backslash\{0\}.
\]
The homogeneous extension estimate in dimension three gives
\[
\|\na\wt{H}_i\|_{L^2(B_1^3)}^2
\leq C\|\na_\top\wt{H}_i\|_{L^2(\pa B_1^3)}^2.
\]
Pushing forward through $\Psi$ yields
\[
\|\na H_i\|_{L^2(\mathcal{B})}
\leq C\|\na H_i\|_{L^2(\pa\mathcal{B})}
\leq C.
\]
Finally, pulling back through $\Phi^{-1}$ and $\Phi_j^{-1}$ defines
$v_i\in H^1(F_j,\cN)$.

\smallskip
\noindent\textbf{(3e) Energy estimates and completion of the inductive step.} All gradient bounds on the faces of $\mathcal{B}$ are controlled by constants independent of $j$ and $i$. For functional energy, a change of variables through $\Phi_j$ (using estimates from \eqref{Phijdef}) gives
\begin{align*}
\int_{F_j}|\na v_i|^2\ud x
&=\int_{\Phi_j(F_j)}|\na H_i\circ\Phi|^2\cdot|\na(\Phi_j^{-1})|^2
\cdot|\det\na\Phi_j^{-1}|\ud y\\
&\leq C(2^j\va_i)^{-2}\cdot C(2^j\va_i)^3\cdot\|\na H_i\|_{L^2(\mathcal{B})}^2,
\end{align*}
so $E_{\va_i}(v_i,F_j)\leq C(2^j\va_i)$. For the bottom face, a change of variables on $\wt{E}_{2^j\va_i}$ using \eqref{nablaHiB1B} and the two-dimensional Jacobian $|\det(\na\Phi_j)|\sim C(2^j\va_i)^{-2}$ gives
\[
\|\na_\top(v_i|_{\wt{E}_{2^j\va_i}})\|_{L^2(\wt{E}_{2^j\va_i})}^2
\leq C(2^j\va_i)^{-2}\cdot C(2^j\va_i)^2\cdot\|\na H_i\|_{L^2(B_1B_2B_2'B_1')}^2
\leq C_1.
\]
Collecting,
\be
E_{\va_i}(v_i,F_j)\leq C(2^j\va_i),\quad
\|\na_\top(v_i|_{\wt{E}_{2^j\va_i}})\|_{L^2(\wt{E}_{2^j\va_i})}\leq C_1,
\label{nablavifinal}
\ee
with $C,C_1>0$ independent of $j$ and $i$. The second bound in
\eqref{nablavifinal} is exactly the invariant \eqref{inductive-invariant} at level $j$ and gives the estimate on $ E_{\va_i}(v_i,F_j) $, completing the inductive step.

\medskip
\noindent\textbf{Step~4: Innermost region $V_1$.}

For $x\in V_1=(\ol{B}_*\cap B_{\va_i})\backslash K_*$, extend $v_i$
homogeneously from $\pa(B_{\va_i}\cap B_*)\backslash K_*$. Using \eqref{nablavifinal}
with $j=1$ and the regularity of $g_{\va_i}$,
\be
E_{\va_i}(v_i,B_{\va_i}\cap B_*)
\leq C\va_i E_{\va_i}(v_i,\pa(B_{\va_i}\cap B_*))\leq C,\label{EvaiBvai}
\ee
with $C$ independent of $i$.

\medskip
\noindent\textbf{Step~5. Energy estimate in the boundary layer.}
Summing \eqref{nablavifinal} over the dyadic shells gives
\[
\int_{V_0}e_{\va_i}(v_i)\ud x
\leq \sum_{j=1}^{N_0}E_{\va_i}(v_i,F_j)
\leq C\(\sum_{j=1}^{N_0}2^j\va_i\)
\leq C.
\]
Together with \eqref{EvaiBvai}, this proves
\[
\limsup_{i\to+\ift}\f{1}{|\log\va_i|}
E_{\va_i}(v_i,B_\delta\cap\om)=0.
\]
Thus, only the conical region $K_*$ contributes at the scale
$|\log\va_i|$.

\medskip
\noindent\textbf{Step~6. Minimality and the equality case in monotonicity.}
By the weak${}^*$ convergence \eqref{munconve} and $\mu_*(\pa(\om_\delta))=0$,
\[
\int_0^\delta h(\rho)\ud\rho=\mu_*(\om_\delta)
=\lim_{i\to+\ift}\f{1}{|\log\va_i|}\int_{\om_\delta}e_{\va_i}(u_{\va_i})\ud x.
\]
Since $v_i=u_{\va_i}$ outside $B_\delta$, the minimality of $u_{\va_i}$ and
\eqref{iiftviestimate} give
\be
\begin{aligned}
\int_0^\delta h(\rho)\ud\rho
&\leq\limsup_{i\to+\ift}\f{1}{|\log\va_i|}E_{\va_i}(v_i,\om_\delta)=\limsup_{i\to+\ift}\f{1}{|\log\va_i|}\int_{K_*}e_{\va_i}(v_i)\ud x\\
&\leq\limsup_{i\to+\ift}\f{\delta}{|\log\va_i|}
\int_{\pa B_\delta\cap\om}e_{\va_i}(u_{\va_i})\ud\HH^2
\leq\delta h(\delta),
\end{aligned}\label{deltafinal}
\ee
where the third inequality uses $|\na_{\pa B_\rho}u_{\va_i}|\leq|\na u_{\va_i}|$
and the cone construction, and the last uses \eqref{geqMr}.

We now choose $\delta$ to minimize $h$ on $(0,r_0]$. Since $h$ has finitely
many values, its minimum $h_{\min}$ is attained on some interval
$(a_{j_1},b_{j_1})$. Take $\delta=\delta_{j_1}\in(a_{j_1},b_{j_1})$
satisfying \eqref{geqMr}. Then $h\geq h_{\min}=h(\delta)$ on $(0,\delta)$,
so
\[
\int_0^\delta h(\rho)\ud\rho\geq\delta h_{\min}=\delta h(\delta).
\]
Combined with \eqref{deltafinal}, equality holds throughout:
$h\equiv h_{\min}$ on $(0,\delta]$, i.e., $r^{-1}\mu_*(\om_r)=h_{\min}$ is constant for $r\in(0,\delta)$.

We now apply the monotonicity formula for the stationary one-dimensional
varifold $V_*$ associated with $\mu_*$; see \cite[Formula~17.4]{Sim83}.
Although the center $0$ lies on $\pa\om$, the usual interior formula applies
here. Indeed, by Proposition~\ref{boundarySstarbehave}\ref{boundary1} and
our choice of $\delta$, we have $ S_*\cap\pa\om\cap\ol{B}_\delta=\{0\} $.

Hence, for any $0<\rho<\sigma<\delta$, the support of $V_*$ in
$\ol{B}_\sigma\backslash B_\rho$ is compactly contained in $\om$. Applying
the standard monotonicity formula in this annular region gives
\[
\f{\mu_*(\om_\sigma)}{\sigma}
-\f{\mu_*(\om_\rho)}{\rho}
=
\int_{S_*\cap(\om_\sigma\backslash\om_\rho)}
\left|\f{x^\perp}{|x|}\right|^2
\f{\ud\mu_*}{|x|},
\]
where $x^\perp$ denotes the component of $x$ perpendicular to the
approximate tangent line to $S_*$ at $x$. Since the left-hand side is zero,
we obtain
\[
x^\perp=0
\quad\text{for }\mu_*\text{-a.e. }x\in S_*\cap\om_\delta.
\]
Therefore, the approximate tangent line to $S_*$ is radial with respect to
the origin at $\mu_*$-a.e. point. It follows that every connected arc of
$S_*$ in $\om_\delta$ is contained in a line segment emanating from $0$.

Since $S_*\cap\pa B_\delta$ is finite for the chosen admissible radius
$\delta$, we conclude that $S_*\cap\ol{B}_\delta$ consists of finitely many
closed straight segments joining points of $S_*\cap\pa B_\delta$ to the
point $0\in\cA$. This proves the asserted finite-segment structure near the
chosen boundary point $x_0=0$. Since $x_0\in\cA$ was arbitrary and $\cA$ is
finite, Proposition~\ref{interiorproperty1} gives the same finite-segment
structure in the interior away from small neighborhoods of $\cA$. Therefore, there exist finitely many closed line segments $\{L_i\}_{i=1}^n\subset\ol{\om}$ such that $ S_*=\cup_{i=1}^nL_i $. This proves Proposition~\ref{boundarySstarbehave}\ref{boundary11}.
\end{proof}

\begin{proof}[Proof of Proposition~\ref{boundarySstarbehave}\ref{boundary10}]

Suppose $x\in\om$ is an endpoint of $L_i\subset S_*$ but is not a branching point. Since the branching points of $S_*$ are locally finite (Proposition~\ref{interiorproperty1}), there exists
$r\in(0,\frac{1}{10}\dist(x,\pa\om))$ such that $S_*\cap\ol{B}_r(x)=
L_i\cap\ol{B}_r(x)$ is the single segment from $x$ to $x':=S_*\cap\pa B_r(x)$. By translation and rotation, assume $x=0$ and the direction of $L_i$ is $(0,0,1)$, so
\[
S_*\cap\ol{B}_r(0)=\{(0,0,t):0\leq t\leq r\}.
\]
By Proposition~\ref{propustar}, $S_0$ is locally finite in $\om\backslash S_*$. Hence, for $\HH^1$-a.e. $\rho_0\in(0,\f{r}{2})$, we have $\pa B_{\rho_0}\cap S_0=\emptyset$. Fix such a $\rho_0$. Since $S_*\cap\pa B_{\rho_0}=\{(0,0,\rho_0)\}$ and $(0,0,\rho_0)$ has $y_3=\rho_0>\f{\rho_0}{4}$, the spherical cap
\[
C_-:=\pa B_{\rho_0}\cap\left\{y_3\leq\frac{\rho_0}{4}\right\}
\]
satisfies $C_-\cap(S_*\cup S_0)=\emptyset$, so $u_*$ is smooth and $\cN$-valued on $C_-$.

Define the closed disk
\[
D_0:=\ol{B}_{\rho_0}\cap\left\{y_3=\frac{\rho_0}{4}\right\}\subset\om.
\]
Its center $y_0:=(0,0,\f{\rho_0}{4})$ satisfies $y_0\in\op{Int}(L_i)$. Consequently, $D_0\cap S_*=\{y_0\}$ and
$\pa D_0\subset\pa B_{\rho_0}$, so $\pa D_0\cap(S_*\cup S_0)=\emptyset$. By~\eqref{Theta1muy} applied to $y_0$ and $D_0$, the class $[u_*|_{\pa D_0}]_{\cN}$ is non-trivial. On the other hand, $C_-$ is a topological disk with $\pa C_-=\pa D_0$, and
$u_*|_{C_-}:C_-\to\cN$ is smooth. Hence $u_*|_{C_-}$ provides a null-homotopy of $u_*|_{\pa D_0}$ in $\cN$, so $[u_*|_{\pa D_0}]_{\cN}=0$. This contradicts the non-triviality above. Therefore, $x$ must be a branching point, establishing property~\ref{boundary10a}.

Using the first property of Proposition~\ref{boundarySstarbehave}, if $x\in\pa\om$ is an endpoint of $L_i$, then $x\in\cA$. Let $D$ and $D_x$ be given by property~\ref{boundary10b}. Let
\begin{align*}
\cA^{\eta}&:=\{y\in\pa\om:\dist(y,\cA)>\eta\},\\
\om^{\eta}&:=\{y\in\om:\dist(y,\cA)<\eta\}.
\end{align*}
We now follow the argument in Part 3 of this proof. From $S_*\subset\op{Conv}(\cA)$, Proposition~\ref{BoundaryClearingout}, and Proposition~\ref{BoundaryPartialRegularity1prop}, we first deduce that there exists $\eta>0$ such that $ \cA^{\eta}\cap\om^{\eta}\cap S_*=\emptyset $ and
\[
u_*\in C^{\ift}(\cA^{\eta}\cap\om^{\eta},\cN)
\cap C^0(\ol{\cA}^{\eta}\cap\ol{\om}^{\eta},\cN).
\]
Since $S_0$ is locally finite in $\om$, there exists a $2$-disk $D'$ such that $\ol{D}'\subset\cA^{\eta}\cap\om^{\eta}$ and $ [u_*|_{\pa D'}]_{\cN}=[u_*|_{\pa D}]_{\cN} $. Note that $[u_*|_{\pa D}]_{\cN}=[g|_{D_x}]_{\cN}$. We now verify \eqref{gDxeq}.
\end{proof}

\subsection{Proof of Theorem \ref{globalminimizersproperties}}
The bound~\eqref{EvaCbound} follows from Proposition~\ref{locallogestimate}. The first property follows from Lemma~\ref{stationvari}. The second and third properties are established in Proposition~\ref{propustar}.

\subsection{Proof of Theorem \ref{structureofthelimitset}}

The first property is a consequence of Proposition~\ref{boundarySstarbehave}\ref{boundary1}--\ref{boundary11}. The second and third properties are derived from
Proposition~\ref{interiorproperty1} and Proposition~\ref{interiorproperty11}, respectively. Finally, the fourth property follows from
Proposition~\ref{boundarySstarbehave}\ref{boundary10}.

\section{Uniform estimates}\label{SectionUniform}

In this section, we establish uniform estimates for the minimizers of \eqref{GLfunctional} and prove Proposition~\ref{W1pestimate}. The results in this section follow the approach developed in \cite{FWW25a, FWW25b}. The arguments differ from those in \cite{BSV25} in the treatment of the bad sets and the covering procedure.

\subsection{Regular scales and bad sets}
In this subsection, we define the regular scales, which characterize the regularity of maps quantitatively.

\begin{defn}[Regular scales]\label{regular}
Let $\va\in(0,1)$ and $\Lda>0$. Assume that $U\subset\R^3$ is a bounded domain and $u\in C^{\ift}(U,\R^m)$. For $x\in U$, we set
\begin{align*}
r(u;x)&:=\sup\left\{0\leq r\leq 1:r^2\|e_{\va}(u)\|_{L^{\ift}(U_r(x))}\leq 1\right\},\\
r^{\Lda}(u;x)&:=\sup\left\{0\leq r\leq 1:\theta_{\va}^U(u;x,r)\leq \Lda\right\}.
\end{align*}
For $\Lda,r>0$, define the type I and II bad sets of $u$ in $U$ by
\begin{align*}
\op{Bad}_{\op{I}}(u;r)&:=\{y\in U:r(u;y)<r\},\\
\op{Bad}_{\op{II}}(u;r,\Lda)&:=\{y\in U:r^{\Lda}(u;y)<r\}.
\end{align*}
\end{defn}

\subsection{Interior estimates}
In this subsection, we discuss the uniform interior estimates.

\subsubsection{Interior \texorpdfstring{$W^{1,q}$}{}-estimates}

The first main result is the following uniform interior $W^{1,q}$-estimate for local minimizers.

\begin{prop}\label{InteriorW1p}
Let $M>0$. Assume that $\{u_{\va}\}_{\va\in(0,1)}\subset H^1(B_{10},\R^m)$ is a family of local minimizers of \eqref{GLfunctional}, satisfying
\[
E_{\va}(u_{\va},B_{10})\leq M(|\log\va|+1),
\quad
\|u_{\va}\|_{L^{\ift}(B_{10})}\leq M.
\]
Then, for any $q\in(1,2)$,
\[
\|\na u_{\va}\|_{L^q(B_1)}\leq C,
\]
where $C>0$ depends only on $f,M,\cN$, and $q$.
\end{prop}

The proof proceeds through a chain of covering lemmas. We begin by identifying a sufficient condition for a point to lie outside the bad set at a given scale.

\begin{lem}\label{regularscalelem}
Let $M>0$, $0<\va<r<1$, and $x_0\in\R^3$. Assume that $u\in H^1(B_{10r}(x_0),\R^m)$ is a local minimizer of \eqref{GLfunctional}, satisfying
\[
\theta_{\va}(u;x_0,10r)+\|u\|_{L^{\ift}(B_{10r}(x_0))}\leq M.
\]
There exists $\delta\in(0,1)$, depending only on $f,M,\cN$, such that if $\va\in(0,\delta r)$ and
\begin{gather*}
\vt_{\va}(u;x_0,r)-\vt_{\va}\(u;x_0,\f{r}{2}\)<\delta,\\
\inf_{v\in\Ss^2}\(\f{1}{r}\int_{B_r(x_0)}|v\cdot\na u|^2\)<\delta,
\end{gather*}
then $r(u;x_0)\geq \f{r}{2}$. In particular, $x_0\notin\op{Bad}_{\op{I}}(u;\f{r}{2})$.
\end{lem}

\begin{proof}
By translation, assume $x_0=0$. Suppose, in contradiction, that the conclusion fails. Then there exist local minimizers $\{u_i\}\subset H^1(B_{10r_i},\R^m)$ of \eqref{GLfunctional} and $\{r_i\}\subset\R_+$ such that $\wt{\va}_i:=\f{\va_i}{r_i}<\delta_i\to 0^+$, and
\begin{gather}
\vt_{\va_i}(u_i;0,r_i)-\vt_{\va_i}\(u_i;0,\f{r_i}{2}\)<\delta_i,\label{assum1}\\
\inf_{v\in\Ss^2}\(\f{1}{r_i}\int_{B_{r_i}}|v\cdot\na u_i|^2\)<\delta_i,\label{assum2}
\end{gather}
but $r(u_i;0)<\f{r_i}{2}$ for all $i$. Let $\wt{u}_i:=u_i(r_i\cdot)$. By Proposition~\ref{interiorcompactnesslem}, passing to a subsequence,
\[
\wt{u}_i\to\wt{u}_0\quad\text{strongly in }H_{\loc}^1(B_{10},\R^m),
\]
where $\wt{u}_0\in H^1(B_{10},\cN)$ is a local minimizer of \eqref{Dirichlet}. It follows from \eqref{assum2} that $\wt{u}_0$ is invariant in some direction and degenerates to a two-dimensional minimizer, which is smooth by \cite[Theorem II]{SU82}. Using \eqref{phiproperty1}, \eqref{Monotone1}, and \eqref{assum1}, we deduce that $\wt{u}_0$ is homogeneous and
\[
\lim_{i\to+\ift}\f{1}{\wt{\va}_i^2}\int_{B_4}f(\wt{u}_i)=0.
\]
Consequently, $\wt{u}_0$ must be a constant and 
\[
\theta_{\va_i}(u_i;0,r_i)=\theta_{\wt{\va}_i}(\wt{u}_i;0,1)\to 0^+
\]
as $i\to+\ift$. This contradicts Proposition~\ref{InteriorPartialRegularity1}.
\end{proof}

Based on Lemma~\ref{regularscalelem}, the next result shows that a small directional energy condition in $B_r(x_0)$ gives a lower bound on the regular scale.

\begin{lem}\label{kplus1}
Let $\va\in(0,1)$, $M>0$, $r\in(0,1)$, and $x_0\in\R^3$. Assume that $u\in H^1(B_{2r}(x_0),\R^m)$ is a local minimizer of \eqref{GLfunctional}, satisfying
\[
\theta_{\va}(u;x_0,2r)+\|u\|_{L^{\ift}(B_{2r}(x_0))}\leq M.
\]
There exists $\delta\in(0,1)$, depending only on $f,M,\cN$, such that if $\va\in(0,\delta r)$ and
\be
\inf_{v\in\Ss^2}\(\f{1}{r}\int_{B_r(x_0)}|v\cdot\na u|^2\)<\delta,\label{detaVsmall}
\ee
then $r(u;y)\geq\f{\delta^{\f{1}{2}}r}{2}$ for any $y\in B_{\f{r}{2}}(x_0)$.
\end{lem}

\begin{proof}
It follows from \eqref{Monotone1} and a dyadic decomposition of the radius $r$ that, for some $r_y\in[\delta^{\f{1}{2}}r,\f{r}{40}]$ with $\delta\in(0,\f{1}{40^2})$,
\be
\vt_{\va}(u;y,r_y)-\vt_{\va}\(u;y,\f{r_y}{2}\)<\f{C}{|\log\delta|}.\label{pinchsmall}
\ee
From \eqref{detaVsmall}, we estimate
\[
\inf_{v\in\Ss^2}\(\f{1}{r_y}\int_{B_{r_y}(y)}|v\cdot\na u|^2\)
\leq
\inf_{v\in\Ss^2}\(\f{1}{\delta^{\f{1}{2}}r}\int_{B_r(x_0)}|v\cdot\na u|^2\)
<\delta^{\f{1}{2}}.
\]
Choosing $\delta=\delta(f,M,\cN)\in(0,1)$ sufficiently small and applying Lemma~\ref{regularscalelem} together with \eqref{pinchsmall}, we conclude that if $\va\in(0,\delta r)$, then $r(u;y)\geq \f{r_y}{2}\geq\f{\delta^{\f{1}{2}}r}{2}$.
\end{proof}

The following lemma localizes the type I bad set to a small neighborhood of a point where the energy-pinching condition holds.

\begin{lem}\label{Fprop}
Let $\beta\in(0,\f{1}{2})$, $0<\va<r<1$, $M>0$, and $x_0\in\R^3$. Assume that $u\in H^1(B_{20r}(x_0),\R^m)$ is a local minimizer of \eqref{GLfunctional}, satisfying
\[
\theta_{\va}(u;x_0,20r)+\|u\|_{L^{\ift}(B_{20r}(x_0))}\leq M.
\]
There exists $\delta\in(0,1)$, depending only on $f,M,\cN$, and $\beta$, such that if $\va\in(0,\delta r)$ and
\be
\vt_{\va}(u;y,r)-\vt_{\va}\(u;y,\f{r}{2}\)<\delta\label{pincheta}
\ee
for some $y\in B_{2r}(x_0)$, then $ \op{Bad}_{\op{I}}(u;\delta r)\cap B_r(x_0)\subset B_{2\beta r}(y) $.
\end{lem}

\begin{proof}
Fix $z\in B_r(x_0)\backslash B_{2\beta r}(y)$. There exists $\sg=\sg(\beta)>0$ such that
\be
B_{\sg r}(z)\subset B_{4r}(y)\cap(B_r(x_0)\backslash B_{2\beta r}(y)).\label{Bsgrz}
\ee
It follows from \eqref{Monotone1} and \eqref{pincheta} that
\be
\int_{B_{4r}(y)}|(\zeta-y)\cdot\na u|^2\ud\zeta\leq C\delta r^3.\label{leqCr3}
\ee
Since $|z-y|\geq 2\beta r$, we compute
\begin{align*}
\int_{B_{\sg r}(z)}\left|\f{z-y}{|z-y|}\cdot\na u\right|^2
&\leq \f{C}{r^2}\(\int_{B_{\sg r}(z)}|(\zeta-y)\cdot\na u|^2\ud\zeta
+\int_{B_{\sg r}(z)}|(\zeta-z)\cdot\na u|^2\ud\zeta\)\\
&\stackrel{\eqref{leqCr3}}{\leq} C\delta r+2\sg^2r\(\f{1}{r}\int_{B_{\sg r}(z)}|\na u|^2\)\\
&\leq C(\beta,M)(\delta+2\sg^2)r,
\end{align*}
where the last inequality uses \eqref{Monotone1} again. After decreasing $\delta$ if necessary, we may take $\sg=\delta^{\f{1}{2}}$ in \eqref{Bsgrz}. Then
\[
\inf_{v\in\Ss^2}\(\f{1}{\sg r}\int_{B_{\sg r}(z)}|v\cdot\na u|^2\)
\leq C(\beta,M)\delta^{\f{1}{2}}.
\]
Choosing $\delta=\delta(f,M,\cN,\beta)\in(0,1)$ sufficiently small, Lemma~\ref{kplus1} yields $\delta'=\delta'(f,M,\cN,\beta)\in(0,1)$ such that $r(u;z)\geq\delta'r$ whenever $\va\in(0,\delta r)$. The replacement of $\delta$ with $\min\{\delta,\delta'\}$ completes the proof.
\end{proof}

We now combine the preceding estimates to confine the type I bad set to a single ball, subject to a quantitative energy-drop condition on its complement.

\begin{lem}\label{cover1}
Let $M>0$, $0<\va<r<R\leq 1$, and $x_0\in\R^3$. Assume that $u\in H^1(B_{40R}(x_0),\R^m)$ is a local minimizer of \eqref{GLfunctional}, satisfying
\[
\theta_{\va}(u;x_0,40R)+\|u\|_{L^{\ift}(B_{40R}(x_0))}\leq M.
\]
There exists $\delta\in(0,1)$, depending only on $f,M,\cN$, such that if $\va\in(0,\delta r)$, the following holds: there is a ball $B_{2r_x}(x)\subset B_{2R}(x_0)$ with $r_x\geq r$ such that
\[
\op{Bad}_{\op{I}}(u;\delta r)\cap B_R(x_0)\subset B_{r_x}(x).
\]
Moreover, either $r_x=r$, or
\[
\sup_{y\in B_{2r_x}(x)}\vt_{\va}\(u;y,\f{r_x}{20}\)
\leq
\sup_{y\in B_{2R}(x_0)}\vt_{\va}(u;y,R)-\delta.
\]
\end{lem}

\begin{proof}
Up to a translation, assume $x_0=0$. For $x\in B_R$ and $0<\rho<R$, define
\[
F_{\delta}(u;x,\rho):=\left\{y\in B_{2\rho}(x):\vt_{\va}\(u;y,\f{\rho}{20}\)>E-\delta\right\},
\]
where $E:=\sup_{y\in B_{2R}}\vt_{\va}(u;y,R)$. Choose $j_0\in\N$ such that $2^{-j_0}R\leq r<2^{-j_0+1}R$, and set $r_j:=\f{R}{2^j}$ for $j\in\Z\cap[1,j_0-1]$. If $F_{\delta}(u;0,R)=\emptyset$, then $B_R$ is the desired ball. We therefore assume $F_{\delta}(u;0,R)\neq\emptyset$.

For $j=1$, choose $x_1\in F_{\delta}(u;0,R)$. Since
\[
\vt_{\va}(u;x_1,R)-\vt_{\va}\(u;x_1,\f{R}{20}\)<\delta,
\]
Lemma~\ref{Fprop}, applied with $\beta=\f{1}{20}$, gives
\[
\op{Bad}_{\op{I}}(u;\delta r)\cap B_R\subset B_{\f{R}{10}}(x_1).
\]
If $F_{\delta}(u;x_1,r_1)=\emptyset$, the ball $B_{r_1}(x_1)$ is the desired one. Otherwise, we apply Lemma~\ref{Fprop} again to obtain $x_2$ with
\[
\op{Bad}_{\op{I}}(u;\delta r)\cap B_R\subset B_{\f{R}{20}}(x_2).
\]
Repeating this procedure, we either stop at some step $j\in\Z\cap[1,j_0-1]$, or arrive at $B_{r_{j_0-1}}(x_{j_0-1})$ satisfying
\[
\op{Bad}_{\op{I}}(u;\delta r)\cap B_R
\subset B_{\f{R}{10\cdot 2^{j_0-2}}}(x_{j_0-1}).
\]
In the latter case, the ball $B_r(x_{j_0-1})$ satisfies all the properties required.
\end{proof}

\begin{rem}\label{remcover}
Under the assumptions of Lemma~\ref{cover1}, a further covering argument yields a collection of balls $ \{B_{r_y}(y)\}_{y\in\cD} $ with the following properties.
\begin{itemize}
\item $ \inf_{y\in\cD}r_y\geq r $.
\item $ \op{Bad}_{\op{I}}(u;\delta r)\cap B_R(x_0)\subset\cup_{y\in\cD}B_{r_y}(y) $.
\item $ \#\cD\leq C $, where $C>0$ is an absolute constant.
\item For any $ y\in\cD $, either $ r_y=r $ or
\be
\sup_{z\in B_{2r_y}(y)}\vt_{\va}(u;z,r_y)\leq\sup_{z\in B_{2R}(x_0)}\vt_{\va}(u;z,R)-\delta.\label{energydrop}
\ee
\end{itemize}
\end{rem}

Iterating the energy-drop condition in Remark~\ref{remcover}, we obtain a finite covering of the type I bad set by balls of the prescribed scale.

\begin{lem}\label{coverpoint}
Let $ M,\va,r,R,x_0 $, and $ u $ be as in Lemma~\ref{cover1}. There exist $ \delta\in(0,1) $ and $ n\in\N $, depending only on $ f,M,\cN $, such that if $ \va\in(0,\delta r) $, then there are points $ \{x_i\}_{i=1}^n\subset B_R(x_0) $ satisfying
\[
\op{Bad}_{\op{I}}(u;\delta r)\cap B_R(x_0)\subset\bigcup_{i=1}^nB_r(x_i).
\]
In particular,
\[
\cL^3(B_r(\op{Bad}_{\op{I}}(u;\delta r)\cap B_R(x_0)))\leq Cr^3,
\]
where $ C>0 $ depends only on $ f,M $, and $ \cN $.
\end{lem}

\begin{proof}
We apply Remark~\ref{remcover} to $B_R(x_0)$. For each ball in the covering whose radius is larger than $r$, we apply Remark~\ref{remcover} again on that ball. Each such application produces an energy drop of $\delta$ as in \eqref{energydrop}. Since $\vt_{\va}$ is nonnegative and the initial supremum is bounded under the assumptions of Lemma~\ref{cover1}, the iteration stops after finitely many steps. The number of generations, and hence the total number of final balls, is bounded by a constant depending only on $f,M$, and $\cN$. This proves the covering statement. The volume estimate follows by covering the $r$-neighborhood of the bad set by balls of radius $2r$.
\end{proof}

We turn next to the type II bad set and establish an analogous covering with an explicit bound on the number of balls.

\begin{lem}\label{coveringlemma2}
Let $ M>0 $ and $ r\in(0,1) $. Assume that $ u\in H^1(B_{40},\R^m) $ is a local minimizer of \eqref{GLfunctional}, satisfying
\be
E_{\va}(u,B_{40})\leq M(|\log\va|+1),\quad\|u\|_{L^{\ift}(B_{40})}\leq M.\label{thetava40}
\ee
For any $ \ga\in(0,1) $, there exist $ \delta\in(0,1) $ and $ \Lda>0 $, depending only on $ f,\ga,M,\cN $, such that when $ \va\in(0,\delta) $ and $ r\in(\va^{\ga},1) $, there exist points $ \{x_i\}_{i=1}^n\subset B_1 $ with $ n\leq\Lda r^{-1} $ satisfying
\[
B_r(\op{Bad}_{\op{II}}(u;r,\Lda)\cap B_1)\subset\bigcup_{i=1}^nB_r(x_i).
\]
\end{lem}

\begin{proof}
If $ r\in[\f{1}{10},1) $, the result is trivial. We therefore assume $ r\in(\va^{\ga},\f{1}{10}) $. By Proposition~\ref{InteriorClearingout}, there exists $ \eta=\eta(f,M,\cN)\in(0,1) $ such that for any $ y\in B_2 $, if $ \va\in(0,\eta r) $ and
\[
\theta_{\va}(u;y,2r)\leq\eta\log\f{r}{\va},
\]
then $ \theta_{\va}(u;y,r)\leq \f{1}{2}C_0(f,M,\cN) $. Choose $ \delta=\delta(\eta,\ga)\in(0,1) $ so that, when $ \va\in(0,\delta) $, we have $ \va\in(0,\eta\va^{\ga})\subset(0,\eta r) $ and
\[
\eta\log\f{r}{\va}\geq\eta\log\f{\va^{\ga}}{\va}\geq\delta|\log\va|.
\]
For any $ y\in E:=\op{Bad}_{\op{II}}(u;r,C_0)\cap B_1 $, the preceding clearing-out statement gives
\be
\theta_{\va}(u;y,2r)\geq\delta|\log\va|.\label{thetageq}
\ee
Indeed, otherwise, we would have $\theta_{\va}(u;y,r)\leq \f{C_0}{2}$, contradicting the definition of $E$.

The family $ \{B_{2r}(y)\}_{y\in E} $ covers $ B_r(E) $. By Vitali's covering lemma, there is a subcollection of pairwise disjoint balls $ \{B_{2r}(y_i)\}_{i\in I} $ such that
\[
B_r(E)\subset\bigcup_{y\in E}B_{2r}(y)\subset\bigcup_{i\in I}B_{10r}(y_i).
\]
Using \eqref{thetava40} and \eqref{thetageq}, we estimate
\[
\sum_{i\in I}r\leq C\sum_{i\in I}\f{E_{\va}(u,B_{2r}(y_i))}{|\log\va|}
\leq \f{C}{|\log\va|}\int_{\bigcup_{i\in I}B_{2r}(y_i)}e_{\va}(u)\leq C.
\]
Covering each $B_{10r}(y_i)$ by balls of radius $r$, we obtain a collection $\{B_r(x_i)\}_{i=1}^n$ with
\[
B_r(E)\subset\bigcup_{i\in I}B_{10r}(y_i)\subset\bigcup_{i=1}^nB_r(x_i)
\]
and $ n\leq C_1(f,M,\cN)r^{-1} $. Setting $ \Lda=\max\{C_0,C_1\} $ completes the proof.
\end{proof}

We are now ready to estimate the measure of the type I bad set.

\begin{lem}\label{interiorlemma}
Let $ M>0 $ and $ \va,r\in(0,1) $. Assume that $ u\in H^1(B_{40},\R^m) $ is a local minimizer of \eqref{GLfunctional}, satisfying
\[
E_{\va}(u,B_{40})\leq M(|\log\va|+1),\quad\|u\|_{L^{\ift}(B_{40})}\leq M.
\]
For any $ \ga\in(0,1) $, there exist $ \delta\in(0,1) $ and $ C>0 $, depending only on $ f,\ga,M,\cN $, such that if $ \va\in(0,\delta) $ and $ r\in(\va^{\ga},1) $, then
\[
\cL^3\(\op{Bad}_{\op{I}}(u;\delta r)\cap B_{\f{1}{4}}\)\leq Cr^2.
\]
\end{lem}

\begin{proof}
Fix $ \ga\in(0,1) $. The case $r\geq\f{1}{10}$ is trivial, so assume $ r\in(\va^{\ga},\f{1}{10}) $. Let $ \Lda>0 $ be chosen below. Pick $ j_0\in\N $ such that $ 2^{j_0}r\leq \f{1}{4}<2^{j_0+1}r $. For $ j\in\Z\cap[0,j_0+1] $, set $ F_j:=\op{Bad}_{\op{II}}(u;2^jr,\Lda)\cap B_1 $ and
\[
\cA_j:=(B_{2^{j}r}(F_j)\backslash B_{2^{j-1}r}(F_{j-1}))\cap B_{\f{1}{4}}, \quad j\in\Z\cap[1,j_0+1].
\]
For $ x\in\cA_j $, we have $ B_{2^{j-1}r}(x)\cap F_{j-1}=\emptyset $, which implies
\be
\theta_{\va}(u;x,2^{j-1}r)\leq\Lda.\label{thetavaleqLda}
\ee
Moreover,
\be
B_{\f{1}{4}}\subset B_r(F_0)\cup\bigcup_{j=1}^{j_0+1}\cA_j.\label{B1decompose}
\ee
By Lemma~\ref{coveringlemma2}, there exist $ \Lda=\Lda(f,\ga,M,\cN)>0 $ and, for each $ j\in\Z\cap[0,j_0+1] $, a collection of balls $ \{B_{2^jr}(x_{jk})\}_{k=1}^{n_j} $ such that
\be
B_{2^jr}(F_j)\subset\bigcup_{k=1}^{n_j}B_{2^jr}(x_{jk}),\quad 0\leq n_j\leq\Lda(2^jr)^{-1}.\label{Ajcover}
\ee
For each $ j\in\Z\cap[1,j_0+1] $ and $ k\in\Z\cap[1,n_j] $, cover $ \cA_j\cap B_{2^jr}(x_{jk}) $ with balls $ \{B_{r_j}(x_{jk}^{(i)})\}_{i=1}^{n_{jk}} $ with $ r_j:=\f{2^jr}{10} $, where $ n_{jk}\leq C_0 $ for some absolute constant $ C_0>0 $ and $ \{x_{jk}^{(i)}\}\subset\cA_j $. Collecting these balls, we obtain $ \{B_{r_j}(x_{jk}')\}_{k=1}^{n_j'} $ such that, for each $ j\in\Z\cap[1,j_0+1] $,
\be
\cA_j\subset\bigcup_{k=1}^{n_j'}B_{r_j}(x_{jk}'),\quad \{x_{jk}'\}_{k=1}^{n_j'}\subset\cA_j,\quad 0\leq n_j'\leq C(2^jr)^{-1}.\label{Njestimate}
\ee
Since $ r\in(\va^{\ga},\f{1}{10}) $, applying \eqref{thetavaleqLda} and Lemma~\ref{coverpoint} after scaling gives $ \delta=\delta(f,\ga,M,\cN)\in(0,1) $ such that, for each ball $ B_{r_j}(x_{jk}') $ with $ j\in\Z\cap[1,j_0+1] $ and $ \va\in(0,\delta) $,
\[
\cL^3(\op{Bad}_{\op{I}}(u;\delta r)\cap B_{r_j}(x_{jk}'))\leq Cr^3.
\]
Together with \eqref{B1decompose}, \eqref{Ajcover}, and \eqref{Njestimate}, this gives
\[
\cL^3(\op{Bad}_{\op{I}}(u;\delta r)\cap B_{\f{1}{4}})
\leq C r^2+\sum_{j=1}^{j_0+1}C(2^{j}r)^{-1}r^3
\leq C(f,\ga,M,\cN)r^2,
\]
completing the proof.
\end{proof}

\begin{proof}[Proof of Proposition~\ref{InteriorW1p}]
Fix $ \ga\in(0,1) $. Applying Lemma~\ref{interiorlemma} on finitely many balls covering $B_1$, after scaling if necessary, we obtain
\[
\cL^3(\op{Bad}_{\op{I}}(u_{\va};\delta r)\cap B_1)\leq Cr^2
\]
for any $ r\in(\va^{\ga},1) $ and $ \va\in(0,\delta(f,\ga,M,\cN)) $. For $ R\in(\va,1) $, since $ R^{\ga}\in(\va^{\ga},1) $ and $ R<R^{\ga} $, we have
\be
\cL^3(\{y\in B_1:r(u_{\va};y)<\delta R\})
\leq \cL^3(\{y\in B_1:r(u_{\va};y)<\delta R^{\ga}\})
\leq CR^{2\ga}.\label{Rlarge}
\ee
For $ R\in(0,\va) $ and any $ y\in B_1 $, Lemma~\ref{Apriori} gives
\[
R|\na u(y)|\leq CR(\va^{-1}+R^{-1})\leq C.
\]
Thus, choosing $\delta$ smaller if necessary, we have $r(u_{\va};y)\geq\delta R$ for any $y\in B_1$, and hence
\[
\cL^3(\op{Bad}_{\op{I}}(u_{\va};\delta R)\cap B_1)=0.
\]
Combining this with \eqref{Rlarge}, we obtain, for any $ R\in(0,1) $,
\[
\cL^3(\{y\in B_1:r(u_{\va};y)<\delta R\})\leq C(f,\ga,M,\cN)R^{2\ga}.
\]
By Definition~\ref{regular}, this means $ \na u_{\va}\in L^{2\ga,\ift} $. Choosing $ \ga >\f{q}{2} $ completes the proof.
\end{proof}

\subsubsection{Interior estimates on the potential}

We next establish a uniform bound on the potential energy. The argument combines the covering structure from the previous subsection with a pointwise decay estimate near regular points.

\begin{prop}\label{Interiorpotential}
Let $ M>0 $ and $ \va\in(0,1) $. Assume that $ \{u_{\va}\}_{\va\in(0,1)}\subset H^1(B_{10},\R^m) $ is a family of local minimizers of \eqref{GLfunctional}, satisfying
\[
E_{\va}(u_{\va},B_{10})\leq M(|\log\va|+1),\quad\|u_{\va}\|_{L^{\ift}(B_{10})}\leq M.
\]
Then there exists $ C>0 $, depending only on $ f,M $, and $ \cN $, such that
\[
\f{1}{\va^2}\int_{B_1}f(u_{\va})\leq C.
\]
\end{prop}

The proof uses two preparatory lemmas. The first gives the estimate on the potential when the energy density is uniformly bounded.

\begin{lem}\label{va3estimates}
Let $ M>0 $, $ r\in(0,1] $, $ \va\in(0,r) $, and $ x_0\in\R^3 $. Assume that $ u\in H^1(B_{4r}(x_0),\R^m) $ is a local minimizer of \eqref{GLfunctional} such that
\[
\theta_{\va}(u;x_0,4r)+\|u\|_{L^{\ift}(B_{4r}(x_0))}\leq M.
\]
Then
\be
\f{1}{\va^2}\int_{B_r(x_0)}f(u)\leq C\va,\label{va3for}
\ee
where $ C>0 $ depends only on $ f,M $, and $ \cN $.
\end{lem}

\begin{proof}
By scaling and translation, assume $ r=1 $ and $ x_0=0 $. Fix $ 0<\nu<1 $. By Lemma~\ref{coverpoint}, there exists $ \delta=\delta(f,M,\cN)\in(0,1) $ such that, for any $ 0<R<1 $ with $ \va\in(0,\delta R) $,
\be
\cL^3(B_R(\op{Bad}_{\op{I}}(u;\delta R)\cap B_1))\leq CR^3.\label{applyBR}
\ee
For any $ \va\in(0,\delta) $, choose $ n(\va)\in\N $ such that $ \nu^{n(\va)-1}\in[\delta^{-1}\va,\nu^{-1}\delta^{-1}\va] $. Decompose
\be
\begin{aligned}
\int_{B_1}f(u)
&\leq \int_{B_{\delta^{-1}\va}(\op{Bad}_{\op{I}}(u;\va)\cap B_1)}f(u)
+\sum_{j=0}^{n(\va)-2}\int_{\cA_j}f(u)+\int_{B_1\backslash B_1(\op{Bad}_{\op{I}}(u;\delta)\cap B_1)}f(u),
\end{aligned}\label{use}
\ee
where
\[
\cA_j:=B_{\nu^j}(\op{Bad}_{\op{I}}(u;\delta\nu^j)\cap B_1)\backslash B_{\nu^{j+1}}(\op{Bad}_{\op{I}}(u;\delta\nu^{j+1})\cap B_1).
\]
From \eqref{applyBR},
\be
\begin{gathered}
\cL^3(\cA_j)\leq C(\delta\nu^j)^3\quad\text{for any }j\in\Z\cap[0,n(\va)-2],\\
\cL^3(B_{\delta^{-1}\va}(\op{Bad}_{\op{I}}(u;\va)\cap B_1))\leq C(\delta^{-1}\va)^3.
\end{gathered}\label{cAestimatevolume}
\ee
By Proposition~\ref{improvedpotential},
\[
f(u)\leq
\begin{cases}
C &\text{in }B_{\delta^{-1}\va}(\op{Bad}_{\op{I}}(u;\va)\cap B_1),\\
C\va^4 &\text{in }B_1\backslash B_1(\op{Bad}_{\op{I}}(u;\delta)\cap B_1),\\
C\va^4\nu^{-4(j+1)} &\text{in }\cA_j.
\end{cases}
\]
Together with \eqref{use} and \eqref{cAestimatevolume}, this yields
\[
\int_{B_1}f(u)\leq C\left(\va^3+\sum_{j=0}^{n(\va)-2}\va^{4}\nu^{-(j+1)}+\va^4\right)\leq C\va^3,
\]
completing the proof.
\end{proof}

The second lemma provides a logarithmic-scale estimate of the potential near points in the type II bad set.

\begin{lem}\label{badsetcoveruse}
Let $ \va\in(0,\f{1}{10}) $ and $ M>0 $. Assume that $ u\in H^1(B_4,\R^m) $ is a local minimizer of \eqref{GLfunctional} with $ \|u\|_{L^{\ift}(B_4)}\leq M $. There exist $ \delta\in(0,1) $ and $ \Lda>0 $, depending only on $ f,M,\cN $, such that the following property holds. If $ \va\in(0,\delta) $ and $ x\in B_1 $ satisfy
\be
\dist\left(x,\op{Bad}_{\op{II}}\left(u;\f{\va^{\f{1}{4}}}{2},\Lda\right)\cap B_1\right)<\f{\va^{\f{1}{4}}}{2},\label{badsetcon}
\ee
then there exists $ r_x\in[\va^{\f{1}{4}},\va^{\f{1}{8}}] $ such that
\be
\int_{B_{r_x}(x)}\f{1}{\va^2} f(u)
\leq\f{C}{|\log\va|}\log\left(2+\f{\theta_{\va}(u;x,\va^{\f{1}{8}})}{|\log\va|}\right)r_x\theta_{\va}(u;x,r_x),\label{badsetresult}
\ee
where $ C>0 $ depends only on $ f,M $, and $ \cN $.
\end{lem}

\begin{proof}
By Proposition~\ref{MonotoneInterior},
\be
\f{\ud}{\ud r}\theta_{\va}(u;x,r)\geq\f{2}{\va^2r^2}\int_{B_r(x)}f(u).\label{FGpre}
\ee
Choose $\delta>0$ small enough so that $2\va^{\f{1}{4}}<\va^{\f{1}{8}}$ whenever $\va\in(0,\delta)$. Define
\[
F_{\va}(r):=\theta_{\va}(u;x,r),\quad G_{\va}(r):=\f{2}{\va^2r}\int_{B_r(x)}f(u),
\]
\[
\wt{F}_{\va}(s):=F_{\va}(\exp(s)),\quad \wt{G}_{\va}(s):=G_{\va}(\exp(s)),
\]
and
\[
I_{\va}:=[s_{\va}^{0},s_{\va}^{1}]
=\left[\log(2\va^{\f{1}{4}}),\f{1}{8}\log\va\right].
\]
For $\va\in(0,\delta)$, we have $s_\va^0<s_\va^1<0$. From \eqref{FGpre},
\be
\f{\ud}{\ud s}\wt{F}_{\va}(s)\geq \wt{G}_{\va}(s)\quad\text{for any }s\in I_{\va}.\label{fdsg}
\ee
We claim that there exists $ s_{\va}\in I_{\va} $ such that
\be
\wt{G}_{\va}(s_{\va})\leq\lda_{\va}\wt{F}_{\va}(s_{\va}),\quad
\lda_{\va}:=\f{1}{s_{\va}^{1}-s_{\va}^{0}}
\log\left(\f{\wt{F}_{\va}(s_{\va}^1)}{\wt{F}_{\va}(s_{\va}^0)}\right).\label{svaex}
\ee
Suppose not. Then $\wt{G}_{\va}(s)>\lda_{\va}\wt{F}_{\va}(s)$ for any $s\in I_\va$. Combined with \eqref{fdsg}, this gives
\[
\f{\ud}{\ud s}\left(\exp(-\lda_{\va}s)\wt{F}_{\va}(s)\right)
=\exp(-\lda_\va s)\left(\f{\ud}{\ud s}\wt{F}_\va(s)-\lda_\va \wt{F}_\va(s)\right)>0
\quad\text{on }I_\va.
\]
Integrating from $s_\va^0$ to $s_\va^1$ yields
\[
\wt{F}_{\va}(s_{\va}^1)>\exp(\lda_{\va}(s_{\va}^1-s_{\va}^0))\wt{F}_{\va}(s_{\va}^0)=\wt{F}_{\va}(s_{\va}^1),
\]
a contradiction, where the equality uses the definition of $\lda_{\va}$ in \eqref{svaex}. This proves the claim.

Set $ r_x:=\exp(s_{\va}) $. Then $r_x\in[2\va^{\f{1}{4}},\va^{\f{1}{8}}]\subset[\va^{\f{1}{4}},\va^{\f{1}{8}}]$. Since $s_\va^1-s_\va^0\geq C^{-1}|\log\va|$ for $\va$ small, \eqref{svaex} gives
\be
\f{2}{\va^2}\int_{B_{r_x}(x)}f(u)
\leq\f{C}{|\log\va|}
\log\left(\f{\theta_{\va}(u;x,\va^{\f{1}{8}})}
{\theta_{\va}(u;x,2\va^{\f{1}{4}})}\right)
r_x\theta_{\va}(u;x,r_x).\label{fQlogva}
\ee

It remains to bound the denominator from below. By \eqref{badsetcon}, there exists
\[
y\in\op{Bad}_{\op{II}}\left(u;\f{\va^{\f{1}{4}}}{2},\Lda\right)\cap B_1
\]
such that $|x-y|<\f{1}{2}\va^{\f{1}{4}}$. Choosing $\Lda\geq C_0(f,M,\cN)$ and applying Proposition~\ref{InteriorClearingout} as in the proof of Lemma~\ref{coveringlemma2}, we obtain, for $\va$ small enough, $ \theta_{\va}(u;y,\va^{\f{1}{4}})\geq C^{-1}|\log\va| $. Since $B_{\va^{\f{1}{4}}}(y)\subset B_{2\va^{\f{1}{4}}}(x)$, it follows that
\[
\theta_{\va}(u;x,2\va^{\f{1}{4}})\geq C^{-1}|\log\va|.
\]
Substituting this lower bound into \eqref{fQlogva} yields \eqref{badsetresult}.
\end{proof}

\begin{proof}[Proof of Proposition~\ref{Interiorpotential}]
Set $r_{\va}:=\f{\va^{\f{1}{4}}}{2}$, and choose $j_0\in\N$ such that
$2^{j_0}r_{\va}<1\leq 2^{j_0+1}r_{\va}$. For
$j\in\Z\cap[0,j_0+1]$, define
\[
F_j:=B_{2^jr_{\va}}\(\op{Bad}_{\op{II}}\(u_{\va};2^jr_{\va},\Lda\)\cap B_1\),
\]
and, for $j\in\Z\cap[1,j_0+1]$, set $ \cA_j:=(F_j\backslash F_{j-1})\cap B_1 $. Then
\be
B_1\subset F_0\cup\bigcup_{j=1}^{j_0+1}\cA_j.
\label{B1decompose2}
\ee

Applying Lemma~\ref{coveringlemma2} with $\ga=\f{1}{16}$, choose
$\delta=\delta(f,M,\cN)>0$ such that, when $\va\in(0,\delta)$, for any
$j\in\Z\cap[0,j_0+1]$ there is a collection of balls
$\{B_{2^jr_{\va}}(x_{jk})\}_{k=1}^{n_j}$ satisfying
\[
\cA_j\subset\bigcup_{k=1}^{n_j}B_{2^jr_{\va}}(x_{jk}),
\quad
0\leq n_j\leq C(2^jr_{\va})^{-1}.
\]
As in the proof of Lemma~\ref{interiorlemma}, for any
$j\in\Z\cap[1,j_0+1]$ we then choose balls
$\{B_{r_j}(x_{jk}')\}_{k=1}^{n_j'}$, where
$r_j:=\f{2^jr_{\va}}{10}$, such that
\[
\cA_j\subset\bigcup_{k=1}^{n_j'}B_{r_j}(x_{jk}'),
\quad
\{x_{jk}'\}_{k=1}^{n_j'}\subset\cA_j,
\quad
0\leq n_j'\leq C(2^jr_\va)^{-1}.
\]
By the definition of $\cA_j$, for each $x_{jk}'$ we have
$r^{\Lda}(u_{\va};x_{jk}')\geq 2^{j-1}r_{\va}$. Hence
\[
\theta_{\va}\(u_{\va};x_{jk}',\f{2^{j+1}r_{\va}}{10}\)\leq C.
\]
Lemma~\ref{va3estimates} gives
\be
\sum_{j=1}^{j_0+1}\int_{\cA_j}f(u_{\va})
\leq
\sum_{j=1}^{j_0+1}\sum_{k=1}^{n_j'}
\int_{B_{r_j}(x_{jk}')}f(u_{\va})
\leq
\sum_{j=1}^{j_0+1}C(2^jr_{\va})^{-1}\va^3
\leq C\va^2.
\label{backslashLdava}
\ee

For $x\in F_0$, condition \eqref{badsetcon} holds by the definition of
$F_0$. Lemma~\ref{badsetcoveruse} yields, for
$\delta=\delta(f,M,\cN)>0$ sufficiently small, a radius
$r_x\in[\va^{\f{1}{4}},\va^{\f{1}{8}}]$ satisfying \eqref{badsetresult}, namely
\be
\int_{B_{r_x}(x)}\f{1}{\va^2} f(u_{\va})
\leq
\f{C}{|\log\va|}
\log \(2+\f{\theta_{\va}(u_{\va};x,\va^{\f{1}{8}})}{|\log\va|}\)
r_x\theta_{\va}(u_{\va};x,r_x).
\label{everyball}
\ee
The family $\{\ol{B}_{r_x}(x)\}_{x\in F_0}$ covers $F_0$. By Besicovitch's
covering theorem \cite[Theorem~1.27]{EG15}, there exist
$\{x_i\}_{i=1}^n\subset F_0$ such that
$F_0\subset\bigcup_{i=1}^n\ol{B}_{r_i}(x_i)$, where $r_i:=r_{x_i}$. The
family $\{\ol{B}_{r_i}(x_i)\}_{i=1}^n$ can be partitioned into
$\ell$ sub-collections $\{\cB_k\}_{k=1}^{\ell}$ of pairwise disjoint closed
balls, where $\ell$ is an absolute constant. By \eqref{everyball} and
Proposition~\ref{MonotoneInterior},
\be
\int_{F_0}\f{1}{\va^2}f(u_{\va})
\leq
\sum_{i=1}^n\int_{B_{r_i}(x_i)}\f{1}{\va^2}f(u_{\va})
\leq
\frac{C}{|\log\va|}
\log\(2+\f{CE_{\va}(u_{\va},B_{40})}{|\log\va|}\)
\sum_{i=1}^nE_{\va}(u_{\va},B_{r_i}(x_i)).
\label{finall2}
\ee
Since the balls in each $\cB_k$ are disjoint,
\[
\sum_{i=1}^n E_{\va}(u_{\va},B_{r_i}(x_i))
\leq
\sum_{k=1}^{\ell}
\sum_{B_{r_i}(x_i)\in\cB_k}E_{\va}(u_{\va},B_{r_i}(x_i))
\leq
\ell E_{\va}(u_{\va},B_{40})
\leq CM(|\log\va|+1).
\]
Substituting this into \eqref{finall2} gives
\[
\int_{F_0}\f{1}{\va^2}f(u_{\va})\leq C(f,M,\cN).
\]
Combining this estimate with \eqref{B1decompose2} and
\eqref{backslashLdava} completes the proof.
\end{proof}

\subsection{Boundary estimates}

The results in this subsection are boundary analogs of the interior estimates
proven above.

\subsubsection{Boundary \texorpdfstring{$W^{1,q}$}{}-estimates}

The main result here is the following boundary $W^{1,q}$-estimate.

\begin{prop}\label{W1pboundaryfinal}
Let $ \om,\cA,g_{\va} $, and $u_{\va}$ be as in
Theorem~\ref{globalminimizersproperties}. Assume that $\om$ is a $C^{2,1}$
domain with parameters $M_{U,2}$ and $r_{U,2}$. Let $x_0\in\pa\om$ and
$r_0\in(0,\f{r_{U,2}}{10})$. For any $q\in(1,2)$,
\be
\|\na u_{\va}\|_{L^q(\om_{\f{r_0}{10}}(x_0))}\leq C,\label{Lqboundary}
\ee
where $C>0$ depends only on $f,M_0,\cN,M_{U,2},r_{U,2},r_0 $, and $q$.
\end{prop}

The proof relies on the following boundary version of Lemma~\ref{interiorlemma}.

\begin{lem}\label{lemW1pfinal}
Assume that $\om,\cA,g_{\va}$, and $u_{\va}$ are as in
Theorem~\ref{globalminimizersproperties}. Assume that $\om$ is a $C^{2,1}$
domain with parameters $M_{U,2}$ and $r_{U,2}$. Let $\ga\in(0,1)$ and
$x_0\in\pa\om$. Let $r_0\in(0,\f{r_{U,2}}{10})$ be such that
$\cA\cap\om_{10r_0}(x_0)$ is either empty or equals $\{x_0\}$. Then there
exist $\delta\in(0,1)$ and $C,\Lda>0$, depending only on
$f,\ga,\cA,\om,M_0,\cN$, and $r_0$, such that the following properties hold.
\begin{enumerate}[label=$(\theenumi)$]
\item\label{lemW1p1} For any $\va\in(0,\delta)$ and
$r\in(\va^{\ga},\f{r_0}{10})$, there exists a collection of balls
$\{B_{r}(x_i)\}_{i=1}^n$ such that
\be
\op{Bad}_{\op{II}}(u_{\va};r,\Lda)\cap\om_{r_0}(x_0)
\subset \bigcup_{i=1}^nB_r(x_i),
\quad
n\leq \Lda r^{-1}.\label{nleqLdarminus1}
\ee

\item\label{lemW1p11} For any $\va\in(0,\delta)$ and
$r\in(\va^{\ga},\f{r_0}{10})$,
\be
\cL^3\(\op{Bad}_{\op{I}}(u_{\va};\delta r)
\cap\om_{\f{r_0}{10}}(x_0)\)
\leq Cr^2.
\label{BadIestimatefinal}
\ee
\end{enumerate}
\end{lem}

\begin{proof}[Proof of Lemma~\ref{lemW1pfinal}\ref{lemW1p1}]
Define
\[
E:=\op{Bad}_{\op{II}}(u_{\va};r,\Lda)\cap\om_{r_0}(x_0)
\]
and decompose $E=E_0\cup E_1$, where
\[
E_0:=\{x\in E:B_{2r}(x)\subset\om\},
\quad
E_1:=E\backslash E_0.
\]

\medskip
\noindent\textit{Step 1: Clearing-out estimates.}
By Proposition~\ref{InteriorClearingout}, there exists
$\eta=\eta(f,M_0,\cN)\in(0,1)$ such that, for any
$x\in E_0\backslash B_{20r}(x_0)$, if $\va\in(0,\eta r)$ and
\[
\theta_{\va}(u_{\va};x,4r)\leq \eta\log\f{r}{\va},
\]
then $\theta_{\va}(u_{\va};x,r)\leq C_0$, where
$C_0=C_0(f,M_0,\cN)>0$.

For any $x\in E_1\backslash B_{20r}(x_0)$, choose
$y_x\in\pa\om$ such that $|x-y_x|<2r$. Then
$B_{2r}(x)\subset B_{4r}(y_x)$. Since
$|x-x_0|\geq 20r$, we also have
$B_{8r}(y_x)\cap B_{2r}(x_0)=\emptyset$. Thus this boundary ball does not
meet the possible singular boundary point $x_0$. Since $r>\va^\ga>\va$ and
$\{g_{\va}\}$ is a family of suitable boundary data, we have
$g_{\va}\in\cN$ on the relevant boundary patch and
\[
r^2\|D_{\top}^jg_{\va}\|_{L^{\ift}(B_{4r}(y_x)\cap\pa\om)}
\leq C(M_0)
\quad\text{for }j\in\{0,1,2\}.
\]
By Proposition~\ref{BoundaryClearingout}, after decreasing $\eta$ if needed,
if $\va\in(0,\eta r)$ and
\[
\theta_{\va}^{\om}(u_{\va};y_x,4r)\leq \eta\log\f{r}{\va},
\]
then $\theta_{\va}^{\om}(u_{\va};x,r)\leq C_1(f,M_0,\cN)$.

For $x\in E_0\backslash B_{20r}(x_0)$, set $y_x:=x$. Choose $ \Lda\geq \max\{C_0,C_1\} $. Since $x\in E=\op{Bad}_{\op{II}}(u_{\va};r,\Lda)$, the preceding clearing-out alternative cannot occur. Hence, for any
$x\in E\backslash B_{20r}(x_0)$,
\[
\theta_{\va}^{\om}(u_{\va};y_x,4r)> \eta\log\f{r}{\va}.
\]
Choose $\eta'=\eta'(\eta,\ga)>0$ such that, when $\va\in(0,\eta')$ and
$r>\va^\ga$, we have $\va<\eta r$ and
$\eta\log\f{r}{\va}\geq \eta'|\log\va|$. Therefore
\be
\theta_{\va}^{\om}(u_{\va};y_x,4r)>\eta'|\log\va|
\quad\text{for any }x\in E\backslash B_{20r}(x_0).
\label{thetavageqboundary}
\ee

\medskip
\noindent\textit{Step 2: Vitali's covering argument.}
First cover $E\cap B_{20r}(x_0)$ by at most $C$ balls of radius $r$. For the
remaining part $E\backslash B_{20r}(x_0)$, apply Vitali's covering lemma to
the family $\{B_{4r}(y_x)\}_{x\in E\backslash B_{20r}(x_0)}$. We obtain a
pairwise disjoint sub-collection $\{B_{4r}(y_i)\}_{i\in I}$ such that
\[
E\backslash B_{20r}(x_0)\subset\bigcup_{i\in I}B_{20r}(y_i).
\]
Using \eqref{thetavageqboundary} and the global energy bound,
\[
(\# I)\eta' r
\leq
\f{1}{|\log\va|}
\sum_{i\in I}E_{\va}(u_{\va},\om_{4r}(y_i))
\leq
\f{C}{|\log\va|}
\int_{\bigcup_{i\in I}\om_{4r}(y_i)}e_{\va}(u_{\va})
\leq C.
\]
Thus $(\# I)r\leq C$. Covering each $B_{20r}(y_i)$ by at most $C$ balls of
radius $r$, and adding the finitely many balls covering $E\cap B_{20r}(x_0)$,
we obtain a collection $\{B_r(x_i')\}_{i=1}^n$ covering $E$, with $ n\leq \Lda r^{-1} $, after increasing $\Lda$ if necessary.
\end{proof}

\begin{proof}[Proof of Lemma~\ref{lemW1pfinal}\ref{lemW1p11}]
Let $\Lda>0$ be given by property \ref{lemW1p1}. If
$r\geq\f{r_0}{40}$, then the estimate follows from the trivial bound
\[
\cL^3\(\op{Bad}_{\op{I}}(u_{\va};\delta r)
\cap\om_{\f{r_0}{10}}(x_0)\)
\leq C r_0^3
\leq C r^2,
\]
after increasing $C=C(r_0)>0$ if necessary. We therefore assume
$r\in(0,\f{r_0}{40})$.

Choose $j_0\in\N$ such that $ 2^{j_0}r\leq \f{r_0}{10}<2^{j_0+1}r $. For $j\in\Z\cap[0,j_0]$, define
\[
F_j:=
\(\op{Bad}_{\op{II}}(u_{\va};2^jr,\Lda)
\cap\om_{\f{r_0}{10}}(x_0)\)
\cup\{x_0\}.
\]
For $j\in\Z\cap[1,j_0+1]$, set
\[
\cA_j^{\om}:=
\(B_{2^{j}r}(F_j)\backslash B_{2^{j-1}r}(F_{j-1})\)
\cap\om_{\f{r_0}{10}}(x_0),
\]
where, for the last layer, we take $F_{j_0+1}:=\om_{\f{r_0}{10}}(x_0)$.
Then
\[
\om_{\f{r_0}{10}}(x_0)
\subset
\(B_r(F_0)\cap\om_{\f{r_0}{10}}(x_0)\)
\cup
\bigcup_{j=1}^{j_0+1}\cA_j^\om.
\]

By property \ref{lemW1p1}, for any $j\in\Z\cap[0,j_0]$ the set $F_j$ is
covered by at most $C(2^jr)^{-1}$ balls of radius $2^jr$, after adding one
extra ball centered at $x_0$. For the last layer $j=j_0+1$, the same bound
follows from the bounded geometry of $\om_{\f{r_0}{10}}(x_0)$, since
$2^{j_0+1}r$ is comparable to $r_0$. Refining these coverings by a fixed
finite number of smaller balls, we obtain, for any $j\in\Z\cap[1,j_0+1]$,
balls $\{B_{\f{2^jr}{16}}(x_{jk}')\}_{k=1}^{n_j'}$ such that
\[
\cA_j^{\om}
\subset
\bigcup_{k=1}^{n_j'}B_{\f{2^jr}{16}}(x_{jk}'),
\quad
\{x_{jk}'\}_{k=1}^{n_j'}\subset\cA_j^{\om},
\quad
0\leq n_j'\leq C(2^jr)^{-1}.
\]
For such a center $x_{jk}'$, the condition
$x_{jk}'\notin B_{2^{j-1}r}(F_{j-1})$ implies that
$x_{jk}'\notin\op{Bad}_{\op{II}}(u_{\va};2^{j-1}r,\Lda)$. Hence
\be
\theta_{\va}^{\om}(u_{\va};x_{jk}',2^{j-1}r)\leq C.
\label{thetava2j1}
\ee

We now estimate the type I bad set inside each ball
$B_{\f{2^jr}{16}}(x_{jk}')$.

\medskip
\noindent\textit{Case a: $B_{2^{j-3}r}(x_{jk}')\subset\om$.}
By \eqref{thetava2j1} and Lemma~\ref{coverpoint}, there exists
$\delta'=\delta'(f,M_0,\cN)\in(0,1)$ such that
\be
\cL^3\(\op{Bad}_{\op{I}}(u_{\va};\delta' r)
\cap B_{\f{2^jr}{16}}(x_{jk}')\)
\leq Cr^3.
\label{interioruselemma}
\ee

\medskip
\noindent\textit{Case b: $B_{2^{j-3}r}(x_{jk}')\cap\pa\om\neq\emptyset$.}
Choose $y_{jk}'\in\pa\om$ such that $ B_{2^{j-3}r}(x_{jk}')\subset B_{2^{j-2}r}(y_{jk}') $. Since $x_0\in F_{j-1}$ and
$x_{jk}'\notin B_{2^{j-1}r}(F_{j-1})$, we have
$|x_{jk}'-x_0|>2^{j-1}r$. It follows that the boundary patch
$T_{2^{j-3}r}^{\om}(y_{jk}')$ stays away from the possible singular boundary
point $x_0$. Therefore the regularity of $g_\va$ gives
\[
(2^{j-2}r)\|\na_{\top}g_{\va}\|_{L^{\ift}(T_{2^{j-3}r}^{\om}(y_{jk}'))}
\leq C.
\]
Using this bound, \eqref{thetava2j1}, and
Proposition~\ref{BoundaryPartialRegularity1prop}, there exists
$\eta>0$ such that
\be
r(u_{\va};x)\geq\f{2^jr}{16C}>\delta' r
\label{geqdeltaprime}
\ee
for any $ x\in
B_{\f{2^jr}{16}}(x_{jk}') $ such that $ \dist(x,\pa\om)\leq 2^j\eta r $ after decreasing $\delta'$ if necessary.

It remains to treat the part away from the boundary. Define
\[
F_{jk}':=
B_{\f{2^jr}{16}}(x_{jk}')
\cap
\{y\in\om:\dist(y,\pa\om)>2^j\eta r\}.
\]
Cover $F_{jk}'$ by balls
$\{B_{\f{2^j\eta r}{100}}(x_{ijk}')\}_{i=1}^{n_{jk}'}$ with
$x_{ijk}'\in F_{jk}'$ and $n_{jk}'\leq C(\eta)$. Since
$\dist(x_{ijk}',\pa\om)>2^j\eta r$, each ball satisfies $ B_{\f{2^{j+1}\eta r}{100}}(x_{ijk}')\subset\om $. Moreover, this ball is contained in $B_{2^{j-1}r}(x_{jk}')$ if $\eta$ is chosen below a fixed dimensional constant. By \eqref{thetava2j1} and monotonicity,
\[
\theta_{\va}\(u_{\va};x_{ijk}',\f{2^{j+1}\eta r}{100}\)\leq C.
\]
Lemma~\ref{coverpoint} gives $\delta''=\delta''(f,M_0,\cN)\in(0,1)$ such that
\[
\cL^3\(\op{Bad}_{\op{I}}(u_{\va};\delta'' r)
\cap B_{\f{2^j\eta r}{100}}(x_{ijk}')\)
\leq Cr^3.
\]
Summing over $i$ and combining this estimate with \eqref{geqdeltaprime}, we
obtain, for $\delta:=\min\{\delta',\delta''\}$,
\[
\cL^3\(\op{Bad}_{\op{I}}(u_{\va};\delta r)
\cap B_{\f{2^jr}{16}}(x_{jk}')\)
\leq Cr^3.
\]

For the initial set $B_r(F_0)$, it follows from \eqref{nleqLdarminus1} that it can be covered by at most $C r^{-1}$ balls of radius comparable to $r$. Hence
\[
\begin{aligned}
\cL^3\(\op{Bad}_{\op{I}}(u_{\va};\delta r)
\cap\om_{\f{r_0}{10}}(x_0)\)
&\leq Cr^2+\sum_{j=1}^{j_0+1}\sum_{k=1}^{n_j'}Cr^3\leq C\sum_{j=0}^{j_0+1}(2^jr)^{-1}r^3
\leq Cr^2.
\end{aligned}
\]
This proves \eqref{BadIestimatefinal}.
\end{proof}

\begin{proof}[Proof of Proposition~\ref{W1pboundaryfinal}]
Let $q\in(1,2)$ be fixed, and choose $\ga\in(\f{q}{2},1)$. By a finite covering argument, using the finiteness of $\cA$, it suffices to prove the estimate on boundary patches for which $\cA\cap\om_{10r_0}(x_0)$ is either empty or equals $\{x_0\}$. We assume this local condition below.

Applying Lemma~\ref{lemW1pfinal}, we obtain that, for any $r\in(\va^{\ga},\f{r_0}{10})$ and $\va\in(0,\delta(f,M_0,\ga,\cN,M_{U,2},r_{U,2},r_0))$,
\[
\cL^3\(\op{Bad}_{\op{I}}(u_{\va};\delta r)\cap\om_{\f{r_0}{10}}(x_0)\)\leq Cr^2.
\]
Set $R_0:=(\f{r_0}{10})^{\ga^{-1}}$. For any $R\in(\va,R_0)$, we have $R^\ga\in(\va^\ga,\f{r_0}{10})$ and $R<R^\ga$. An argument analogous to \eqref{Rlarge} gives
\be
\cL^3(\{y\in\om_{\f{r_0}{10}}(x_0):r(u_{\va};y)<\delta R\})\leq CR^{2\ga}.
\label{Rgaes}
\ee
For $R\in[R_0,1)$, the same estimate follows from the trivial volume bound, after increasing $C=C(r_0)$ if necessary. If $R\in(0,\va)$, then for any $y\in\om_{\f{r_0}{10}}(x_0)$, Lemma~\ref{AprioriBoundary} gives $R|\na u_{\va}(y)|\leq C$, so choosing $\delta$ smaller if necessary,
\[
\cL^3(\op{Bad}_{\op{I}}(u_{\va};\delta R)\cap\om_{\f{r_0}{10}}(x_0))=0.
\]
Together with \eqref{Rgaes}, this implies that $\na u_{\va}\in L^{2\ga,+\ift}(\om_{\f{r_0}{10}}(x_0))$. Since $q<2\ga$, the embedding $L^{2\ga,+\ift}\hookrightarrow L^q$ on bounded sets gives \eqref{Lqboundary}.
\end{proof}

\subsubsection{Boundary estimates on the potential}

The following proposition is the boundary analog of Proposition~\ref{Interiorpotential}. The proof combines the covering structure from Lemma~\ref{lemW1pfinal} with a logarithmic-scale potential estimate near points that lie close to the type II bad set but away from the singular set $\cA$.

\begin{prop}\label{potentialfinalestimate}
Let $\om,\cA,g_{\va}$, and $u_{\va}$ be as in Theorem~\ref{globalminimizersproperties}. Assume that $\om$ is a $C^{2,1}$ domain with parameters $M_{U,2}$ and $r_{U,2}$. Let $x_0\in\pa\om$ and $r_0\in(0,\f{r_{U,2}}{10})$. Then
\be
\f{1}{\va^2}\int_{\om_{\f{r_0}{10}}(x_0)}f(u_{\va})\leq C,
\label{potentialfinalestimateeq}
\ee
where $C>0$ depends only on $\cA,f,M_0,M_{U,2},r_{U,2}$, and $\cN$.
\end{prop}

The key ingredient is the following lemma, which provides a pointwise logarithmic-scale estimate on the potential near a point close to the type II bad set. It is the boundary analog of Lemma~\ref{badsetcoveruse}.

\begin{lem}\label{boundarypotentialrx}
Let $\om,\cA,g_{\va}$, and $u_{\va}$ be as in Theorem~\ref{globalminimizersproperties}. Assume that $\om$ is a $C^{2,1}$ domain with parameters $M_{U,2}$ and $r_{U,2}$. Set $r_{\va}:=\f{\va^{\f{1}{4}}}{2}$ and $\sg_{\va}:=\va^{\f{1}{8}}$. Assume $x\in\om$ satisfies $\dist(x,\cA)>\va^{\f{1}{16}}$. There exist $\delta\in(0,1)$ and $\Lda>0$, depending only on $f,M_0,\cA,\om$, and $\cN$, such that the following property holds. If $\va\in(0,\delta)$ and
\be
\dist\(x,\op{Bad}_{\op{II}}\(u_{\va};r_{\va},\Lda\)\)<r_{\va},
\label{Badboundaryf}
\ee
then there exists $r_x\in[2r_{\va},\sg_{\va}]$ such that
\[
\int_{\om_{r_x}(x)}\f{1}{\va^2}f(u_{\va})\leq\f{C}{|\log\va|}r_x\theta_{\va}^{\om}(u_{\va};x,r_x),
\]
where $C>0$ depends only on $\cA,f,M_0,M_{U,2},r_{U,2}$, and $\cN$.
\end{lem}

\begin{proof}
For small $\va\in(0,1)$, we have $\va^{\f{1}{16}}>2\sg_{\va}$. Since $\dist(x,\cA)>\va^{\f{1}{16}}$, the boundary patch $T_{2\sg_{\va}}^{\om}(x)$ stays away from $\cA$. Hence $g_{\va}\in\cN$ on this patch and
\[
\sg_{\va}\|\na_{\top}g_{\va}\|_{L^{\ift}(T_{2\sg_{\va}}^{\om}(x))}\leq C.
\]
By Proposition~\ref{locallogestimate}, we have
\be
\theta_{\va}^{\om}(u_{\va};x,2\sg_{\va})\leq C|\log\va|.
\label{upperboundF}
\ee
Define
\[
F_{\va}(r):=\theta_{\va}^{\om}(u_{\va};x,r)+C|\log\va|r,
\quad
G_{\va}(r):=\f{2}{\va^2r}\int_{\om_r(x)}f(u_{\va}).
\]
By Proposition~\ref{MonotoneBoundary}, $F_{\va}'(r)\geq G_{\va}(r)$ for $r\in[2r_{\va},\sg_{\va}]$.

From condition \eqref{Badboundaryf} and Proposition~\ref{BoundaryClearingout}, after choosing $\Lda$ large enough and $\delta$ small enough, we have
\be
\theta_{\va}^{\om}(u_{\va};x,2r_{\va})\geq C^{-1}|\log\va|.
\label{lowerboundF}
\ee
Indeed, otherwise the clearing-out estimate applied at a point of $\op{Bad}_{\op{II}}(u_{\va};r_{\va},\Lda)$ within distance $r_{\va}$ from $x$ would contradict the definition of the type II bad set.

Set
\[
\wt{F}_{\va}(s):=F_{\va}(\exp(s)),\quad
\wt{G}_{\va}(s):=G_{\va}(\exp(s)),\quad
I_{\va}:=\left[\f{\log\va}{4},\f{\log\va}{8}\right].
\]
Then $\exp(s)\in[2r_{\va},\sg_{\va}]$ for $s\in I_{\va}$, and $\wt{F}_{\va}'(s)\geq\wt{G}_{\va}(s)$. Arguing as in the proof of Lemma~\ref{badsetcoveruse}, there exists $s_{\va}\in I_{\va}$ such that
\[
\wt{G}_{\va}(s_{\va})\leq\lda_{\va}\wt{F}_{\va}(s_{\va}),
\quad
\lda_{\va}:=\f{8}{|\log\va|}\log\f{F_{\va}(\sg_{\va})}{F_{\va}(2r_{\va})}\geq0.
\]
Set $r_x:=\exp(s_{\va})\in[2r_{\va},\sg_{\va}]$. Scaling back, we deduce that
\be
\f{1}{\va^2}\int_{\om_{r_x}(x)}f(u_{\va})
\leq\f{\lda_{\va}r_x}{2}F_{\va}(r_x)
\leq\f{C r_xF_{\va}(r_x)}{|\log\va|}
\log\f{F_{\va}(\sg_{\va})}{F_{\va}(2r_{\va})}.
\label{fQlogvaboundary}
\ee
We bound each factor. By \eqref{lowerboundF}, $F_{\va}(2r_{\va})\geq C^{-1}|\log\va|$. From \eqref{upperboundF} and the monotonicity of $F_{\va}$,
\[
F_{\va}(\sg_{\va})\leq F_{\va}(2\sg_{\va})\leq C|\log\va|.
\]
Hence
\[
\log\f{F_{\va}(\sg_{\va})}{F_{\va}(2r_{\va})}\leq C.
\]
Using Proposition~\ref{MonotoneBoundary}, we have
\begin{align*}
\theta_{\va}^{\om}(u_{\va};x,r_x)
&\geq \theta_{\va}^{\om}(u_{\va};x,2r_{\va})-C|\log\va|r_x\\
&\geq \theta_{\va}^{\om}(u_{\va};x,2r_{\va})-C\va^{\f{1}{8}}|\log\va|
\geq C^{-1}|\log\va|.
\end{align*}
Since $r_x\leq\sg_{\va}$, it follows that
\[
F_{\va}(r_x)
=\theta_{\va}^{\om}(u_{\va};x,r_x)+C|\log\va|r_x
\leq C\theta_{\va}^{\om}(u_{\va};x,r_x).
\]
Substituting this into \eqref{fQlogvaboundary} completes the proof.
\end{proof}

\begin{proof}[Proof of Proposition~\ref{potentialfinalestimate}]
By a finite covering argument, using the finiteness of $\cA$, it suffices to prove the estimate on boundary patches for which $\cA\cap\om_{10r_0}(x_0)$ is either empty or equals $\{x_0\}$. We assume this local condition below.

Set $r_{\va}:=\f{\va^{\f{1}{4}}}{2}$ and $\sg_{\va}:=\va^{\f{1}{8}}$. Choose $j_0\in\N$ such that $2^{j_0}r_{\va}<\f{r_0}{10}\leq2^{j_0+1}r_{\va}$. For $j\in\Z\cap[0,j_0+1]$, define
\[
F_j^{\om}:=(\op{Bad}_{\op{II}}(u_{\va};2^jr_{\va},\Lda)\cap\om_{\f{r_0}{10}}(x_0))\cup\{x_0\},
\]
and, for $j\in\Z\cap[1,j_0+1]$,
\[
\cA_j^{\om}:=(B_{2^jr_{\va}}(F_j^{\om})\backslash B_{2^{j-1}r_{\va}}(F_{j-1}^{\om}))\cap\om_{\f{r_0}{10}}(x_0).
\]
We have the decomposition
\[
\om_{\f{r_0}{10}}(x_0)\subset B_{r_{\va}}(F_0^{\om})\cup\bigcup_{j=1}^{j_0+1}\cA_j^{\om}.
\]

\medskip
\noindent\textbf{Step 1: Estimate on $\cup_j\cA_j^{\om}$.}
For $j\in\Z\cap[0,j_0]$, Lemma~\ref{lemW1pfinal} gives a cover of $F_j^\om$ by at most $C(2^jr_{\va})^{-1}$ balls of radius $2^jr_{\va}$, after adding the ball centered at $x_0$. Enlarging and subdividing these balls by a fixed finite factor gives the same type of cover for $B_{2^jr_{\va}}(F_j^\om)\cap\om_{\f{r_0}{10}}(x_0)$. For the last layer $j=j_0+1$, the same bound follows from the bounded geometry of $\om_{\f{r_0}{10}}(x_0)$, since $2^{j_0+1}r_{\va}$ is comparable to $r_0$. Thus, for each $j\in\Z\cap[1,j_0+1]$, the same secondary covering argument as in Lemma~\ref{lemW1pfinal} gives balls $\{B_{\f{2^jr_{\va}}{16}}(x_{jk}')\}_{k=1}^{n_j'}$ such that
\[
\cA_j^{\om}\subset\bigcup_{k=1}^{n_j'}B_{\f{2^jr_{\va}}{16}}(x_{jk}'),
\quad
\{x_{jk}'\}_{k=1}^{n_j'}\subset\cA_j^\om,
\quad
0\leq n_j'\leq C(2^jr_{\va})^{-1}.
\]
Since $x_{jk}'\notin B_{2^{j-1}r_{\va}}(F_{j-1}^{\om})$, we have
\be
\theta_{\va}^{\om}(u_{\va};x_{jk}',2^{j-1}r_{\va})\leq C.
\label{energyboundxjk}
\ee
We estimate the integral of $\f{1}{\va^2}f(u_{\va})$ in $B_{\f{2^jr_{\va}}{16}}(x_{jk}')\cap\om$. We divide it into two cases.

\textit{Case a: $B_{2^{j-3}r_{\va}}(x_{jk}')\subset\om$.}
Applying Lemma~\ref{va3estimates} with radius $2^{j-4}r_{\va}$ and using the energy bound \eqref{energyboundxjk},
\[
\f{1}{\va^2}\int_{B_{\f{2^jr_{\va}}{16}}(x_{jk}')}f(u_{\va})\leq C\va.
\]

\textit{Case b: $B_{2^{j-3}r_{\va}}(x_{jk}')\cap\pa\om\neq\emptyset$.}
Select $y_{jk}'\in\pa\om$ with $B_{2^{j-3}r_{\va}}(x_{jk}')\subset B_{2^{j-2}r_{\va}}(y_{jk}')$. Since $x_0\in F_{j-1}^{\om}$ and $x_{jk}'\notin B_{2^{j-1}r_{\va}}(F_{j-1}^{\om})$, we have $|x_{jk}'-x_0|>2^{j-1}r_{\va}$. Hence the relevant boundary patch stays a distance comparable to $2^jr_{\va}$ from $x_0$. By the local condition on $\cA$ and the regularity of $g_{\va}$ away from $\cA$, we obtain
\[
(2^{j-2}r_{\va})\|\na_{\top}g_{\va}\|_{L^{\ift}(T_{2^{j-2}r_{\va}}^{\om}(y_{jk}'))}\leq C.
\]
By Proposition~\ref{BoundaryPartialRegularity1prop}, there exists $\eta=\eta(f,M_0,\cN)>0$ such that
$r(u_{\va};y)\geq2^j\eta r_{\va}$ for any
\[
y\in E_{jk}':=
B_{\f{2^jr_{\va}}{16}}(x_{jk}')
\cap
\{y\in\om:\dist(y,\pa\om)\leq2^j\eta r_{\va}\}.
\]
Moreover,
\[
f(u_{\va}(y))\leq\f{C\va^4}{(2^jr_{\va})^4}
\quad\text{for any }y\in E_{jk}'.
\]
Since $r_{\va}=\f{\va^{\f{1}{4}}}{2}$, we get $(2^jr_{\va})^4=\f{2^{4j}\va}{16}$ and
\[
\f{1}{\va^2}\int_{E_{jk}'}f(u_{\va})
\leq C\cL^3(E_{jk}')\cdot\f{\va^2}{(2^jr_{\va})^4}
\leq C\va.
\]
For the complementary region $B_{\f{2^jr_{\va}}{16}}(x_{jk}')\backslash E_{jk}'$, any point $z$ satisfies $\dist(z,\pa\om)>2^j\eta r_{\va}$. We cover this region by a finite number, bounded by $C(\eta)$, of balls of radius $\f{2^j\eta r_{\va}}{2}$ lying in $\om$. After decreasing $\eta$ by a dimensional factor if necessary, these balls are contained in $B_{2^{j-1}r_{\va}}(x_{jk}')$. Hence Lemma~\ref{va3estimates} applies on each such ball, using \eqref{energyboundxjk} and Proposition~\ref{MonotoneInterior}. For any ball $B$ in this re-covering, we obtain
\[
\f{1}{\va^2}\int_B f(u_{\va})\leq C\va.
\]
Therefore
\[
\f{1}{\va^2}\int_{B_{\f{2^jr_{\va}}{16}}(x_{jk}')\backslash E_{jk}'}f(u_{\va})\leq C\va.
\]

In both cases, we have
\[
\f{1}{\va^2}\int_{B_{\f{2^jr_{\va}}{16}}(x_{jk}')\cap\om}f(u_{\va})\leq C\va.
\]
Summing over $j\in\Z\cap[1,j_0+1]$,
\be
\f{1}{\va^2}\sum_{j=1}^{j_0+1}\int_{\cA_j^{\om}}f(u_{\va})
\leq
\sum_{j=1}^{j_0+1}Cn_j'\va
\leq
C\va\(\sum_{j=1}^{\ift}(2^jr_{\va})^{-1}\)
\leq
\f{C\va}{r_{\va}}
\leq C.
\label{Ajompotential}
\ee

\medskip
\noindent\textbf{Step 2: Estimate on $B_{r_{\va}}(F_0^{\om})$.}
By Proposition~\ref{locallogestimate},
\be
\f{1}{\va^2}\int_{\om_{\va^{\f{1}{16}}}(x_0)}f(u_{\va})
\leq C\va^{\f{1}{16}}(|\log\va|+1)
\leq C.
\label{Bva16bound}
\ee
Set
\[
E_\va^0:=
\(B_{r_{\va}}(F_0^{\om})\cap\om_{\f{r_0}{10}}(x_0)\)
\backslash \om_{\va^{\f{1}{16}}}(x_0).
\]
For any $x\in E_\va^0$, there exists $z\in F_0^\om$ such that $|x-z|<r_{\va}$. Since $r_{\va}<\va^{\f{1}{16}}$ for small $\va$ and $x\notin\om_{\va^{\f{1}{16}}}(x_0)$, this point $z$ cannot be $x_0$. Hence $ z\in\op{Bad}_{\op{II}}(u_{\va};r_{\va},\Lda)\cap\om_{\f{r_0}{10}}(x_0) $, and condition~\eqref{Badboundaryf} holds. Moreover, by the local separation of $\cA$ and after decreasing $\delta$ if necessary, we have $\dist(x,\cA)>\va^{\f{1}{16}}$ for all such $x$. 

By Lemma~\ref{boundarypotentialrx}, there exists $r_x\in[2r_{\va},\sg_{\va}]$ with
\[
\int_{\om_{r_x}(x)}\f{1}{\va^2}f(u_{\va})
\leq
\f{C}{|\log\va|}r_x\theta_{\va}^{\om}(u_{\va};x,r_x).
\]
The family $\{\ol{B}_{r_x}(x)\}_{x\in E_\va^0}$ covers $E_\va^0$. By Besicovitch's covering theorem, we extract points $\{x_i\}$ from $E_\va^0$, with $r_i:=r_{x_i}$, such that $\{\ol{B}_{r_i}(x_i)\}$ can be partitioned into $\ell$ collections $\{\cB_k\}_{k=1}^{\ell}$ of pairwise disjoint closed balls, where $\ell$ is an absolute constant. Then
\[
\int_{E_\va^0}\f{1}{\va^2}f(u_{\va})
\leq
\sum_i\int_{\om_{r_i}(x_i)}\f{1}{\va^2}f(u_{\va})
\leq
\f{C}{|\log\va|}\sum_i E_{\va}^{\om}(u_{\va},\om_{r_i}(x_i)).
\]
By disjointness within each $\cB_k$,
\[
\sum_i E_{\va}^{\om}(u_{\va},\om_{r_i}(x_i))
\leq
\sum_{k=1}^{\ell}
\sum_{\ol{B}_{r_i}(x_i)\in\cB_k}
E_{\va}^{\om}(u_{\va},\om_{r_i}(x_i))
\leq
\ell E_{\va}(u_{\va},\om)
\leq
CM_0(|\log\va|+1).
\]
Hence
\[
\int_{E_\va^0}\f{1}{\va^2}f(u_{\va})\leq C.
\]
Combining this estimate with \eqref{Ajompotential} and \eqref{Bva16bound} gives \eqref{potentialfinalestimateeq}, completing the proof.
\end{proof}

\subsection{Proof of Theorem \ref{W1pestimate}} 

The estimate \eqref{W1pestimateeq1} follows by applying Proposition \ref{InteriorW1p} and \ref{W1pboundaryfinal}. The estimate \eqref{W1pestimateeq11} follows analogously from \ref{Interiorpotential} and \ref{potentialfinalestimate}.

\section{Minimality: Proof of Theorem~\ref{minimalitytheorem}}\label{SectionMinimality}

\begin{proof}[Proof of Theorem~\ref{minimalitytheorem}]
The fact that $\Gamma_*\in\cC(\om,g)$ follows from Theorem~\ref{structureofthelimitset} and the meridian characterization of the limiting singular set obtained there. It remains to prove the minimality inequality. Let
\[
\Gamma=\{(K_i,\sg_i):i\in I\}\in\cC(\om,g)
\]
be an admissible competitor, and set $S:=S_\Gamma=\bigcup_{i\in I}K_i$. By subdividing the family at the finitely many branching points, if necessary, and assigning the same label to the resulting sub-segments, we may assume that the endpoints of the segments are precisely the vertices of the network. This operation does not change $S_\Gamma$, admissibility, or $\MM(\Gamma)$.

Let $v\in H_{\loc}^1(\ol{\om}\backslash S,\cN)$ be the map given by property~\ref{property4} in Definition~\ref{defadmissible}. We construct a sequence $w_n\in H^1(\om,\R^m)$ such that $w_n|_{\pa\om}=g_{\va_n}$ and
\be
\limsup_{n\to+\ift}\f{1}{|\log\va_n|}\int_{\om}e_{\va_n}(w_n)
\leq \MM(\Gamma)
=\sum_{i\in I}|\sg_i|_*\HH^1(K_i).
\label{limitwi}
\ee
Given such $w_n$, the minimality of $u_{\va_n}$ and the shared boundary condition $u_{\va_n}=w_n$ on $\pa\om$ yield
\[
\sum_{j\in J}|\al_j|_*\HH^1(L_j)
\leq
\liminf_{n\to+\ift}\f{1}{|\log\va_n|}\int_{\om}e_{\va_n}(u_{\va_n})
\leq
\limsup_{n\to+\ift}\f{1}{|\log\va_n|}\int_{\om}e_{\va_n}(w_n)
\leq
\MM(\Gamma),
\]
which gives the desired inequality.

\medskip
\noindent\textbf{Step 1: Preliminary setup.}
Let $P$ be the finite set of branching points of $S$, namely the points of $\ol\om$ that are endpoints of at least two segments in the family. Let $Q$ be the finite set of boundary endpoints of $S$ that are incident to exactly one segment. Choose $\eta\in(0,1)$ so small that the balls $B_{2\eta}(x)$, $x\in P\cup Q$, are pairwise disjoint and meet only the segments incident to their centers. Then choose $\delta\in(0,\eta/10)$.

Since $v\in H_{\loc}^1(\ol{\om}\backslash S,\cN)$, the coarea formula applied away from $S$ provides $\delta'\in(\f{3\delta}{4},\f{5\delta}{4})$ and $\eta'\in(\f{3\eta}{4},\f{5\eta}{4})$ such that, with
\[
S_{\delta'}:=\{y\in\om:\dist(y,S)\leq\delta'\},
\]
the following properties hold:
\begin{itemize}
\item $v|_{\pa S_{\delta'}\cap\om}\in H^1(\pa S_{\delta'}\cap\om,\cN)$;
\item for any $x\in P$, $ v|_{(\pa B_{\eta'}(x)\backslash S_{\delta'})\cap\om}
\in H^1((\pa B_{\eta'}(x)\backslash S_{\delta'})\cap\om,\cN) $. 
\end{itemize}
To simplify the notation, we write $\delta'=\delta$ and $\eta'=\eta$. Throughout this proof, $C$ denotes a positive constant independent of $n,\eta,\delta$, though it may depend on the fixed competitor $\Gamma$.

\medskip
\noindent\textbf{Step 2: Construction on a cylinder.}
Fix $i\in I$. Up to rotation and translation, write
\[
K_i=\{(0,0)\}\times[-h_i,h_i],
\quad
h_i:=\f{1}{2}\HH^1(K_i)>0.
\]
Set $\Lda_i:=B_\delta^2\times(z_{i,-},z_{i,+})$, where
\[
z_{i,-}:=
\begin{cases}
-h_i+\sqrt{\eta^2-\delta^2} & \text{if }(0,0,-h_i)\in P,\\
-h_i+2\delta & \text{if }(0,0,-h_i)\in Q,
\end{cases}
\]
and define $z_{i,+}$ symmetrically. Let $\Ga_i$ be the lateral surface of $\Lda_i$. By the meridian condition in Definition~\ref{defadmissible}, the homotopy class of $v|_{\pa D}$ is $\sg_i$ for any admissible meridian disk $D$ centered at an interior point of $K_i$ and orthogonal to $K_i$.

By Lemma~\ref{cylinderex}, there exists $w_n|_{\Lda_i}\in H^1(\Lda_i,\R^m)$ such that
\be
\int_{\Lda_i} e_{\va_n}(w_n)\leq\HH^1(K_i)|\sg_i|_*\log\f{\delta}{\va_n}
+C\left(\f{(\HH^1(K_i))^2}{\delta}+\delta\right)
\int_{\Ga_i}|\na_{\Ga_i}v|^2\ud\HH^2
+C\HH^1(K_i).
\label{wionLda}
\ee
Here we have used $\Ga_i\subset\pa S_\delta\cap\om$ and $v|_{\pa S_\delta\cap\om}\in H^1(\pa S_\delta\cap\om,\cN)$. The same construction gives, for any flat base disk $D$ of $\Lda_i$,
\be
\int_D e_{\va_n}(w_n)\ud\HH^2
\leq
C\left(\f{\HH^1(K_i)}{\delta}+\f{\delta}{\HH^1(K_i)}\right)
\int_{\Ga_i}|\na_{\Ga_i}v|^2\ud\HH^2
+|\sg_i|_*\log\f{\delta}{\va_n}
+C.
\label{wionbaseD}
\ee

\medskip
\noindent\textbf{Step 3: Construction near a boundary endpoint.}
Let $x\in Q$, and let $K_i$ be the unique segment of $S$ ending at $x$. Up to rotation and translation, write
\[
K_i=\{(0,0)\}\times[-h_i,h_i],
\quad
x=(0,0,h_i).
\]
Define the cap region
\[
\Lda_{\delta,x}:=\{y=(y_1,y_2,t)\in\om:\dist(y,K_i)<\delta,\ t>h_i-2\delta\}.
\]
Since $\om$ is $C^{2,1}$ and $\ol{\om}$ is strongly convex at every point $y\in\pa\om$, the closure $\ol{\Lda}_{\delta,x}$ is bi-Lipschitz homeomorphic to $\ol{B}_\delta$ via some map $ \Psi_x:\ol{\Lda}_{\delta,x}\to\ol{B}_\delta $, with a bi-Lipschitz constant depending only on the geometry of $\pa\om$.

The map $w_n$ has already been defined on the cylindrical part adjacent to $\Lda_{\delta,x}$. By \eqref{wionbaseD}, applied to the disk $B_\delta^2\times\{h_i-2\delta\}$, we have
\be
\begin{aligned}
\int_{B_\delta^2\times\{h_i-2\delta\}}e_{\va_n}(w_n)\ud\HH^2
&\leq
C\left(\f{\HH^1(K_i)}{\delta}+\f{\delta}{\HH^1(K_i)}\right)
\int_{\Ga_i}|\na_{\Ga_i}v|^2\ud\HH^2
+|\sg_i|_*\log\f{\delta}{\va_n}
+C.
\end{aligned}
\label{wionLdadelta}
\ee
Now define $\wt{g}_n:\pa\Lda_{\delta,x}\to\R^m$ by
\[
\wt{g}_n(y):=
\begin{cases}
g_{\va_n}(y) & \text{if }y\in\pa\Lda_{\delta,x}\cap\pa\om,\\
v|_{\pa S_\delta\cap\om}(y) & \text{if }y\in\pa S_\delta\cap\om,\\
w_n|_{B_\delta^2\times\{h_i-2\delta\}}(y) & \text{if }y\in B_\delta^2\times\{h_i-2\delta\}.
\end{cases}
\]
By the trace condition for $v$ and the construction of the boundary data $g_{\va_n}$ on regular boundary pieces, the first two traces agree on their common edge. The second and third traces agree on $\pa B_\delta^2\times\{h_i-2\delta\}$ by the cylinder construction. Hence $\wt{g}_n\in H^1(\pa\Lda_{\delta,x},\R^m)$. From \eqref{wionLdadelta} and $\delta\in(0,1)$,
\be
\int_{\pa\Lda_{\delta,x}}e_{\va_n}(\wt{g}_n)\ud\HH^2
\leq
C\log\f{\delta}{\va_n}
+C(\delta^{-1}+1)\int_{\Ga_i}|\na_{\Ga_i}v|^2\ud\HH^2
+C.
\label{gvaipaLdadelta}
\ee

We extend $w_n$ to $\Lda_{\delta,x}$ by the homogeneous extension
\[
w_n|_{\Lda_{\delta,x}}(y):=
\wt{g}_n\left(\Psi_x^{-1}\left(\f{\delta\Psi_x(y)}{|\Psi_x(y)|}\right)\right),
\quad
y\in\Lda_{\delta,x}\backslash\{\Psi_x^{-1}(0)\}.
\]
The standard energy estimate for radial extensions, together with \eqref{gvaipaLdadelta}, gives
\be
\int_{\Lda_{\delta,x}}e_{\va_n}(w_n)
\leq
C\delta\int_{\pa\Lda_{\delta,x}}e_{\va_n}(\wt{g}_n)\ud\HH^2
\leq
C\delta\log\f{\delta}{\va_n}
+C\int_{\Ga_i}|\na_{\Ga_i}v|^2\ud\HH^2
+C\delta.
\label{wiLdadeltaestimate}
\ee

\medskip
\noindent\textbf{Step 4: Construction near a branching point.}
Let $x\in P$, and let $K_{i_1},\ldots,K_{i_\ell}$ be the segments of $S$ incident to $x$. By construction,
\[
v|_{(\pa B_\eta(x)\backslash S_\delta)\cap\om}
\in
H^1((\pa B_\eta(x)\backslash S_\delta)\cap\om,\cN).
\]
For each $j\in\Z\cap[1,\ell]$, let $D_j$ be the open $2$-disk orthogonal to $K_{i_j}$, with radius $\delta$, and with center on $K_{i_j}$ at distance $\sqrt{\eta^2-\delta^2}$ from $x$. We fill the gaps in $(\pa B_\eta(x)\backslash S_\delta)\cap\om$ with the disks $\ol D_1,\ldots,\ol D_\ell$ and set
\[
\M_x:=(\pa\om\cap\ol{B}_\eta(x))
\cup((\pa B_\eta(x)\backslash S_\delta)\cap\om)
\cup\bigcup_{j=1}^{\ell}\ol{D}_j.
\]
Then $\M_x$ is a compact $2$-manifold without boundary, bi-Lipschitz homeomorphic to $\pa B_\eta$. If $x\in\om$, we have $\pa\om\cap\ol B_\eta(x)=\emptyset$ by the choice of $\eta$. By construction, $w_n|_{D_j}\in H^1(D_j,\R^m)$ for any $j\in\Z\cap[1,\ell]$.

Define $v_n:\M_x\to\R^m$ by
\[
v_n(y):=
\begin{cases}
g_{\va_n}(y) & \text{if }y\in\pa\om\cap\ol{B}_\eta(x),\\
v|_{(\pa B_\eta(x)\backslash S_\delta)\cap\om}(y) & \text{if }y\in(\pa B_\eta(x)\backslash S_\delta)\cap\om,\\
w_n|_{\ol{D}_j}(y) & \text{if }y\in\ol{D}_j,\ j\in\Z\cap[1,\ell].
\end{cases}
\]
The traces agree on the common edges by the trace condition for $v$, the construction of $g_{\va_n}$ on regular boundary pieces, and the cylinder construction. Hence $v_n\in H^1(\M_x,\R^m)$.

Let $E_x$ be the open bounded set with $\pa E_x=\M_x$, and let $ \Phi_x:\ol{E}_x\to\ol{B}_\eta $ be a bi-Lipschitz homeomorphism. We define
\[
w_n|_{E_x}(y):=
v_n\left(\Phi_x^{-1}\left(\f{\eta\Phi_x(y)}{|\Phi_x(y)|}\right)\right),
\quad
y\in E_x\backslash\{\Phi_x^{-1}(0)\}.
\]
Using \eqref{wionbaseD} to bound $\int_{D_j}e_{\va_n}(w_n)\ud\HH^2$, the standard energy estimate for homogeneous extensions gives
\be
\begin{aligned}
\int_{E_x}e_{\va_n}(w_n)
&\leq C\eta\int_{\M_x} e_{\va_n}(v_n)\ud\HH^2\\
&\leq
C\eta\int_{\pa\om\cap B_\eta(x)}e_{\va_n}(g_{\va_n})\ud\HH^2
+C\eta\int_{(\pa B_\eta(x)\backslash S_\delta)\cap\om}e_{\va_n}(v)\ud\HH^2\\
&\quad
+C\eta\sum_{j=1}^{\ell}\int_{D_j}e_{\va_n}(w_n)\ud\HH^2\\
&\leq
C\eta\log\f{\delta}{\va_n}
+C\eta\int_{(\pa B_\eta(x)\backslash S_\delta)\cap\om}e_{\va_n}(v)\ud\HH^2
+\f{C\eta}{\delta}\int_{\pa S_\delta\cap\om}e_{\va_n}(v)\ud\HH^2.
\end{aligned}
\label{wiEestimates}
\ee

\medskip
\noindent\textbf{Step 5: Global definition.}
Define
\[
E_{\eta,\delta}:=S_\delta\cup\bigcup_{x\in P}\ol{\om}_\eta(x).
\]
For $y\in\om\backslash E_{\eta,\delta}$, set $w_n(y):=v(y)$. On the cylindrical pieces, on the cap regions near points of $Q$, and on the regions $E_x$ near points of $P$, define $w_n$ as above. Since all traces agree at every interface, $w_n\in H^1(\om,\R^m)$ and $w_n|_{\pa\om}=g_{\va_n}$.

\medskip
\noindent\textbf{Step 6: Conclusion.}
Summing \eqref{wionLda}, \eqref{wiLdadeltaestimate}, and \eqref{wiEestimates} over all segments, all boundary endpoints in $Q$, and all branching points in $P$, we obtain, for all sufficiently large $n$,
\[
\begin{aligned}
\int_\om e_{\va_n}(w_n)
&\leq
\MM(\Gamma)\log\f{\delta}{\va_n}
+C\eta\log\f{\delta}{\va_n}
+C\eta\sum_{x\in P}\int_{(\pa B_\eta(x)\backslash S_\delta)\cap\om}e_{\va_n}(v)\ud\HH^2\\
&\quad
+\f{C}{\delta}\int_{\pa S_\delta\cap\om}e_{\va_n}(v)\ud\HH^2
+\int_{\om\backslash E_{\eta,\delta}}e_{\va_n}(v)\ud x
+C.
\end{aligned}
\]
Dividing both sides by $|\log\va_n|$ and letting $n\to+\ift$, we note that $v\in\cN$ on $\om\backslash S$, so $f(v)=0$ there and $e_{\va_n}(v)=\f{1}{2}|\na v|^2$ on $\om\backslash S$. Since
\[
v\in H^1(\pa S_\delta\cap\om,\cN),\quad
v\in H^1((\pa B_\eta(x)\backslash S_\delta)\cap\om,\cN)\quad\text{for any }x\in P,
\]
and $v\in H^1(\om\backslash E_{\eta,\delta},\cN)$, all integrals involving $v$ are bounded independently of $n$. Their contributions therefore vanish after division by $|\log\va_n|$. Moreover,
\[
\f{\log(\delta\va_n^{-1})}{|\log\va_n|}\to 1
\quad\text{as }n\to+\ift.
\]
We conclude that
\[
\limsup_{n\to+\ift}\f{1}{|\log\va_n|}\int_\om e_{\va_n}(w_n)
\leq
\MM(\Gamma)+C\eta.
\]
Since $\eta>0$ is arbitrary and $C$ is independent of $\eta$, this proves \eqref{limitwi}. As $\Gamma\in\cC(\om,g)$ was arbitrary, the minimality inequality follows.
\end{proof}

\section*{Acknowledgment}

Haotong Fu and Wei Wang are partially supported by the National Key R$\&$D Program of China under Grant 2023YFA1008801 and NSF of China under Grant 12288101. Giacomo Canevari is partially supported by INdAM--GNAMPA, project ref.~E53C25002010001.

\bibliographystyle{plain}

\end{document}